\documentclass[a4paper,reqno]{amsart}
\usepackage[utf8]{inputenc}
\usepackage[T1]{fontenc}
\usepackage{a4wide}
\usepackage{amsmath,amsthm,amsfonts,amssymb}
\usepackage{mathrsfs}
\usepackage{array}
\usepackage{tikz}
\usetikzlibrary{calc,decorations.pathreplacing}
\usepackage{url,enumitem}
\usepackage{multirow}
\usepackage{microtype}
\usepackage[bookmarks=false,pdfborder={0 0 0.05}]{hyperref}
\hypersetup{colorlinks=true, citecolor=darkblue, linkcolor=darkblue,
  urlcolor=darkblue}

\usepackage{graphicx}
\graphicspath{{figures/}}

\usepackage[labelfont=up]{caption}
\usepackage{subcaption} 
\usetikzlibrary{arrows,positioning,decorations.markings,calc,shapes,intersections,spy}

\usepackage{mdframed}

\theoremstyle{plain}
\newtheorem{theorem}{Theorem}[section]
\newtheorem{lemma}[theorem]{Lemma}
\newtheorem{corollary}[theorem]{Corollary}
\newtheorem{proposition}[theorem]{Proposition}

\theoremstyle{definition}
\newtheorem{definition}[theorem]{Definition}
\newtheorem{remark}[theorem]{Remark}
\newtheorem{example}[theorem]{Example}

\definecolor{darkblue}{rgb}{0,0,0.7} 
\newcommand{\darkblue}{\color{darkblue}} 
\newcommand{\Dfn}[1]{\emph{\darkblue #1}} 
\newcommand{\eps}{\varepsilon}
\newcommand\Sage{\texttt{Sage}}
\newcommand\plantri{\texttt{plantri}}
\newcommand\Bertini{\texttt{Bertini}}
\newcommand\deltassr{\mathrm{SSR}}
\newcommand\deltaL{\mathrm{SSL}} 
\newcommand\deltaC{\mathrm{SSC}} 
\newcommand\graph{\Gamma}
\newcommand\simplgraph{\overline{\graph}}
\newcommand\graphtri{\graph_T}
\newcommand\graphdis{\graph_D}
\newcommand\hz{\hphantom{0}}

\newcommand\range{{\rm R}}
\newcommand\RMS{{\rm RMS}} 

\title[Area difference bounds for odd dissections of a square]{Area difference bounds for dissections of a square\\ into an odd number of triangles}
\thanks{This research was supported by the DFG Collaborative Research Center TRR~109 ``Discretization in Geometry and Dynamics.''}

\author[J.-P.~Labb\'e]{Jean-Philippe Labb\'e} 
\address[J.-P. Labb\'e]{Institut f\"ur Mathematik, Freie Universit\"at Berlin, Arnimallee 2, 14195 Berlin, Germany}
\email{labbe@math.fu-berlin.de}
\urladdr{http://page.mi.fu-berlin.de/labbe}

\author[G. Rote]{G\"unter Rote}
\address[G. Rote]{Institut f\"ur Informatik, Freie Universit\"at Berlin, Takustra\ss e 9, 14195 Berlin, Germany}

\email{rote@inf.fu-berlin.de}
\urladdr{http://page.mi.fu-berlin.de/rote} 

\author[G.M. Ziegler]{G\"unter M. Ziegler} 
\address[G.M. Ziegler]{Institut f\"ur Mathematik, Freie Universit\"at Berlin, Arnimallee 2, 14195 Berlin, Germany}
\email{ziegler@math.fu-berlin.de}
\urladdr{http://page.mi.fu-berlin.de/gmziegler} 

\begin{document}

\begin{abstract}
Monsky's theorem from 1970 states that a square cannot be dissected
into an odd number $n$ of triangles of the same area, but it does not
give a lower bound for the area differences that must occur.

We extend Monsky's theorem to ``constrained framed maps''; based on this 
we can apply a gap theorem from semi-algebraic geometry to a polynomial
area difference measure and thus get a lower bound for the area differences 
that decreases doubly-exponentially with $n$. 
On the other hand, we obtain the first superpolynomial upper bounds
for this problem, derived from an explicit construction that uses
the Thue--Morse sequence.
\end{abstract}

\maketitle

\tableofcontents

%
%

\section{Introduction}\label{sec:Intro}

Around fifty years ago, in 1967, the following unsolved geometric problem appeared in
the \textit{American Mathematical Monthly} \cite{richman_problem_1967}:
\begin{quote}
    \emph{Let $N$ be an odd integer. Can a rectangle be dissected into 
    $N$ nonoverlapping triangles, all having the same area?}
\end{quote} 
The answer is now known to be negative, there is no such dissection:
This was first established by Thomas \cite{thomas_dissection_1968}
for dissections for which all vertex coordinates are rational with odd denominator; 
Monsky \cite{monsky_dividing_1970} subsequently 
extended the proof to general dissections.
Monsky's theorem inspired studies of generalizations and
related problems for trapezoids, centrally symmetric polygons, 
as well as higher-dimensional versions: 
A nice 1994 book by Stein and Szab\'o \cite[Chap.~5]{stein_algebra_1994} surveyed
this, but also after that there was continued interest and a lot of additional work, see  \cite{stein_cutting_1999,stein_generalized_2000,praton_cutting_2002,stein_cutting_2004,rudenko_equidissection_2013,rudenko_arithmetic_2014}.

Monsky's theorem says that a dissection of a square into an odd number of triangles
cannot be done if we require the triangles to have the same area \emph{exactly}, 
but it does not say how \emph{close} one can get.
Indeed, neither Monsky's proof, nor the only known alternative Ansatz by 
Rudenko \cite{rudenko_arithmetic_2014}, seems to yield an estimate for this.  
The \emph{quantitative} study of dissections, i.e., the 
necessary differences in areas of triangles, was formalized as an optimization problem around 15 years ago, see \cite{mansow_ungerade_2003,ziegler_open_2006}. 

To measure the area differences, we use the \emph{range} of the 
areas $a_1,a_2,\dots,a_n$ of the triangles in a dissection~$D$:%
\begin{equation}
  \label{eq:range-def}
	\range(D):=\max_{i, j\in[n]}|a_i-a_j|.
\end{equation}
Thus we are interested in the behavior of the function
\[
	\Delta(n):= 
\min\{\,\range(D) \mid D \text{ is a dissection of the square into } n \text{ triangles}\,\},
\]
which measures the minimal range of a dissection into $n$ triangles,
for odd $n\ge3$. For example, $\Delta(3)=1/4$,
because the best dissections with 3 triangles have areas $1/4,1/4,1/2$. 

The main results of this paper are the bounds 
\begin{equation}
\frac1{2^{2^{O(n)}}} 
\le {\Delta(n)} \leq 
\frac1{2^{\Omega(\log^2{n})}},
\tag{$*$}
		\label{eq:main-reformed}
\end{equation}
or alternatively, on a logarithmic scale,
\begin{equation*}
 2^{O(n)} \ge
  \log \frac 1{\Delta(n)}
\ge\Omega(\log^2{n}).
\end{equation*} 
Although in our main results 
we quantify the area differences in a dissection $D$ in terms of the range,
other measures are possible, and turn out to be useful.
Alternatives include the root mean square error of the areas 
(the standard deviation).
This differs from the range at most by a factor of $\sqrt n$, and thus has the same asymptotics
on the logarithmic scale, but we will demonstrate that 
for specific values of $n$ we get different optimal solutions.

The left inequality 
in~\eqref{eq:main-reformed} provides a doubly-exponential lower bound for the range of areas
in any odd dissection. 
To prove it, we introduce in Section~\ref{sec:area-discrep}
the {sum of squared residuals}, without taking square roots, as
this is a polynomial function that can be directly treated with
real algebraic techniques.
This is then used in Section~\ref{sec:lower_bound} to derive the
lower bound for the range from a general ``gap theorem'' in real
algebraic geometry (Theorem~\ref{thm:double_exp_lower}). 
We emphasize that we do \emph{not} obtain a new independent proof
of Monsky's theorem, as we use Monsky's result in the proof.

The first improvement over the trivial upper bound of $\Delta(n)=O(1/n^2)$
was due to Schulze \cite{schulze_area_2011},
who in 2011 provided a family of triangulations with 
$\Delta(n)=O(1/n^3)$.
This is still the best known bound for triangulations, and so far,
there were also no better bounds for the more general class of dissections.
Our upper bound in~\eqref{eq:main-reformed} goes far beyond this bound.
We prove it in Section~\ref{superpolynomial} by constructing a family
of dissections with the help of the Thue--Morse sequence
(Theorem~\ref{thm:superpoly}).
On a logarithmic scale, Schulze's upper bound on $\Delta(n)$ can be written as
 $\log\frac{1}{\Delta(n)} \ge 3 \log n\, (1+o(1))$.
Our bound $\log \frac1{\Delta(n)} \ge \Omega(\log^2n)$ provides 
the first \emph{superpolynomial} upper bound for the range of areas in
a family of dissections.

The lower bounds we construct on the area differences of dissections
are, in particular, valid for the special case of triangulations.
On the other hand, we do not have a construction of triangulations
that would improve on Schulze's upper bounds, but we hope
that this could be achieved by an extension of our techniques.

The present text is structured as follows. 
Section~\ref{sec:monsky_revisit} provides definitions, background and
a review of the proof of Monsky's theorem (Theorem~\ref{thm:monsky}).
In Section~\ref{sec:recast} we construct a setting of “framed maps” and “constrained framed maps” 
that generalizes dissections, and extend Monsky's theorem to
framed maps (Theorem~\ref{thm:frame_monsky}). Section~\ref{sec:area-discrep} 
introduces the area difference polynomials. They allow us to 
apply estimates from semi-algebraic geometry in order to obtain
the lower bounds of~\eqref{eq:main-reformed} in Section~\ref{sec:lower_bound} 
(Theorem~\ref{thm:double_exp_lower}).
In Section~\ref{sec:experiment} we report results of computational 
enumerations of combinatorial types and optimal dissections for small 
odd $n$ for various area difference measures.
In Section~\ref{sec:upper_bounds}  (Theorem~\ref{thm:superpoly})
we prove the superpolynomial upper bounds in~\eqref{eq:main-reformed}.
Finally, in Section~\ref{sec:even}, we briefly discuss dissections into an even
number of triangles. There are combinatorial types of dissections or
triangulations for which even areas cannot be achieved, for various
reasons. We show that our lower bounds  carry over to such cases.

\section{Background}\label{sec:monsky_revisit}

In this section we define dissections of simple polygons and review the coloring used in Monsky's proof.

\subsection{Dissection of simple polygons}

\begin{definition}[Simple polygon, sides, corners]
A \Dfn{simple polygon} is a compact subset $P\subset\mathbb{R}^2$
whose boundary is a simple (i.e., nonintersecting) closed curve formed by finitely many line segments.
The \Dfn{sides} of~$P$ are the maximal line segments on the boundary of~$P$.
The \Dfn{corners} of $P$ are the endpoints of the sides of~$P$.
\end{definition}

A simple polygon with $k$ sides is a simple $k$-gon, and
a simple polygon with three sides is a triangle.
Hence, we may refer to corners and sides of a triangle on the plane.

\begin{definition}[Dissections and triangulations of a simple polygon]
A \Dfn{dissection} of a simple polygon~$P$ is a finite set of triangles with disjoint interiors that cover~$P$.
If every pairwise intersection is either empty, a corner, or a common side of both triangles, the dissection is a \Dfn{triangulation}.
\end{definition}

To distinguish between dissections and triangulations, we say that a simple polygon is dissected, or triangulated.
Figure~\ref{fig:ex_polygon} illustrates the distinction.
The main tool used to encode the combinatorial structure of a dissection is the following labeled graph \cite{monsky_dividing_1970,abrams_spaces_2014}.

\begin{definition}[Skeleton graph of a dissection, nodes, edges]
Let $D=\{t_1,t_2,\dots,t_n\}$ be a dissection.
The \Dfn{nodes} $V(\graphdis)$ of the \Dfn{skeleton graph} $\graphdis$ of $D$ are the corners of the triangles in $D$.
There is an \Dfn{edge} between two nodes if they are on the same side of a triangle of $D$ and the line segment joining them does not contain other
nodes.
\end{definition}

Since the skeleton graph of a dissection comes with a specific embedding on the plane,
i.e.\ as a plane graph,
it is possible to define a \Dfn{face} of the skeleton graph~$\graph$ as the cycle obtained on the boundary of a triangle of the dissection.
We abuse language and refer to a \Dfn{triangular face} for either the triangle (as a subset of the plane) 
or to its boundary cycle with distinguished corners.
We also consider the outside face as a face of~$\graph$: It is the
cycle $B$ formed by the edges on the sides of the simple polygon.

The face structure is actually uniquely defined just by the graph~$\graph$ together
with the boundary cycle $B$. It is easy to show that the skeleton
$\graph$ 
must be \emph{internally $3$-connected}: It must become $3$-connected if we add an outside vertex
and connect it to all vertices of $B$.
Otherwise, $\graph$ cannot even be drawn with convex faces.
(If the graph has no degree-2 vertices except on $B$, this condition
is also sufficient; see Tutte~\cite{tutte-60}
or Thomassen~\cite[Theorem 5.1]{thomassen-80}
 for more precise statements.)
 It is well-known that $3$-connected graphs have a unique
combinatorial embedding in the plane, i.e., the set of face cycles is fixed.

\begin{definition}[Boundary nodes, internal nodes, corner nodes, side nodes, boundary edges, internal edges]
\Dfn{Boundary nodes} of $\graphdis$ are corners of triangles lying on the boundary of~$P$ (i.e., the outside face of $\graphdis$).
Nodes of $\graphdis$ that are not boundary nodes are \Dfn{internal nodes}.
\Dfn{Corner nodes} of $\graphdis$ are boundary nodes that are corners of $P$ (i.e., corners of the outside face).
Boundary nodes of $\graphdis$ that are not corners nodes are \Dfn{side nodes}.
\Dfn{Boundary edges} of $\graphdis$ lie on the sides of $P$ (i.e., on the sides of the outside face).
Other edges are called \Dfn{internal}.

We define similarly the boundary, corner and side nodes with respect to each triangular face of~$\graphdis$.
They all have three \Dfn{corner nodes} and possibly \Dfn{side nodes} that are corners of other triangles of~$D$ that lie in the interior of their sides.
\end{definition}

\begin{figure}[!ht]
\resizebox{0.8\hsize}{!}{\begin{tabular}{cc}

%
%
%
%
%
%
%
%
%
%
%
\begin{tikzpicture}[scale=1,
edge/.style={draw=blue!95!black,thick,line join=bevel,line cap=round},
vertex/.style={circle,inner sep=1pt,fill=blue}]

\coordinate (1) at (0,0);
\coordinate (2) at (2,0);
\coordinate (3) at (3,3);
\coordinate (4) at (1.5,2.5);
\coordinate (5) at (2,5);
\coordinate (6) at (-2,4);
\coordinate (7) at (0,3);
\coordinate (8) at (-2,2);

\coordinate (9) at (0,2);
\coordinate (10) at (1.75,1.25);

\coordinate (11) at (-1,1);

\filldraw[fill=blue!20!white, edge] (1) -- (2) -- (3) -- (4) -- (5) -- (6) -- (7) -- (8) -- cycle;

\draw[edge] (5) -- (7) -- (4) -- (1) -- (9) -- (8);
\draw[edge] (7) -- (9);

\draw[edge] (2) -- (10) -- (3);
\draw[edge] (10) -- (4);

\draw[edge] (9) -- (11);

\node[vertex] at (1) {};
\node[vertex] at (2) {};
\node[vertex] at (3) {};
\node[vertex] at (4) {};
\node[vertex] at (5) {};
\node[vertex] at (6) {};
\node[vertex] at (7) {};
\node[vertex,label=left:corner] at (8) {};
\node[vertex,label=right:{\small internal}] at (9) {};
\node[vertex] at (10) {};
\node[vertex,label=left:side] at (11) {};

\draw[thick, decorate,decoration={brace,amplitude=10pt},xshift=-1.25cm,yshift=0cm] (-0.75,0.75) -- (-2,2) node [midway,xshift=-1.25cm,yshift=-0.25cm] {boundary};

\end{tikzpicture}

&

\begin{tikzpicture}[scale=1,
edge/.style={draw=blue!95!black,thick,line join=bevel,line cap=round},
vertex/.style={circle,inner sep=1pt,fill=blue}]

\coordinate (1) at (0,0);
\coordinate (2) at (2,0);
\coordinate (3) at (3,3);
\coordinate (4) at (1.5,2.5);
\coordinate (5) at (2,5);
\coordinate (6) at (-2,4);
\coordinate (7) at (0,3);
\coordinate (8) at (-2,2);

\coordinate (9) at (0,2);
\coordinate (10) at (1.75,1.25);
\coordinate (11) at (1,1.25);

\coordinate (12) at (-1,1);
\coordinate (13) at (0,1.5);

\filldraw[fill=blue!20!white, edge] (1) -- (2) -- (3) -- (4) -- (5) -- (6) -- (7) -- (8) -- cycle;

\draw[edge] (1) -- (11) -- (2) -- (4) -- (11) -- (7) -- (13) -- (1);
\draw[edge] (4) -- (7) -- (5);

\draw[edge] (7) -- (12) -- (13) -- (11);

\node[vertex] at (1) {};
\node[vertex] at (2) {};
\node[vertex] at (3) {};
\node[vertex] at (4) {};
\node[vertex] at (5) {};
\node[vertex] at (6) {};
\node[vertex] at (7) {};
\node[vertex] at (8) {};
\node[vertex] at (11) {};
\node[vertex] at (12) {};
\node[vertex] at (13) {};

\end{tikzpicture}

\end{tabular}}
\caption{A dissection and a triangulation of a simple polygon; some boundary/\allowbreak corner/\allowbreak side/\allowbreak internal nodes are marked.}\label{fig:ex_polygon}
\end{figure}
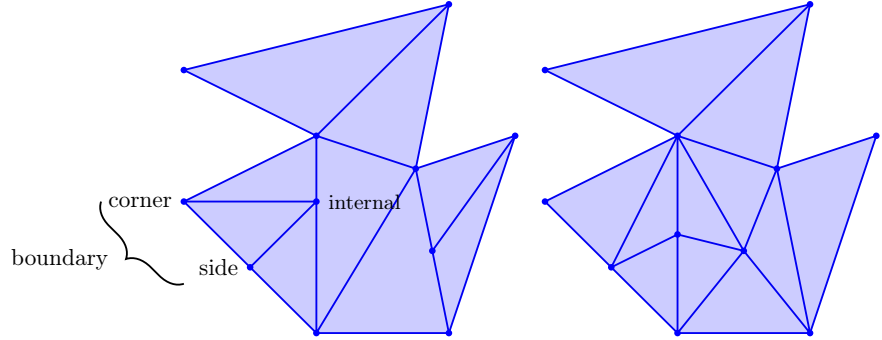

\begin{example}
The skeleton graphs of the triangulation and of the dissection in Figure~\ref{fig:ex_polygon} both have $11$ nodes,
namely $9$ boundary nodes ($8$ corner nodes and $1$ side node) and $2$ internal nodes.
They have $9$ boundary edges and all other edges are internal edges. 
\end{example}

In the following we will omit the subscript from $\graphdis$ 
and simply write~$\graph$ whenever the context is clear.
 We let $n$ be the number of triangles in a dissection, 
unless otherwise stated, and denote the triangles as $t_1,t_2,\dots,t_n$.

\subsection{Monsky's theorem}

We define a $3$-coloring of the nodes of the skeleton graph $\graph$, using a 
\Dfn{$2$-adic valuation} $|\cdot|_2$.
To define this on $\mathbb{Q}$, we set $|0|_2:=0$ and
\[
	\left|2^n{r}/{s}\right|_2:=2^{-n}.
\]
for any nonzero rational number $2^n{r}/{s}$ with $n\in\mathbb{Z}$ and odd integers $r,s$.
This function on $\mathbb{Q}$ satisfies the axioms of a valuation:
\[
1)\ |x|\geq 0, \text{ and } |x|=0 \Longleftrightarrow x=0,\qquad
2)\ |xy|=|x|\cdot |y|, \qquad 
3)\ |x+y|\leq \max\{|x|,|y|\}.
\] 
Condition 3) is stronger than the usual triangle inequality satisfied by a norm.
All valuations also satisfy $|{+}1|=|{-}1|=1$. In 3), equality holds unless $|x|=|y|$.
This valuation can be extended (in a non-canonical way) from
$\mathbb{Q}$ 
to $\mathbb{R}$.
For a complete account on how to produce such an extension, we refer the reader to~\cite[Chap.~5]{stein_algebra_1994} or \cite[Chap.~22]{aigner_proofs_2014}.
 
Any $2$-adic valuation on $\mathbb{R}$ determines a coloring of the points
 $(x,y)\in\mathbb{R}^2$ of the plane, as follows.
Let $m(x,y):=\max\{|x|_2,|y|_2,1\}$.
The first entry of $(|x|_2, |y|_2, 1)$ that equals $m(x,y)$ determines which of the three colors the point $(x,y)$ receives:
\begin{itemize}
	\item [red] if $|x|_2=m(x,y)$,
	\item [green] if $|x|_2<|y|_2=m(x,y)$, and 
	\item [blue] if $|x|_2,|y|_2<1=m(x,y)$. 
\end{itemize}
These cases cover all possibilities and are mutually exclusive.
A triangle is \Dfn{colorful} if it has corners of all three colors.
We classify line segments according to the colors of their endpoints.

The $3$-coloring of the plane
has the following crucial properties.
\begin{lemma}[{see e.g.\ \cite[Chapter~22, Lemma~1]{aigner_proofs_2014}}]\label{lem:property_coloring}
The $2$-adic valuation of the area of a colorful triangle is at least $2$. 
In particular, the area of a colorful triangle cannot be $0$ or
of the form $r/s$ with integer $r$ and odd~$s$, since $|r/s|_2\le 1$ in this case.
Consequently, every line contains points of at most
\textup(indeed, exactly\textup) two different colors.
\end{lemma}

From these properties, we derive the following version of Monsky's theorem,
which can readily be obtained from the original proof.
In Section~\ref{ssec:extension_monsky}, we prove an extension of this theorem to a larger class of objects (Theorem~\ref{thm:frame_monsky}), namely
\emph{constrained framed maps} of skeleton graphs of dissections. 

\begin{theorem}[Monsky \cite{monsky_dividing_1970}]\label{thm:monsky}
Let $E$ be a positive integer and $P$ be a simple polygon of area~$E$.
If~$P$ has an odd number of red-blue sides, then $P$ cannot be dissected into an
odd number of triangles of equal area.
\end{theorem}

Note that the assumptions depend on properties that have, per se,
nothing to do with the problem: Dissectability into equal-area triangles
is invariant under affine transformations, whereas the
assumptions
(integral area and the coloring of the corners) are obviously not.
The coloring is not even invariant under translations. 
Moreover, if $P$ has irrational corners, the coloring is not canonical,
since the extension of the valuation from $\mathbb{Q}$
to $\mathbb{R}$ depends on arbitrary choices. 
Once such a valuation on $\mathbb{R}$ is fixed, the coloring of the
corners is determined, and the assumptions of the theorem can be checked.

\subsection{Error measures}

In the introduction we summarized our main results in terms of the {range} $\range(D)$  
of the triangle areas of an odd dissection~\eqref{eq:range-def}. 
Alternative measures for the deviation of the areas from the average value
will be important in the following. In particular, 
given a dissection $D$ of a simple polygon $P$ of area $E$ into triangles of areas $a_1,a_2,\dots,a_n$,
the \Dfn{root mean square} (\RMS) error is defined as
\[
	\RMS(D):= \sqrt{\frac{1}{n}\sum_{i=1}^n\left(a_i-\frac{E}{n}\right)^2}.
\]

If we restrict ourselves to the set of framed maps coming from dissections, the following proposition 
shows that obtaining a lower bound for the~\RMS\ error implies directly a lower bound on the range, and 
an upper bound on the range gives an upper bound on the \RMS~error.

\begin{proposition}\label{prop:lower_bound}
For a dissection $D$ into $n$ triangles,
the range $\range(D)$ and the root mean square error $\RMS(D)$ are related as follows\textup:
\[
	\frac{\range(D)}{2\sqrt n}\leq \RMS(D) \leq \range(D)
\]
\end{proposition}
The upper bound can actually be strengthened to $\RMS(D) \leq
\range(D)/2$, which is tight for all even~$n$.
For simplicity, we only prove the weaker bound.

\begin{proof}
Let $E$ be the area of $P$, and
write $\lambda_i=a_i-\frac{E}{n}$.
Then $\sqrt n\cdot\RMS(D)
=\sqrt{\sum_{i=1}^n \lambda_i^2}$ is
the $2$-norm of the vector $(\lambda_i)_{i=1}^{n}$, whereas the range is related to the maximum
norm as follows:
\begin{equation}
  \label{eq:max-norm}
\max_{i\in[n]} |\lambda_i|\le
\range(D)\leq 2 \max_{i\in[n]} |\lambda_i|. 
\end{equation}
 The standard bound between the maximum norm and the 2-norm gives 
$\max_{i} |\lambda_i|\leq\sqrt n\cdot\RMS(D)\le \sqrt n \max_{i} |\lambda_i|$.
Together with \eqref{eq:max-norm}, this gives the claimed result.
\end{proof}

\section{Monsky's theorem for constrained framed maps of dissection skeleton graphs}\label{sec:recast}

The goal of this section is to extend Monsky's theorem to a more general version 
for which lower bounds on the range can be obtained more easily.  
 We define combinatorial types of dissections in
 Section~\ref{ssec:diss_comb_type}. Section~\ref{ssec:diss_coll_constr}
 describes how we deal with collinearity constraints.
In Section~\ref{ssec:framed_maps}, we define framed maps and constrained framed maps of skeleton graphs of dissections. 
In Section~\ref{ssec:signed_areas}, we define signed areas with respect to framed maps. 
Finally, in Section~\ref{ssec:extension_monsky}, we extend Monsky's
theorem to constrained framed maps of skeleton graphs of dissections
of a simple polygon.

\subsection{Combinatorial types of dissections}\label{ssec:diss_comb_type}

\begin{definition}[Combinatorial data of a dissection]
The \Dfn{combinatorial data} of a dissection $D$ of a simple
$k$-gon~$P$ into triangles $t_1,\dots, t_n$ are given by the quadruple
$(\graph,B,C,\mathcal{T})$:~$\graph$ is
the skeleton graph of $D$; $B\subseteq V(\graph)$ is the vertex set of the boundary cycle
of~$\graph$;
$C=(c_1,\dots,c_k)$ is the sequence of
 corner nodes of $P$ in cyclic order; and finally,
\[
	\mathcal{T}=\left\{\{v^i_1,v^i_2,v^i_3\}\mid 1\leq i\leq n\right\},
\]	
where $\{v^i_1,v^i_2,v^i_3\}$ are the corners of $t_i$. 
\end{definition}
  
\begin{definition}[Abstract dissection]
An \Dfn{abstract dissection of a $k$-gon} is a quadruple
$\mathcal{D}=(\graph,B,C,\mathcal{T})$ with the following conditions:
\begin{enumerate}
\item $\graph$ is a planar graph, with a plane drawing bounded by a simple cycle with vertex set $B$. 
\item $\graph$ is internally $3$-connected with respect this drawing.
\item $C=(c_1,\dots,c_k)$ is a sequence of $k$ vertices in $B$, which occur in this order on the boundary cycle.
\item $\mathcal{T}$ consists of, for each interior face of $\graph$,
 a triplet of vertices from the boundary of this face.
\label{triplets}
\end{enumerate}
\end{definition}

As was mentioned before, the face structure of an internally
$3$-connected graph is unique, given the outer face~$B$; thus,
condition~(\ref{triplets}) is well-defined even if $\graph$ is just
given as an abstract graph.
Alternatively,
we might consider an abstract dissection as a
plane graph together with the additional data $C$ and $\mathcal{T}$.

An isomorphism between abstract dissections is a graph isomorphism that
preserves $B$, $C$ and~$\mathcal{T}$.
Abstract dissections capture the notion of a {combinatorial type}
of a dissection.
We say that
two dissections $D$ and $D'$ of a simple polygon $P$ have the same
\Dfn{combinatorial type} if their combinatorial data are isomorphic
when considered as abstract dissections.
A combinatorial type is thus an isomorphism class of abstract
dissections.

\begin{example}
Figure~\ref{fig:combi_type} shows two dissections of the square that have isomorphic skeleton graphs.
The isomorphism  fixes the corner nodes, but it does not induce a bijection between the corners of
the triangles: the internal node~$x$ is a corner of $4$ triangles in the first dissection, while no node is a corner of $4$ triangles in the second.
Hence they do not have the same combinatorial type.

\begin{figure}[!ht]
\begin{tabular}{c@{\hspace{2cm}}c}

\begin{tikzpicture}[scale=2.5,
edge/.style={color=blue!95!black,thick,line join=bevel,line cap=round},
vertex/.style={circle,inner sep=1pt,fill=blue}]

\coordinate (1) at (0.000000,0.000000);
\coordinate (2) at (1.00000,0.000000);
\coordinate (3) at (1.00000,1.00000);
\coordinate (4) at (0.000000,1.00000);
\coordinate (5) at (0.250000,0.666667);
\coordinate (6) at (0.500000,0.333333);

\draw[edge] (1) -- (2);
\draw[edge] (1) -- (4);
\draw[edge] (1) -- (6);
\draw[edge] (2) -- (3);
\draw[edge] (2) -- (6);
\draw[edge] (3) -- (4);
\draw[edge] (3) -- (5);
\draw[edge] (3) -- (6);
\draw[edge] (4) -- (5);
\draw[edge] (5) -- (6);

\node[vertex] at (1) {};
\node[vertex] at (2) {};
\node[vertex] at (3) {};
\node[vertex] at (4) {};
\node[vertex] at (5) {};
\node[vertex,label=below:$x$] at (6) {};

\end{tikzpicture}

&

\begin{tikzpicture}[scale=2.5,
edge/.style={color=blue!95!black,thick,line join=bevel,line cap=round},
vertex/.style={circle,inner sep=1pt,fill=blue}]

\coordinate (1) at (0.000000,0.000000);
\coordinate (2) at (1.00000,0.000000);
\coordinate (3) at (1.00000,1.00000);
\coordinate (4) at (0.000000,1.00000);
\coordinate (5) at (0.333333,0.666667);
\coordinate (6) at (0.166667,0.333333);

\draw[edge] (1) -- (2);
\draw[edge] (1) -- (4);
\draw[edge] (1) -- (6);
\draw[edge] (2) -- (3);
\draw[edge] (2) -- (6);
\draw[edge] (3) -- (4);
\draw[edge] (3) -- (5);
\draw[edge] (3) -- (6);
\draw[edge] (4) -- (5);
\draw[edge] (5) -- (6);

\node[vertex] at (1) {};
\node[vertex] at (2) {};
\node[vertex] at (3) {};
\node[vertex] at (4) {};
\node[vertex] at (5) {};
\node[vertex,label={[shift={(0.1,-0.5)}]$x$}] at (6) {};

\end{tikzpicture}

\end{tabular}
\caption{Two different combinatorial types of dissections of the square with isomorphic skeleton graphs}\label{fig:combi_type}
\end{figure}
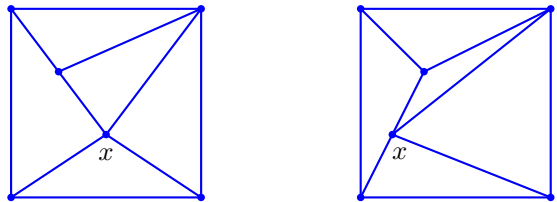
\end{example}

\subsection{Collinearity constraints of dissections}\label{ssec:diss_coll_constr}

Let $D$ be a dissection of a simple polygon $P$.
Any set of three distinct nodes of $\graph$ that lie on a side of a face of $\graph$ (which might be the outside face) form a \Dfn{collinearity constraint} of $D$.
We now describe the selection of a certain set of collinearity constraints that will play an important role later on.

\begin{definition}[Reduced system of collinearity constraints of a dissection, simplicial graph of a dissection]\label{def:red_system}
Let $D$ be a dissection of a simple polygon $P$.
Let $t\in D$ and let $c, c'$ be two corners of~$t$.
Assume that the line segment between $c$ and $c'$ contains in its interior the side nodes $v_1,\dots,v_j$ ordered from $c$ to $c'$, with $j\geq 1$.
Add to $\graph$ the edges $cv_i$ for $1<i\le j$, as well as $cc'$:
The resulting graph is again plane, and it gets the triangles
$cv_iv_{i+1}$ for $1\le i<j$, as well as $cv_jc'$. 
Repeat this procedure for each side of a triangle~$t\in D$, and
for all sides of the outside face of $\graph$.  
All the sets of three corners of a triangle added in this process
are put together in order to get a \Dfn{reduced system of collinearity constraints}.
See Figure~\ref{fig:schem_coll} for an illustration.

\begin{figure}[!ht]
\begin{tikzpicture}[scale=1.5,
edge/.style={color=blue!95!black,thick,line join=bevel,line cap=round},
lincon/.style={color=blue!95!black,thick,line join=bevel,line cap=round,dashed},
side/.style={color=blue!95!black,thin,line join=bevel,line cap=round,dotted},
vertex/.style={circle,inner sep=1pt,fill=blue}]

\coordinate (1) at (0.000000,0.000000);
\coordinate (2) at (0:2.5);
\coordinate (3) at (60:2.5);

\coordinate (121) at (0:1.25);
\coordinate (231) at ($(2)!0.25!(3)$);
\coordinate (232) at ($(2)!0.5!(3)$);
\coordinate (233) at ($(2)!0.75!(3)$);

\coordinate (131) at ($(1)!0.25!(3)$);
\coordinate (132) at ($(1)!0.5!(3)$);
\coordinate (133) at ($(1)!0.75!(3)$);

\draw[side] (-0.5,-0.25) -- (3,-0.25) -- (3,2.5) -- (-0.5,2.5) -- cycle;

\node at (2.75,2.25) {$P$};
\node at (1.25,-0.86) {};

\draw[edge] (1) -- (2);
\draw[edge] (2) -- (3);
\draw[edge] (3) -- (1);
\draw[edge] (121) -- +(300:0.1);

\draw[edge] (231) -- +(75:0.1);
\draw[edge] (232) -- +(45:0.1);
\draw[edge] (233) -- +(30:0.1);

\draw[edge] (131) -- +(90:0.1);
\draw[edge] (132) -- +(135:0.1);
\draw[edge] (133) -- +(150:0.1);

\draw[edge] (3) -- +(10:0.1);
\draw[edge] (3) -- +(170:0.1);
\draw[edge] (3) -- +(90:0.1);

\draw[edge] (1) -- +(150:0.1);
\draw[edge] (2) -- +(90:0.1);

\draw[lincon] (1) to [out=20,in=160] (2);

\draw[lincon] (1) to [out=30,in=270] (132);
\draw[lincon] (1) to [out=30,in=260] (133);
\draw[lincon] (1) to [out=30,in=270] (3);

\draw[lincon] (231) to [out=140,in=270] (3);
\draw[lincon] (232) to [out=150,in=270] (3);
\draw[lincon] (2) to [out=150,in=270] (3);

\node[vertex] at (1) {};
\node[vertex] at (2) {};
\node[vertex] at (3) {};
\node[vertex] at (121) {};
\node[vertex] at (231) {};
\node[vertex] at (232) {};
\node[vertex] at (233) {};
\node[vertex] at (131) {};
\node[vertex] at (132) {};
\node[vertex] at (133) {};

\coordinate (a) at (-0.5,-0.25);
\coordinate (b) at (3,-0.25);

\coordinate (aba) at ($(a)!0.25!(b)$);
\coordinate (abb) at ($(a)!0.5!(b)$);
\coordinate (abc) at ($(a)!0.75!(b)$);

\draw[edge] (a) -- (b);

\draw[edge] (a) -- +(45:0.1);

\draw[edge] (aba) -- +(75:0.1);
\draw[edge] (abb) -- +(45:0.1);
\draw[edge] (abb) -- +(135:0.1);
\draw[edge] (abc) -- +(30:0.1);
\draw[edge] (abc) -- +(135:0.1);

\draw[edge] (b) -- +(120:0.1);

\draw[lincon] (a) to [out=-30,in=210] (b);
\draw[lincon] (aba) to [out=-30,in=200] (b);
\draw[lincon] (abb) to [out=-20,in=200] (b);

\node[vertex] at (a) {};
\node[vertex] at (b) {};
\node[vertex] at (aba) {};
\node[vertex] at (abb) {};
\node[vertex] at (abc) {};

\end{tikzpicture}
\caption{The addition of edges to the skeleton graph $\graph$ yields new triangular faces representing collinearity constraints.
}\label{fig:schem_coll}
\end{figure}
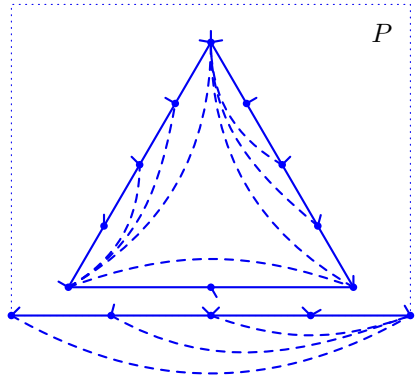

\noindent This procedure yields a supergraph $\simplgraph$ of $\graph$ that contains no new nodes but may contain some new edges between side nodes of triangles and $P$.
We refer to the graph $\simplgraph$ as a \Dfn{simplicial graph} of~$D$.
\end{definition}

A similar procedure is described implicitly in~\cite[Sect.~3]{rudenko_arithmetic_2014} and \cite[Sect.~2]{abrams_spaces_2014}.
By construction, $\simplgraph$ is a plane graph, and we take it with the
specified embedding, so its bounded faces (which are triangles)
inherit the  orientation from the plane.
In other words, the graph~$\simplgraph$ is the 1-skeleton of a simplicial complex that is homeomorphic to a 2-dimensional ball, whose triangles correspond to triangles of the original dissection and possibly triangles given by a reduced system of collinearity constraints.
The edges of the outer face corresponds to the sides of the polygon~$P$.

\begin{example}\label{ex:lin_constraints}
Consider the dissection $D$ of the square shown in Figure~\ref{fig:lin_constraints}.
This dissection has two simplicial graphs with reduced systems of collinearity constraints $\{\{b_1,c_1,c_2\},\{i_1,b_1,i_2\},\{i_2,b_1,c_3\}\}$ and
$\{\{b_1,c_1,c_2\},$ $\{i_1,b_1,c_3\},\{i_2,i_1,c_3\}\}$.

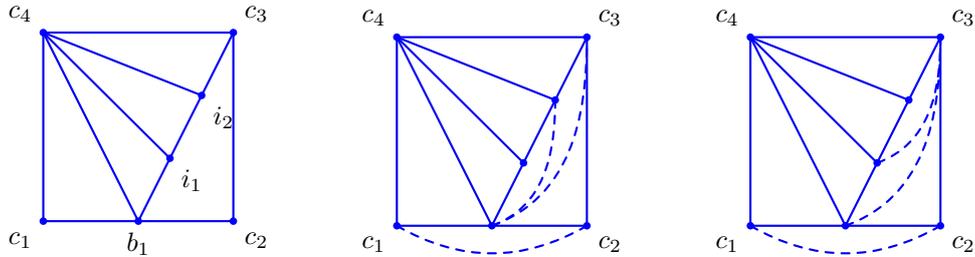
\begin{figure}[!ht]
\begin{tabular}{c@{\hspace{1cm}}c@{\hspace{1cm}}c}

\begin{tikzpicture}[scale=2.5,
edge/.style={color=blue!95!black,thick,line join=bevel,line cap=round},
vertex/.style={circle,inner sep=1pt,fill=blue}]

\coordinate (1) at (0.000000,0.000000);
\coordinate (2) at (1.00000,0.000000);
\coordinate (3) at (1.00000,1.00000);
\coordinate (4) at (0.000000,1.00000);
\coordinate (5) at (0.500000,0.000000);
\coordinate (6) at (0.666667,0.333333);
\coordinate (7) at (0.833333,0.666667);

\node[inner sep = 0pt, label=below left:$c_1$] at (1) {};
\node[inner sep = 0pt, label=below right:$c_2$] at (2) {};
\node[inner sep = 0pt, label=above right:$c_3$] at (3) {};
\node[inner sep = 0pt, label=above left:$c_4$] at (4) {};

\node[inner sep = 0pt, label=below:$b_1$] at (5) {};

\node[inner sep = 0pt, label=below right:$i_1$] at (6) {};
\node[inner sep = 0pt, label=below right:$i_2$] at (7) {};

\draw[edge] (1) -- (4);
\draw[edge] (1) -- (5);
\draw[edge] (2) -- (3);
\draw[edge] (2) -- (5);
\draw[edge] (3) -- (4);
\draw[edge] (3) -- (7);
\draw[edge] (4) -- (5);
\draw[edge] (4) -- (6);
\draw[edge] (4) -- (7);
\draw[edge] (5) -- (6);
\draw[edge] (6) -- (7);

\node[vertex] at (1) {};
\node[vertex] at (2) {};
\node[vertex] at (3) {};
\node[vertex] at (4) {};
\node[vertex] at (5) {};
\node[vertex] at (6) {};
\node[vertex] at (7) {};

\end{tikzpicture}

&

\begin{tikzpicture}[scale=2.5,
edge/.style={color=blue!95!black,thick,line join=bevel,line cap=round},
lincon/.style={color=blue!95!black,thick,line join=bevel,line cap=round,dashed},
vertex/.style={circle,inner sep=1pt,fill=blue}]

\coordinate (1) at (0.000000,0.000000);
\coordinate (2) at (1.00000,0.000000);
\coordinate (3) at (1.00000,1.00000);
\coordinate (4) at (0.000000,1.00000);
\coordinate (5) at (0.500000,0.000000);
\coordinate (6) at (0.666667,0.333333);
\coordinate (7) at (0.833333,0.666667);

\node[inner sep = 0pt, label=below left:$c_1$] at (1) {};
\node[inner sep = 0pt, label=below right:$c_2$] at (2) {};
\node[inner sep = 0pt, label=above right:$c_3$] at (3) {};
\node[inner sep = 0pt, label=above left:$c_4$] at (4) {};

\node[inner sep = 0pt] at (5) {};

\node[inner sep = 0pt] at (6) {};
\node[inner sep = 0pt] at (7) {};

\draw[edge] (1) -- (4);
\draw[edge] (1) -- (5);
\draw[edge] (2) -- (3);
\draw[edge] (2) -- (5);
\draw[lincon] (1) to [out=-30,in=210] (2);
\draw[lincon] (5) to [out=20,in=270] (3);
\draw[lincon] (5) to [out=20,in=270] (7);
\draw[edge] (3) -- (4);
\draw[edge] (3) -- (7);
\draw[edge] (4) -- (5);
\draw[edge] (4) -- (6);
\draw[edge] (4) -- (7);
\draw[edge] (5) -- (6);
\draw[edge] (6) -- (7);

\node[vertex] at (1) {};
\node[vertex] at (2) {};
\node[vertex] at (3) {};
\node[vertex] at (4) {};
\node[vertex] at (5) {};
\node[vertex] at (6) {};
\node[vertex] at (7) {};

\end{tikzpicture}

&

\begin{tikzpicture}[scale=2.5,
edge/.style={color=blue!95!black,thick,line join=bevel,line cap=round},
lincon/.style={color=blue!95!black,thick,line join=bevel,line cap=round,dashed},
vertex/.style={circle,inner sep=1pt,fill=blue}]

\coordinate (1) at (0.000000,0.000000);
\coordinate (2) at (1.00000,0.000000);
\coordinate (3) at (1.00000,1.00000);
\coordinate (4) at (0.000000,1.00000);
\coordinate (5) at (0.500000,0.000000);
\coordinate (6) at (0.666667,0.333333);
\coordinate (7) at (0.833333,0.666667);

\node[inner sep = 0pt, label=below left:$c_1$] at (1) {};
\node[inner sep = 0pt, label=below right:$c_2$] at (2) {};
\node[inner sep = 0pt, label=above right:$c_3$] at (3) {};
\node[inner sep = 0pt, label=above left:$c_4$] at (4) {};

\node[inner sep = 0pt] at (5) {};

\node[inner sep = 0pt] at (6) {};
\node[inner sep = 0pt] at (7) {};

\draw[edge] (1) -- (4);
\draw[edge] (1) -- (5);
\draw[edge] (2) -- (3);
\draw[edge] (2) -- (5);
\draw[lincon] (1) to [out=-30,in=210] (2);
\draw[lincon] (5) to [out=20,in=270] (3);
\draw[lincon] (6) to [out=20,in=270] (3);
\draw[edge] (3) -- (4);
\draw[edge] (3) -- (7);
\draw[edge] (4) -- (5);
\draw[edge] (4) -- (6);
\draw[edge] (4) -- (7);
\draw[edge] (5) -- (6);
\draw[edge] (6) -- (7);

\node[vertex] at (1) {};
\node[vertex] at (2) {};
\node[vertex] at (3) {};
\node[vertex] at (4) {};
\node[vertex] at (5) {};
\node[vertex] at (6) {};
\node[vertex] at (7) {};

\end{tikzpicture}

\end{tabular}
\caption{A dissection $D$ of a square with $5$ collinearity constraints (left) and two sketches of the graph $\simplgraph$ with new dashed
edges showing two different reduced systems of collinearity constraints (center and right)}\label{fig:lin_constraints}
\end{figure}
\end{example}

\begin{example}\label{illegal}
Figure~\ref{fig:does-not-enforce-collinearity}a--b shows a more elaborate
example of the construction of~$\simplgraph$.
Figure~\ref{fig:does-not-enforce-collinearity}c--d demonstrates
that a set of \emph{reduced} collinearity constraints is
per se not a substitute for the full set of collinearity constraints:
Although each triple in the reduced set
$\{
\{1,2,5\},
\{2,4,5\},
\{2,3,4\}
\}$
of collinearity constraints is collinear,
the line $15243$, which is supposed to be straight,
has a kink.

Nevertheless, our adaptation of Monsky's proof will exclude even such ``illegal''
solutions from having equal-area triangles.

  \begin{figure}[!ht]
    \centering
\includegraphics{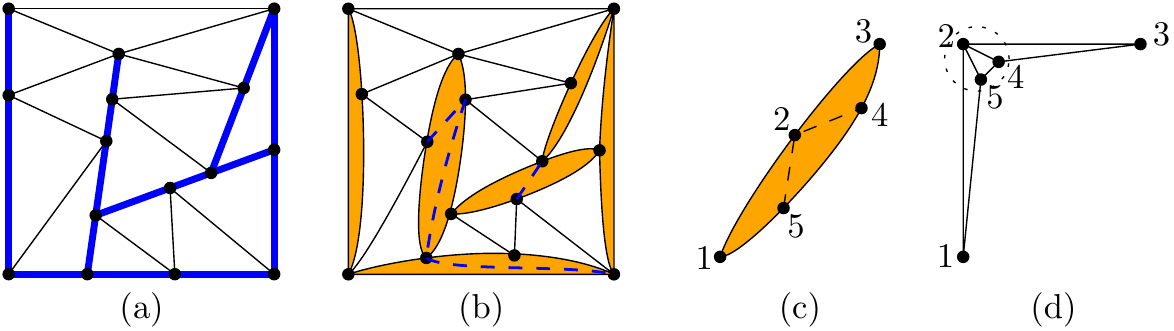}
\caption{(a) A dissection with collinearity constraints highlighted in
  blue; (b)~the graph $\simplgraph$ with new dashed
  edges showing a reduced systems of collinearity constraints;
(c)~a set of reduced collinearity constraints coming from a single
straight segment; (d) a drawing where all reduced collinearity
constraints are satisfied, shown symbolically:
The three points $2,4,5$ in the dotted circle should actually coincide.}
\label{fig:does-not-enforce-collinearity}
\end{figure}
  
\end{example}

The procedure to get a reduced system of collinearity constraints is not unique.
Nevertheless, as the next lemma shows, the collinearity constraints of a dissection 
are given by its combinatorial type.
Furthermore, the size of a reduced system of collinearity constraints is an invariant of the combinatorial type of the dissection:
it is equal to the total number of side nodes in $\graph$.

\begin{lemma}
Let $D$ and $D'$ be two dissections of a simple polygon $P$.
If $D$ and $D'$ have the same combinatorial type, then they have the
same sets of collinearity constraints.  
\end{lemma} 

Next, we give bounds on the number of nodes and the cardinality of a reduced system of collinearity constraints 
in relation to the number of triangles of a dissection.

\begin{lemma}\label{lem:combinatorial_prop}
Let $P$ be a simple polygon with $K$ corners and $D$ a dissection of~$P$ into~$n$ triangles, with a reduced system of collinearity
constraints $L$ of cardinality $\ell$.
\begin{enumerate}[label={\rm(\roman{enumi})},ref=(\roman{enumi})]
	\item The number of nodes of $\graph$ is $({n + K + \ell + 2})/{2}$. \label{lem:combinatorial_prop_i}
	\item The number of collinearity constraints satisfies $\ell\leq n - K + 2$. \label{lem:combinatorial_prop_ii}
\end{enumerate}
Therefore the number of nodes of $\graph$ is at least $({n+K+2})/{2}$ and at most $n+2$.
\end{lemma}

\begin{proof}
\ref{lem:combinatorial_prop_i} Let $N$ be the number of nodes of $\graph$ (and $\simplgraph$).
Consider the simplicial graph $\simplgraph$ and denote by~$e$ its number of edges.
The graph $\simplgraph$ has $n+\ell$ triangular faces (excluding the outside face) given by the $n$ triangular faces of $\graph$ and the $\ell$
triangles corresponding to collinearity constraints in $L$.
We count the number of occurrences ``an edge of a triangular face of $\simplgraph$'' in two ways.
First, each triangular face of $\simplgraph$ contributes 3 such occurrences, getting $3(n+\ell)$.
Doing the previous counting, all internal edges are counted twice, while the boundary edges are counted once, getting $2e-K$ (observe that $\simplgraph$ does not have side nodes in any triangular face).
Therefore, we have
\begin{align}\label{eq:counting}
	3n=2e-K-3\ell.
\end{align}
Euler's equation on $\simplgraph$ gives $e = N + (n + \ell) -1$.
Substituting this into \eqref{eq:counting}, we get $2N = n + K + \ell + 2$.

\ref{lem:combinatorial_prop_ii} The number of linear dependencies $\ell$ is bounded above by $N-K$.
Therefore, using the last equation for $2N$,
\[
	2\ell\leq 2(N-K) = n-K+\ell+2.\qedhere
\]
\end{proof}

\subsection{Framed maps and constrained framed maps}\label{ssec:framed_maps}
 
A dissection with a given combinatorial type is characterized by the following requirements:
\begin{enumerate}
\item [(i)]
All
vertices that lie on an edge of some triangle or of $P$ are collinear.
\item [(ii)]
The corner nodes coincide with the corners of $P$.
\item [(iii)]
The triangles are properly oriented and nonoverlapping.
\end{enumerate}
We now define {framed maps} and constrained framed maps, in which
conditions (i) and (iii) or just condition (iii) is relaxed.
The key property we get from this generalization is that
the spaces of framed maps and of constrained framed maps have a simple
structure.
In Section~\ref{sec:area-discrep}, this will allow us to treat the minimization of the ``sum of squared
residuals'' for a combinatorial type of dissection as a polynomial minimization problem on a Euclidean space.

\begin{definition}
[Framed map and constrained framed map of the skeleton graph of a dissection]
Let $\graph$ be the skeleton graph of a dissection~$D$ of a simple polygon $P$, and $V(\graph)$ its set of nodes.
  \begin{enumerate}
  \item 
A \Dfn{framed map} is a map $\phi\colon V(\graph)\rightarrow \mathbb{R}^2$ 
that sends the corner nodes of $\graph$ to the corresponding corners of $P$.
\item 
A framed map $\phi$ is a \Dfn{constrained framed map} of $\graph$ if for every side of every face of $\graph$ (including the outside face), $\phi$ sends the side nodes and the two corners of that side to a line.
  \end{enumerate}
\end{definition}
Constrained framed maps for the special case of triangulations were
already considered in \cite[Sect.~3, p.~137]{abrams_spaces_2014} under
the name \emph{drawings}.

Clearly, for any dissection~$D$ of a simple polygon $P$, there is a corresponding constrained framed map $\phi_D$ of $\graphdis$.
The converse is false in general; not all constrained framed maps of the skeleton graph $\graphdis$ are obtained from a dissection, as described in the next example.

\begin{example}\label{ex:boundary_cond}
Consider the triangulation $T$ of the square shown on the left in Figure~\ref{fig:boundary_cond1}.
The side node $e$ can be moved towards the right or the left to obtain different framed maps of the skeleton graph of the triangulation which are not constrained framed maps.
To have a constrained framed map the node $e$ should be sent on the vertical line spanned by the corners $b$ and $c$; if it is not between them, the constrained framed map does not
represent a dissection.

\begin{figure}[!ht]
\resizebox{0.9\hsize}{!}{\begin{tabular}{c@{\hspace{0.5cm}}c@{\hspace{0.5cm}}c@{\hspace{0.5cm}}c}
\begin{tikzpicture}[scale=2.5,
edge/.style={color=blue!95!black,thick,line join=bevel,line cap=round},
vertex/.style={circle,inner sep=1pt,fill=blue}]

\coordinate (1) at (1.000000,0.000000);
\coordinate (2) at (1,0.500000);
\coordinate (3) at (1.00000,1.000000);
\coordinate (4) at (0.00000,1.00000);
\coordinate (5) at (0.000000,0.00000);

\draw[edge] (1) -- (2);
\draw[edge] (1) -- (5);
\draw[edge] (2) -- (3);
\draw[edge] (2) -- (4);
\draw[edge] (2) -- (5);
\draw[edge] (3) -- (4);
\draw[edge] (4) -- (5);

\node[vertex, label=below:$b$] at (1) {};
\node[vertex, label=right:$e$] at (2) {};
\node[vertex, label=above:$c$] at (3) {};
\node[vertex, label=above:$d$] at (4) {};
\node[vertex, label=below:$a$] at (5) {};

\end{tikzpicture}

&

\begin{tikzpicture}[scale=2.5,
edge/.style={color=blue!95!black,thick,line join=bevel,line cap=round},
vertex/.style={circle,inner sep=1pt,fill=blue}]

\coordinate (1) at (1.000000,0.000000);
\coordinate (2) at (1.2,0.500000);
\coordinate (3) at (1.00000,1.000000);
\coordinate (4) at (0.00000,1.00000);
\coordinate (5) at (0.000000,0.00000);

\draw[edge] (1) -- (2);
\draw[edge] (1) -- (5);
\draw[edge] (2) -- (3);
\draw[edge] (2) -- (4);
\draw[edge] (2) -- (5);
\draw[edge] (3) -- (4);
\draw[edge] (4) -- (5);

\fill[blue,fill opacity=0.2] (2) -- (1) -- (3) -- cycle;

\draw[edge,loosely dotted] (1) -- (3);

\node[vertex, label=below:$b$] at (1) {};
\node[vertex, label=right:$e$] at (2) {};
\node[vertex, label=above:$c$] at (3) {};
\node[vertex, label=above:$d$] at (4) {};
\node[vertex, label=below:$a$] at (5) {};

\end{tikzpicture}

&

\begin{tikzpicture}[scale=2.5,
edge/.style={color=blue!95!black,thick,line join=bevel,line cap=round},
vertex/.style={circle,inner sep=1pt,fill=blue}]

\coordinate (1) at (1.000000,0.000000);
\coordinate (2) at (0.8,0.500000);
\coordinate (3) at (1.00000,1.000000);
\coordinate (4) at (0.00000,1.00000);
\coordinate (5) at (0.000000,0.00000);

\draw[edge] (1) -- (2);
\draw[edge] (1) -- (5);
\draw[edge] (2) -- (3);
\draw[edge] (2) -- (4);
\draw[edge] (2) -- (5);
\draw[edge] (3) -- (4);
\draw[edge] (4) -- (5);

\fill[blue,fill opacity=0.2] (2) -- (1) -- (3) -- cycle;

\draw[edge,loosely dotted] (1) -- (3);

\node[vertex, label=below:$b$] at (1) {};
\node[vertex, label=left:$e$] at (2) {};
\node[vertex, label=above:$c$] at (3) {};
\node[vertex, label=above:$d$] at (4) {};
\node[vertex, label=below:$a$] at (5) {};

\end{tikzpicture}

&

\begin{tikzpicture}[scale=2.5,
edge/.style={color=blue!95!black,thick,line join=bevel,line cap=round},
vertex/.style={circle,inner sep=1pt,fill=blue}]

\coordinate (1) at (1.000000,0.000000);
\coordinate (2) at (1,1.200000);
\coordinate (3) at (1.00000,1.000000);
\coordinate (4) at (0.00000,1.00000);
\coordinate (5) at (0.000000,0.00000);

\draw[edge] (1) -- (2);
\draw[edge] (1) -- (5);
\draw[edge] (2) -- (3);
\draw[edge] (2) -- (4);
\draw[edge] (2) -- (5);
\draw[edge] (3) -- (4);
\draw[edge] (4) -- (5);

\node[vertex, label=below:$b$] at (1) {};
\node[vertex, label=right:$e$] at (2) {};
\node[vertex, label=right:$c$] at (3) {};
\node[vertex, label=above:$d$] at (4) {};
\node[vertex, label=below:$a$] at (5) {};

\end{tikzpicture}

\end{tabular}}
\caption{A triangulation $T$ of the square (left), two illustrations of framed maps of the graph $\graphtri$ that are not constrained framed maps with a dashed
edge belonging to $\simplgraph$ (middle), and a constrained framed map that {does not} describe a dissection (right).}\label{fig:boundary_cond1}
\end{figure}
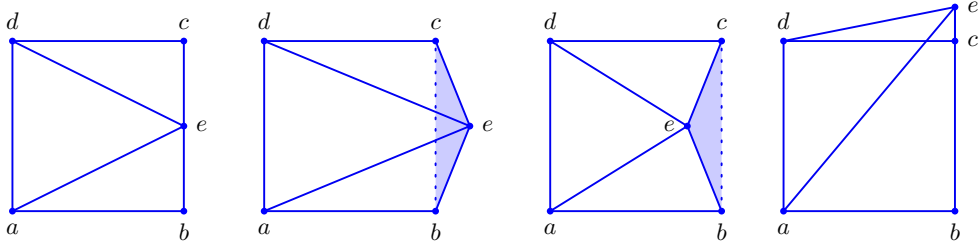
\end{example}

\subsection{Signed area of a triangular face}\label{ssec:signed_areas}

Given a dissection $D$, we define the signed areas of triangular faces of $\simplgraph$ with respect to framed maps of $\graph$.

\begin{definition}[Signed area of triangular faces with respect to a framed map]\label{def:signedarea}
	Let $D$ be a dissection, let $f$ be a triangular face of $\simplgraph$ 
	with corner nodes $c_1,c_2$, and $c_3$ labeled counterclockwise, and let
	$\phi$ be a framed map of $\graph$.
The \Dfn{signed area} $a_{\phi}(f)$ of $f$ with respect to $\phi$ is
\[
	a_{\phi}(f):=\frac{1}{2}\begin{vmatrix}
1 & 1 & 1 \\
x_{\phi(c_1)} & x_{\phi(c_2)} & x_{\phi(c_3)} \\
y_{\phi(c_1)} & y_{\phi(c_2)} & y_{\phi(c_3)} \\
\end{vmatrix},
\]
where $(x_{\phi(c_i)},y_{\phi(c_i)})$, with $1\leq i\leq 3$, are the coordinates of the corners nodes of $f$ given by~$\phi$.
\end{definition}

The signed area of a triangular face with respect to a framed map is a well-defined quantity even if the framed map does
not come from a dissection.

\begin{example}[Example \ref{ex:boundary_cond} continued]
  \label{slide-outside}
  When moving the node $e$ to the right
 in the triangulation~$T$ shown in Figure~\ref{fig:boundary_cond1}, the sum of the signed areas of
 the triangles becomes greater than the area of the square. 
 When moving $e$ to the left, it becomes smaller than the
 area of the square: the shaded area determined by the triangle
 ${bce}$ does not get added. 
 In the framed map shown on the right of Figure~\ref{fig:boundary_cond1}, the
 sum of the signed areas of triangles is equal to the area of the square.
\end{example}

The following lemma shows the invariance of the sum of the signed areas of triangular faces of~$\simplgraph$.

\begin{proposition}\label{prop:area_invariance}
Let $P$ be a simple polygon of area $E$, let $D=\{t_1,t_2,\dots,t_n\}$ be a dissection of~$P$, let~$L$ be a reduced system of collinearity constraints of $D$,
and let~$\phi$ be a framed map of the skeleton graph of $D$.
The sum of the signed areas of triangular faces of $\simplgraph$ equals $E$:
\[
	E=\sum_{i=1}^{n}a_\phi(t_i) + \sum_{\ell\in L} a_\phi(\ell).
\]
\end{proposition}

\begin{proof}
Let $s$ be a node of~$\graph$ which is not a corner of $P$ with $\phi(s)=(x_s,y_s)$.
In the simplicial graph~$\simplgraph$, the node $s$ is contained in at least three triangular faces which piece up together around~$s$.
Denote by $\tau_1=(s,w_0,w_1),
\tau_2=(s,w_1,w_2),\dots, \tau_k=(s,w_{k-1},w_0)$ the triangles of
$\simplgraph$ with $s$ as a corner in counterclockwise order.
Changing the entries of $\phi(s)$ affects only the signed area of
these triangles $\tau_1,\tau_2,\dots,\tau_k$.
Now compute the sum of the signed areas
\[
\sum_{i=1}^k a_{\phi}(\tau_i) =
\sum_{i=1}^k \frac{1}{2}
\begin{vmatrix} 
	1 & 1 & 1 \\ 
	x_s & x_{w_{i-1}} & x_{w_{i}} \\ 
	y_s & y_{w_{i-1}} & y_{w_{i}} 
\end{vmatrix},
\]
where $\phi(w_i)=(x_{w_i},y_{w_i})$. 
The ordering $(s, i-1, i)$ is the same as the one obtained in Definition~\ref{def:signedarea} of signed area.
Developing the determinants, and factoring the terms $x_s$ and $y_s$, we deduce that~$x_s$ and $y_s$ get multiplied by $0$.
Therefore the position of the node~$s$ does not influence the sum of signed areas.
Since this sum is equal to $E$ when $\phi=\phi_D$, the result follows. 
\end{proof}

\begin{corollary}\label{cor:area_invariance}
If $\phi$ is a constrained framed map of the skeleton graph of $D$, then
\[
	E = \sum_{i=1}^{n}a_{\phi}(t_i).
\]
\end{corollary}

\subsection{Monsky's theorem for constrained framed maps of skeleton graphs}\label{ssec:extension_monsky}

Monsky's original result provides more than the result for the square.
As Monsky already noted in~\cite{monsky_dividing_1970}, his result holds 
for all simple polygons of integral area with an odd number of sides of type red-blue.

The following result plays a key role in Section~\ref{sec:lower_bound} to prove the positivity of a polynomial measure of area differences for which we provide a lower bound.
It extends Monsky's result (Theorem~\ref{thm:monsky}) to constrained framed maps, which---as we have seen---are considerably more general than dissections. 
We get Monsky's original result when the simple polygon is the square with corners $(0,0)$ (colored blue), $(1,0)$ (colored red), $(1,1)$ (colored
red), and $(0,1)$ (colored green) and take constrained framed maps coming from dissections.

\begin{theorem}\label{thm:frame_monsky} 
Let $P$ be a simple polygon of integer area~$E$ and let $\phi$ be a constrained framed map of the skeleton graph~$\graph$ of a dissection $D$ of $P$ into an odd number $n$ of triangles.
If~$P$ has an odd number of red-blue sides, then there exists a triangular face of $\graph$ whose signed area with
respect to~$\phi$ is different from $E/n$.
\end{theorem}

\begin{proof}
The proof uses a parity argument analogous to the proof of Sperner's lemma.
We count the number of pairs
$(e,t)$, where $e$ is a red-blue edge of $\graph$ 
on the boundary of a triangular face~$t$ of $\graph$.
Internal edges of $\graph$ appear in two pairs, while boundary edges of $\graph$ appear in only one.
By Lemma~\ref{lem:property_coloring} and since $\phi$ is a constrained framed map, side nodes of $\graph$ lying on a side of $P$ of type
red-blue have to be red or blue with respect to $\phi$.
Because $P$ has an odd number of red-blue sides, we deduce that there is an odd number of red-blue boundary edges with respect to~$\phi$ in the skeleton graph~$\graph$.
Hence the number of above pairs is odd.

Again using the fact that $\phi$ is a constrained framed map, each colorful triangular face contributes an odd number of red-blue edges while any other
triangular face contributes an even number of red-blue edges.
This shows that the number of colorful triangular faces is congruent modulo 2 to the number of red-blue sides of $P$.
Since we assumed this number to be odd, $\graph$ has to contain a colorful triangular face with respect to $\phi$.
By Lemma~\ref{lem:property_coloring}, the corners of the colorful triangle cannot be collinear and as $E$ is an integer the signed area cannot be $\pm E/n$ with respect to $\phi$.
\end{proof}

\begin{remark}
The theorem falls back on a proof of Monsky's theorem.
However, it applies to the more general family of framed maps, which turns out to be essential to study how small the range of areas of dissections of the square can be. 
\end{remark}

If in a dissection all triangles have the same area, then the
unsigned area is $E/n$. However, constrained framed maps might contain
triangles of negative orientation, and therefore all triangles could
have the same unsigned area, different from $E/n$.
An example of this with an even
 number of triangles is given in
Example~\ref{strange} and shown in
 Figure~\ref{fig:ex_frame_emb2} below.
 The following corollary rules
out this possibility when $n$ is odd.

\begin{corollary}
If $P$ has an odd number of red-blue sides, then the triangular faces of $\graph$ cannot all have the same unsigned area with respect to $\phi$.
\end{corollary}

\begin{proof}
We prove it by contradiction.
Let $\alpha := |a_\phi(t)|>0$ be the common area of the triangular
faces $t$ of $\graph$. 
Suppose there are $b$ triangular faces with negative signed area.
By Corollary~\ref{cor:area_invariance},
\[
	E=\sum_{i=1}^{n}a_\phi(t_i)= \underbrace{\alpha + \dots + \alpha}_{n-b \text{ times}} - \underbrace{(\alpha + \dots + \alpha)}_{b\text{ times}}= (n-2b)\cdot \alpha.
\]
We get 
 $\alpha=\frac{E}{n-2b}$,
with integral $E$ and odd $n-2b$.
By Lemma~\ref{lem:property_coloring}, a colorful triangular face cannot have area $\pm \alpha$. 
On the other hand, there exists a colorful triangular face, and we have a contradiction.
\end{proof}

Figure~\ref{fig:newpoly} shows a simple polygon satisfying the condition of the theorem.

\begin{figure}[!ht]
\begin{tikzpicture}[scale=0.8,
edge/.style={draw=blue!95!black,thick,line join=bevel,line cap=round},
vertexred/.style={rectangle,inner sep=3pt,fill=red},
vertexgreen/.style={circle,inner sep=2pt,fill=green},
vertexblue/.style={diamond,inner sep=2pt,fill=blue},
]

\coordinate (1) at (0,0);
\coordinate (2) at (2,0);
\coordinate (3) at (3,3);
\coordinate (4) at (1.5,2.5);
\coordinate (5) at (2,5);
\coordinate (6) at (-2,4);
\coordinate (7) at (0,3);
\coordinate (8) at (-2,2);

\fill[fill=blue!20!white] (1) -- (2) -- (3) -- (4) -- (5) -- (6) -- (7) -- (8) -- cycle;
\draw[help lines,step=1] (-2,0) grid (3,5);
\draw[edge] (3) -- (4) -- (5) -- (6) -- (7) -- (8) -- (1) -- (2);
\draw[edge,line width=2pt,loosely dashed] (2) -- (3);

\node[label={[align=center]below:$(0,0)$}] at (1) {};
\node[label={[align=center]below:$(2,0)$}] at (2) {};
\node[label={[align=center]right:$(3,3)$}] at (3) {};
\node[label={[align=center]below:$(3/2,5/2)$}] at (4) {};
\node[label={[align=center]above:$(2,5)$}] at (5) {};
\node[label={[align=center]above:$(-2,4)$}] at (6) {};
\node[label={[align=center]left:$(0,3)$}] at (7) {};
\node[label={[align=center]left:$(-2,2)$}] at (8) {};

\node[vertexblue] at (1) {};
\node[vertexblue] at (2) {};
\node[vertexred] at (3) {};
\node[vertexred] at (4) {};
\node[vertexgreen] at (5) {};
\node[vertexblue] at (6) {};
\node[vertexgreen] at (7) {};
\node[vertexblue] at (8) {};

\end{tikzpicture}
\caption{A simple polygon $P$ where Monsky's coloring approach
  applies. It has area $13$ and one side of $P$ is of type red-blue,
 dashed and fattened with a red square
and a blue diamond as corners.}\label{fig:newpoly}
\end{figure}
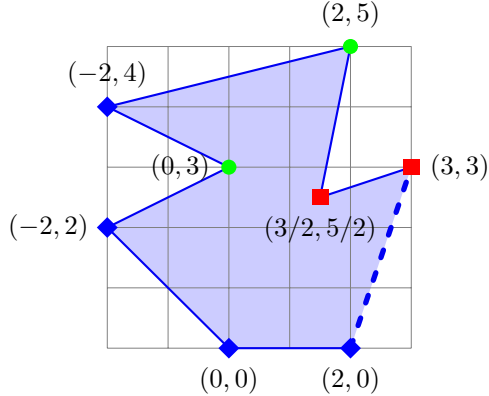

The following example emphasizes that the previous theorem concerns constrained framed maps of skeleton graphs of dissections and not the more general framed maps.

\begin{example}\label{strange}
It is possible to find a framed map of the skeleton graph of a dissection where all signed areas are equal to the average $E/n$, emphasizing that this is
possible for framed maps that are not constrained framed maps.
Consider the two framed maps shown in Figure~\ref{fig:ex_frame_emb}.
In the framed map shown on the right, the coordinates of the three internal nodes are $a=(\frac{2}{5},\sqrt{\frac{14}{25}})$,
$b=(\frac{3}{5},\frac{3}{5})$, and $c=(\sqrt{\frac{14}{25}},\frac{2}{5})$.
With these coordinates, the signed area of the five bounded faces are all equal to $1/5$, but the regions determined by the bounded faces \emph{are not} triangles anymore.
It is also possible to find constrained framed maps where all unsigned areas are equal as shown in Figure~\ref{fig:ex_frame_emb2}.
\end{example}

\begin{figure}[!ht]
\begin{tabular}{c@{\hspace{2cm}}c}

\begin{tikzpicture}[scale=2.5,
edge/.style={color=blue!95!black,thick,line join=bevel,line cap=round},
vertex/.style={circle,inner sep=1pt,fill=blue}]

\coordinate (1) at (0.000000,0.000000);
\coordinate (2) at (1.00000,0.000000);
\coordinate (3) at (1.00000,1.00000);
\coordinate (4) at (0.000000,1.00000);
\coordinate (5) at (0.333333,0.666667);
\coordinate (6) at (0.500000,0.500000);
\coordinate (7) at (0.666667,0.333333);

\draw[edge] (1) -- (2);
\draw[edge] (1) -- (4);
\draw[edge] (1) -- (5);
\draw[edge] (1) -- (7);
\draw[edge] (2) -- (3);
\draw[edge] (2) -- (7);
\draw[edge] (3) -- (4);
\draw[edge] (3) -- (6);
\draw[edge] (4) -- (5);
\draw[edge] (5) -- (6);
\draw[edge] (6) -- (7);

\node[vertex] at (1) {};
\node[vertex] at (2) {};
\node[vertex] at (3) {};
\node[vertex] at (4) {};
\node[vertex,label=above:$a$] at (5) {};
\node[vertex,label=below left:$b$] at (6) {};
\node[vertex,label=below:$c$] at (7) {};

\end{tikzpicture}

&

\begin{tikzpicture}[scale=2.5,
edge/.style={color=blue!95!black,thick,line join=bevel,line cap=round},
vertex/.style={circle,inner sep=1pt,fill=blue}]

\coordinate (1) at (0.000000,0.000000);
\coordinate (2) at (1.00000,0.000000);
\coordinate (3) at (1.00000,1.00000);
\coordinate (4) at (0.000000,1.00000);
\coordinate (5) at (0.4,0.748331477);
\coordinate (6) at (0.6,0.6);
\coordinate (7) at (0.748331477,0.4);

\draw[edge] (1) -- (2);
\draw[edge] (1) -- (4);
\draw[edge] (1) -- (5);
\draw[edge] (1) -- (7);
\draw[edge] (2) -- (3);
\draw[edge] (2) -- (7);
\draw[edge] (3) -- (4);
\draw[edge] (3) -- (6);
\draw[edge] (4) -- (5);
\draw[edge] (5) -- (6);
\draw[edge] (6) -- (7);

\node[vertex] at (1) {};
\node[vertex] at (2) {};
\node[vertex] at (3) {};
\node[vertex] at (4) {};
\node[vertex,label=above:$a$] at (5) {};
\node[vertex,label=below left:$b$] at (6) {};
\node[vertex,label=below:$c$] at (7) {};

\end{tikzpicture}

\end{tabular}
\caption{On the left, a dissection of the square into $5$ triangles.
On the right, a framed map of the skeleton graph of the same dissection which is not a constrained framed map.}\label{fig:ex_frame_emb}
\end{figure}
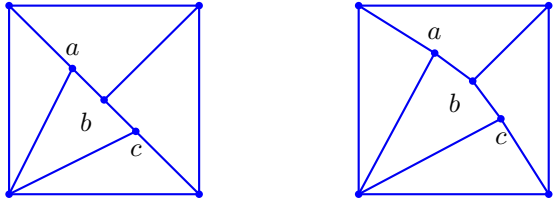

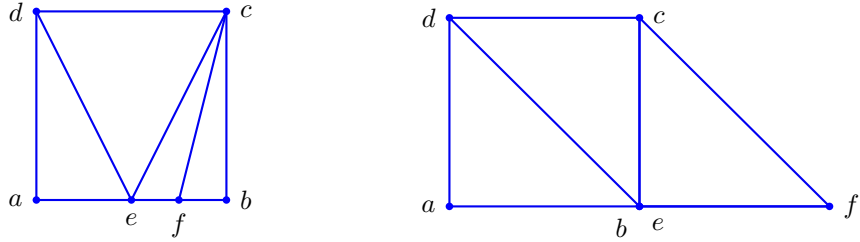
\begin{figure}[!ht]
\begin{tabular}{c@{\hspace{2cm}}c}

\begin{tikzpicture}[scale=2.5,
edge/.style={color=blue!95!black,thick,line join=bevel,line cap=round},
vertex/.style={circle,inner sep=1pt,fill=blue}]

\coordinate (1) at (0.000000,0.000000);
\coordinate (2) at (1.00000,0.000000);
\coordinate (3) at (1.00000,1.00000);
\coordinate (4) at (0.000000,1.00000);
\coordinate (5) at (0.500000,0);
\coordinate (6) at (0.7500000,0);

\draw[edge] (1) -- (4);
\draw[edge] (1) -- (5);
\draw[edge] (2) -- (3);
\draw[edge] (2) -- (6);
\draw[edge] (3) -- (4);
\draw[edge] (3) -- (5);
\draw[edge] (3) -- (6);
\draw[edge] (4) -- (5);
\draw[edge] (5) -- (6);

\node[vertex,label=left:$a$] at (1) {};
\node[vertex,label=right:$b$] at (2) {};
\node[vertex,label=right:$c$] at (3) {};
\node[vertex,label=left:$d$] at (4) {};
\node[vertex,label=below:$e$] at (5) {};
\node[vertex,label=below:$f$] at (6) {};

\end{tikzpicture}

&

\begin{tikzpicture}[scale=2.5,
edge/.style={color=blue!95!black,thick,line join=bevel,line cap=round},
vertex/.style={circle,inner sep=1pt,fill=blue}]

\coordinate (1) at (0.000000,0.000000);
\coordinate (2) at (1.00000,0.000000);
\coordinate (3) at (1.00000,1.00000);
\coordinate (4) at (0.000000,1.00000);
\coordinate (5) at (1,0);
\coordinate (6) at (2,0);

\draw[edge] (1) -- (4);
\draw[edge] (1) -- (5);
\draw[edge] (2) -- (3);
\draw[edge] (2) -- (6);
\draw[edge] (3) -- (4);
\draw[edge] (3) -- (5);
\draw[edge] (3) -- (6);
\draw[edge] (4) -- (5);
\draw[edge] (5) -- (6);

\node[vertex,label=left:$a$] at (1) {};
\node[vertex,label=below left:$b$] at (2) {};
\node[vertex,label=right:$c$] at (3) {};
\node[vertex,label=left:$d$] at (4) {};
\node[vertex,label=below right:$e$] at (5) {};
\node[vertex,label=right:$f$] at (6) {};

\end{tikzpicture}

\end{tabular}
\caption{On the left, a dissection of the square into $4$ triangles.
On the right, a constrained framed map of the skeleton graph of the same
dissection where all unsigned areas are equal to $1/2$.}\label{fig:ex_frame_emb2}
\end{figure}

\section{Area differences of dissections and framed maps}\label{sec:area-discrep}

In this section we set the stage to use a gap theorem to give a lower
bound the range of areas of dissections of the square.  Before
 introducing the area difference polynomial that we will use for this
 purpose, we want to point out that there is an alternative approach: 
Abrams and Pommersheim \cite{abrams_spaces_2014}
have recently shown
that the areas $a_1,\dots,a_n$ of a triangulation with a given combinatorial type
satisfy a non-trivial polynomial equation.
This opens up, in principle,
another way of obtaining a lower bound on the range of areas.
However, this polynomial typically has high degree
and seems hard to describe explicitly.

We consider the abstract dissection $\mathcal{D}=(\graph,B,C,\mathcal{T})$ arising from a dissection $D=\{t_1,t_2,\dots,t_n\}$  of a
dissection of a simple polygon $P$ of area $E$.
Let $X_{\mathcal{D}}:=\left((x_v,y_v)\right)_{v\in V(\graph)}$ be the plane coordinates for the nodes of $\graph$.
We consider $x_v$ and~$y_v$ as variables describing a framed map.
If $\graph$ has $M$ nodes,
$X_{\mathcal{D}}$ contains $2M$ variables. 

The area difference polynomial is a sum of three quadratic penalty terms:
The first term is the \emph{sum of squared residuals} of the signed areas of the
 triangular faces. It
is related to the RMS-error, but it avoids the square root and the
division by~$n$:
\begin{equation}
  \label{eq:SSR}
	\deltassr(X_\mathcal{D}):=\sum_{i\in[n]}\left(a(t_i)-\frac{E}{n}\right)^2
	= \RMS(X_\mathcal{D})^2\cdot n,
\end{equation}
with
\[
a(f) = \frac{1}{2}
\begin{vmatrix}
1 & 1 & 1 \\
x_{v_1} & x_{v_2} & x_{v_3} \\
y_{v_1} & y_{v_2} & y_{v_3}
\end{vmatrix},
\]
where $v_1,v_2$, and $v_3$ are the corner nodes of the triangular face $f$ of $\simplgraph$ ordered counterclockwise.

The second term takes care of collinearities
(condition~(i) from the
beginning of Section~\ref{ssec:framed_maps}).
If~$L$ is a reduced system of collinearity constraints of the dissection~$D$,
we denote by $\deltaL(X_\mathcal{D})$ the sum of squares of signed areas of
these constraints:
\[
	\deltaL(X_\mathcal{D}):=\sum_{\ell\in L}a(\ell)^2.
\]

Finally, we want the corners to lie on their assigned positions
(condition~(ii) from the
beginning of Section~\ref{ssec:framed_maps}).
Let $C$ be the set of corner nodes of the skeleton graph of $D$, and let $(p_v,q_v)$
for $v\in C$
 denote the
coordinates of the corners of~$P$.
We denote by $\deltaC(X_\mathcal{D})$ the sum of squared distances of the
corner nodes from their target positions:
\[
	\deltaC(X_\mathcal{D}):=
        \sum_{v\in C}\left((x_v-p_v)^2 + (y_v-q_v)^2\right).
\]

\begin{definition}[Area difference polynomial of an abstract dissection]
	Let $P$ be a simple polygon of area $E$ and $\mathcal{D}=(\graph,B,C,\mathcal{T})$ be an abstract dissection of $P$.
	The \Dfn{area difference polynomial} $\pi_\mathcal{D}\in\mathbb{R}\left[X_\mathcal{D}\right]$ of~$\mathcal{D}$ is the polynomial
	\begin{align*}
		\pi_\mathcal{D}(X_\mathcal{D})&=\deltassr(X_\mathcal{D})+\deltaL(X_\mathcal{D})+\deltaC(X_\mathcal{D})\\
	&=	\sum_{i\in[n]}\left(a(t_i)-\frac{E}{n}\right)^2 +
	        \sum_{\ell\in L}a(\ell)^2
	+
	        \sum_{v\in C}\left((x_v-p_v)^2 + (y_v-q_v)^2\right)
	,
	\end{align*}
	where $\{t_i~|~i\in[n]\}$ are the internal faces of $\graph$, $L$ is a reduced system of collinearity constraints, and $\{(p_v,q_v)~|~v\in C\}$ are the coordinates of the corners of $P$.
\end{definition}

The following lemma is an immediate consequence of this definition.  

\begin{lemma}\label{lem:discrep_poly}
Let $\mathcal{D}$ be an abstract dissection of a simple polygon $P$ of area $E$ with $n$ internal faces.
The area difference polynomial
${\pi}_\mathcal{D}(X_\mathcal{D})$ is always nonnegative, and it is zero if and only if 
$X_\mathcal{D}$ describes a constrained framed map and all signed areas of triangles of $\mathcal{D}$ are equal to $E/n$.\qed
\end{lemma}

\section{Lower bound for the range of areas of dissections}\label{sec:lower_bound}

In this section, we use a gap theorem from real algebraic geometry as a black box to obtain a lower bound on the range of areas of dissections.
First we obtain the necessary conditions in Section~\ref{ssec:prop_area_dev_poly} and then apply the theorem in Section~\ref{ssec:lower_bound}.

\subsection{Properties of the area difference polynomial}\label{ssec:prop_area_dev_poly}

\begin{proposition}\label{prop:discrep_properties}
  Let $P$ be a simple polygon of area $E$ and $D$ a dissection of~$P$
  into~$n$ triangles.  The area difference polynomial~$\pi_\mathcal{D}$ has the
  following properties.
\begin{enumerate}[label={\rm(\roman{enumi})},ref=(\roman{enumi})]
	\item It has degree $4$. \label{prop:discrep_properties_i}
	\item The number of variables is at most
$2n+4$. \label{prop:discrep_properties_ii}

	\item
If all corner coordinates are between $0$ and $b$, then
the constant term of $\pi_\mathcal{D}$ is bounded in absolute value by 
${E^2/n}+(2n+4) b^2$, and
 the remaining coefficients are
 bounded in
          absolute value by 
$\max\left\{1,{E/n},2b\right\}$.
 \label{cor:discrep_properties_iii}
	\item
If the area $E$ and all corner coordinates $(p_v,q_v)$ of $P$ are multiples of $1/s$ for
some integer~$s$, then
the polynomial $4ns^2\pi _\mathcal{D}$ is an integer polynomial. \label{cor:discrep_properties_iv}

\end{enumerate}
\end{proposition}

\begin{proof}
\ref{prop:discrep_properties_i} This is straightforward from the definition.

\ref{prop:discrep_properties_ii} There are two variables per node, and
by Lemma~\ref{lem:combinatorial_prop}, the number of nodes is at most $n+2$.

\ref{cor:discrep_properties_iii}
We first analyze the first two components $\deltassr+\deltaL$ of
	$\pi_\mathcal{D}$.  Expanding, grouping the terms
by degree, and denoting the faces of the simplicial graph $\simplgraph$ by $F(\simplgraph)$, we get
\begin{align*}\nonumber  
\deltassr(X_\mathcal{D})+\deltaL(X_\mathcal{D}) & = 
\underbrace{\sum_{f\in F(\simplgraph)} a(f)^2}_{S_4} - \underbrace{\frac{2E}{n}\sum_{i\in [n]}
a(t_i)}_{S_2}+\underbrace{\frac{E^2}{n}}_{S_0}
\end{align*}
We proceed to compute the coefficients in $S_4$ and $S_2$.
The monomials in the area formula
\[
a(f) = \tfrac{1}{2}
\bigl(
(x_{v_1}y_{v_2}-x_{v_2}y_{v_1}) +
(x_{v_2}y_{v_3}-x_{v_3}y_{v_2}) +
(x_{v_3}y_{v_1}-x_{v_1}y_{v_3})
\bigr)
\]
are grouped into three pairs, each corresponding to an edge of~$f$.
The term $a(f)^2$ has coefficients $\frac14,\pm\frac12$.
If an edge belongs to two triangles, the
 square of the corresponding term,
which has coefficients $\frac14$ and~$-\frac12$, will be taken twice,
contributing terms with coefficients $\frac12$ and $-1$.
All other terms appear only once.
Thus the coefficients in $S_4$ are in $\left\{\frac{1}{4},\pm\frac{1}{2},-1\right\}$.

In $S_2$, the monomials for an edge which is a side of two triangles of $D$ cancel
because they appear in opposite orientations.
The remaining terms appear once, and hence
the coefficients in $S_2$ are $\pm\frac{E}{n}$.
Since the terms of $S_4$ are of degree 4 and 
 the terms of $S_2$ are of degree 2, there is no interference between
the parts.
Thus the coefficients of
$\deltassr+\deltaL$ are in
$\left\{\frac{1}{4},\pm\frac{1}{2},-1,\pm\frac{E}{n}\right\}$.

We still have to add the terms in $\deltaC$ for the corner
coordinates.
They are of the form
$x_v^2-2x_vp_v + p_v^2 +
y_v^2-2y_vq_v + q_v^2$,
with $0\le p_v,q_v\le b$.
There is at most one constant term 
($p_v^2$ or~$q_v^2$)
per corner variable, of absolute
value at most $b^2$, and since there are at most $2n+4$ variables in
total, this establishes the bound
${E^2}/{n}+(2n+4) b^2$
 on the overall
constant term, including the constant term $S_0=E^2/n$ from the
first two parts.

The coefficients of the quadratic terms are 1, and the coefficients of
the linear terms are bounded by $2b$ in absolute value.
Since the degree-2 terms in
 $\deltaC$ are purely quadratic
 and
 the degree-2 terms in $S_2$ are mixed, there is no interference
 between the different subexpressions.
Overall, we get the claimed bound on the size of the coefficients.

 \ref{cor:discrep_properties_iv}
From the above calculations we see that
all coefficients are multiples of $1/4$ or of $1/(ns^2)$. Thus
multiplication by $4ns^2$ makes every coefficient integral.
\end{proof}

\subsection{Lower bound using a gap theorem}\label{ssec:lower_bound}

To derive
the lower bound
on $\pi_\mathcal{D}$, we use the following gap theorem.
The domain over which the polynomial is minimized is the
$k$-dimensional simplex $\Sigma_k:=\{x\in \mathbb{R}^k_{\geq 0}:
\sum_{i=i}^k x_i\leq 1\}$.

\begin{theorem}[{Emiris--Mourrain--Tsigaridas~\cite[Section~4]{emiris_dmm_2010}}]\label{thm:EMT_gap}
Let $f\in \mathbb{Z}[X_1,\dots,X_k]$ be a multivariate polynomial of
total degree $d$ which is positive on the $k$-simplex $\Sigma_k$ and has
coefficients bounded by $2^\tau$. The minimal value $m:=\min \{f(x):
x\in \Sigma_k\}$ of $f$ on $\Sigma_k$ is bounded from below by 
\[ 
{m} \geq {m_{\normalfont\texttt{DMM}}},
\]
where
\begin{equation}
\frac{1}{m_{\normalfont\texttt{DMM}}} = 
2^{d(d - 1)^{(k - 1)}\left(
(k^2 + 3k + 1)\log_2{d}
+
(k + 1) (d\log_2{k} + \tau)
+ 3k + d + 2  
\right)} \times 2^{(k^2 +
  k)\log_2{\sqrt{d}}}. 
\label{eq:EMT} 
\end{equation}
\end{theorem}
 The subscript \texttt{DMM} stands for Davenport--Mahler--Mignotte.
In the published version of \cite[formula~(22)]{emiris_dmm_2010}, a  
term $d(d-1)^{k-1}$ in the exponent of~\eqref{eq:EMT} was lost by
splitting the expression over two lines.
This was confirmed by the authors (personal communication);
the above theorem corrects the omission. 

We are now ready to deduce our main lower bound.

\begin{theorem}[Doubly exponential lower bound on range]\label{thm:double_exp_lower}
Let $P$ be a simple polygon of integer area $E$
 with integer corner coordinates and
 an odd number of red-blue sides.
The range of any dissection of $P$ into an odd number $n$ of triangles
 is bounded from below by 
\[
	\frac{1}{2^{O(9^{n}n^2)}}, 
\] 
where the constant implied by the $O$-notation depends on~$P$.
\end{theorem}

\begin{proof}
We apply a translation so that the coordinates of the corners of $P$
are nonnegative integers and bounded above by $Y$, for some constant
$Y\geq1$ that depends on $P$.

Consider a dissection $D$, and
let $k$ denote the number of variables of $\pi _\mathcal{D}$.
By Proposition~\ref{prop:discrep_properties}, $k\le X:=2n+4$.
To ensure that the minimum we are looking for lies in the simplex $\Sigma_k$,
we apply a second linear transformation, multiplying the coordinates
by $1/XY$.
We obtain a polygon $P'$ of area $E'=E/(XY)^2$ where the sum of the  node coordinates in any
dissection of $P'$ is at most $1$.

By Theorem~\ref{thm:frame_monsky}, there is no dissection $D$ of the
original polygon $P$ (before the translation) with all areas equal to
$E/n$.  It follows that the translated and scaled
polygon $P'$ also cannot have a dissection $D'$ with all areas equal
to $E'/n$.  By Lemma~\ref{lem:discrep_poly}, the area
difference polynomial $\pi _{\mathcal{D}'}$ is therefore positive.

To apply Theorem~\ref{thm:EMT_gap}, we need to make the coefficients
of the polynomial integral, and we need to know a bound on the size of
the coefficients.
With the help of
Proposition~\ref{prop:discrep_properties}\ref{cor:discrep_properties_iii},
it is easy to establish that the largest coefficient
of $\pi_{\mathcal{D}'}$ is~1:
The corners of the polygon $P'$ lie in a square of side length $b=1/X$,
and hence its area $E'$ is bounded by $1/X^2$.
Thus the constant term is at most
${E'^2/n}+(2n+4) b^2 \le 1/(X^2n)+X/X^2 <1/X+1/X<1$.
As for the other coefficients,
  the largest term in our bound
$\max\{1,\frac{E'}{n},2b\}$
 on these coefficients
is~$1$.

The area of $P'$ is $E/(XY)^2$, and its corner coordinates
are
multiples of
$1/(XY)$.
We can thus apply
Proposition~\ref{prop:discrep_properties}\ref{cor:discrep_properties_iv}
with $s=(XY)^2$ and
conclude that
$4ns^2\pi _{\mathcal{D}'}
=
4nX^4Y^{4}\pi _{\mathcal{D}'}$ 
is an integer polynomial.
Its
coefficients are bounded by
\[
	Q :=
4nX^4Y^{4}
        = O(nX^4)
        = O(n^5)
.
\]
We now apply Theorem~\ref{thm:EMT_gap} to the polynomial
$4nX^4Y^{4}\pi _{\mathcal{D}'}$, with $k\le X= 2n+4$ variables, degree
$d=4$, and coefficient bitsize
$\tau=\lceil\log_2 Q\rceil=O(\log n)$.
Substituting these data into \eqref{eq:EMT},
we obtain that
the minimum value $m$ of $4nX^4Y^{4}\pi _{\mathcal{D}'}$ on the
$k$-simplex satisfies
\begin{equation}
\label{d-exp-bound}
	\frac{1}{m} 
\leq 2^{4\cdot 3^{2n+4} O(n^2)}
\times 2^{O(n^2)}
\leq 2^{O(9^{n}n^2)}.
\end{equation}
The constant in the $O$-notation depends only on $P$ and not on the
dissection $D$.
To bound the minimum of $\pi _{\mathcal{D}'}$,
we have to divide $m$ by the factor $4n X^4Y^{4}$,
which was used to make the polynomial integral.
 \begin{align*}
\frac m{4n X^4Y^{4}}
= \min \{\,\pi_{\mathcal{D}'}(X_{\mathcal{D}'}) \,\}
 &=
\min\left\{ \,
\deltassr_{\mathcal{D}'}(X_{\mathcal{D}'})
 + \deltaL_{\mathcal{D}'}(X_{\mathcal{D}'}) + \deltaC_{\mathcal{D}'}(X_{\mathcal{D}'})
\,\right\}\\
&\le
\min\left\{ \,
\deltassr_{\mathcal{D}'}(X_{\mathcal{D}'}) \mid
 \deltaL_{\mathcal{D}'}(X_{\mathcal{D}'})=\deltaC_{\mathcal{D}'}(X_{\mathcal{D}'})=0
\,\right\}   
 \end{align*}
From the last expression and \eqref{d-exp-bound}
we conclude that, for any dissection of $P'$,
the sum of squared residuals
$\deltassr_{\mathcal{D}'}$ is at least
$1/2^{O(9^{n}n^2)}$.
(The polynomial factor $4n X^4Y^{4}=O(n^5)$ 
is swallowed
by the $O$-notation in the exponent.)

The range $\range$ is related to
the
sum of squared residuals
 $\deltassr$
 by taking the square root and a multiplicative factor which is at
least $1/\sqrt n$
(see equation~\eqref{eq:SSR} and Proposition~\ref{prop:lower_bound}).
 These operations do not change the
doubly-exponential character of the lower bound
$1/2^{O(9^{n}n^2)}$. 

Finally, we have to translate the result back to the original
polygon $P$. The area range is multiplied by $(XY)^2$
to compensate the scaling of~$P$. Again, this polynomial factor does not
influence the bound. This concludes the proof of the theorem.
\end{proof}

Since the unit square satisfies the assumptions of the theorem (cf.\
the beginning of Section~\ref{ssec:extension_monsky}) we have
established the lower-bound part of our main
result~\eqref{eq:main-reformed} as presented in the introduction.

\section{Enumeration and optimization results}\label{sec:experiment}

We have computed the best dissections of a square with respect to the \RMS~error, 
for small numbers of triangles.  For this purpose,
we enumerated all combinatorial types of dissections of the unit
square with a given number of nodes, and we 
minimized the \RMS\ area deviation for each type.

Below we describe our computational approach and report the
results.  Due to the combinatorial
explosion of the number of cases and the algebraic difficulty of solving each case,
we could only treat dissections with up to 8 nodes before we
reached the limit of computing power.
Our calculations complement earlier attempts of
Mansow~\cite{mansow_ungerade_2003}, who had considered only
triangulations, and optimized the range~\range~of the areas.

In Section~\ref{sec:upper_bounds}, we will report further
computational experiments on
dissections and triangulations with special
structure, which allowed us to treat larger numbers of triangles.
 
\subsection{Enumeration of combinatorial types}

To generate the combinatorial types of dissections of the
unit square, we used a combination of \plantri~\cite{plantri} and \Sage~\cite{sage}.
The software \plantri\ efficiently enumerates planar graphs with
prescribed properties.  
We used it to generate all $3$-connected planar graphs on $N+1$ nodes.  
For each graph, we choose one vertex to be ``at infinity'', and after
discarding it, we use its neighbors as boundary nodes.
Among the boundary nodes, we select four to be the corner nodes; the remaining boundary nodes get assigned to the sides of the square.
For each interior face of the graph,
we choose three nodes to be the corners of that triangular face.
There are many combinations of choices that do not lead to a valid combinatorial type of a dissection of a square, 
and these are discarded. Here are a few easy-to-state necessary conditions that we used (some others are more intricate):
\begin{itemize}
 \item A boundary node in $B$ cannot be a side node of a triangular face of $\graph$.
 \item An internal node cannot be in a collinearity constraint with
   two boundary nodes which lie on the same side of $P$.
 \item An internal node can be a side node of at most one
   triangular face of $\graph$.
 \item A series of collinearity constraints forces successive edges on
   a line segment (thus fulfilling their role; see for example nodes $1,2,3,4, 5$ in
   Figure~\ref{fig:does-not-enforce-collinearity}c). 
   It can happen that 
   (parts of) two such line segments are connected in such a way that
   they enclose some triangles between them. 
   Such a combinatorial type can be discarded. 
\end{itemize}
Furthermore, since $P$ is a square, we reduce number of
abstract dissections considered by using the symmetries of $P$.
 
\subsection{Finding the optimal dissection for each combinatorial type}

Once the combinatorial type is fixed, we can 
write down the area difference polynomial. 
We are interested in the minimum of
the sum of squared residuals
$\deltassr(X_\mathcal{D})$ under the side constraints
 $\deltaL(X_\mathcal{D})=\deltaC(X_\mathcal{D})=0$.
We take care of the framing constraint $\deltaC(X_\mathcal{D})=0$ by directly
substituting the desired corner coordinates into the polynomial
$\deltassr(X_\mathcal{D})$,
resulting in a polynomial
       $\hat{\pi}_\mathcal{D}(X'_\mathcal{D})$ with a reduced set of variables $X'_\mathcal{D}$.
 We then incorporate the constraint
 $\deltaL(X_\mathcal{D}')
=0$
 with
 a Lagrange multiplier $\gamma$ and
get the integer polynomial
\[
       \hat{\pi}_\mathcal{D}(X_\mathcal{D}',\gamma)
:=\hat{\pi}_\mathcal{D}(X_\mathcal{D}')
 +4\gamma\deltaL(X_\mathcal{D}').
\]
Then we set up a system of polynomial equations by setting
the gradient of $\hat{\pi}_\mathcal{D}(X_\mathcal{D}',\gamma)$ to~$0$.
This gives all critical points of
$\hat{\pi}_\mathcal{D}(X_\mathcal{D}',\gamma)$, including
 the configurations that represent legal dissections and minimize $\deltassr$.

To find all real solutions to the system, we use \Bertini~\cite{bertini},
a program that uses homotopy continuation to find numerical solutions of systems of polynomial equations. 
According to \Bertini's user manual \cite{bertini}, \Bertini\ finds all 
isolated solutions; nevertheless, this highly depends on the tolerance parameters and 
the dimension of the solution set. On the one hand, if the solution set to the system of polynomial 
equations is zero-dimensional, then one could opt to use Groebner bases to solve the system 
of polynomial equations. However already for $7$ nodes, 
computing the Groebner bases in the zero-dimensional cases to get all
solutions was hopeless on a large scale.
On the other hand, many combinatorial types had a solution sets of positive dimension.
Hence we do not claim that the solutions we found are optimal.

\subsection{Minimal area deviation for dissections with at most 8 nodes} 

The  process of generating combinatorial types of dissections with up to 8
nodes and computing
 coordinates with smallest $\deltassr$-deviation for each
of them was parallelized on 36 processors (i5 CPU@2.80GHz) and took 3~days.

In Table~\ref{tab:diss_discr}, we present the results for triangulations and dissections (that are not triangulations) of the square with $3,5,7$, and $9$ triangles.
We used the sum of squared residuals
$\deltassr$ in the computations, because it is a polynomial,
but the tables report the $\RMS$ numbers, because they are on the same
scale with the area range $\range$.

\begin{table}[htp]
  \centering
  \begin{tabular}{|l|@{\ \ }ll|@{\ \ }ll||l|}
\hline
&\multicolumn{4}{c||}{RMS-optimal dissections}
&\hfil\cite{mansow_ungerade_2003}\\
\hline
&\multicolumn{2}{c|@{\ \ }}{triangulations}
&\multicolumn{2}{c||}{dissections$^\dag$}
&\hfil triang.%
\\
&\hfil RMS&\hfil$\range$
&\hfil RMS&\hfil$\range$
&\hfil
$\range$
\\\hline
\hz3 triangles, \hz5 nodes
&  \llap{$^*$}$0.117\,851$ & 0.25
&  \llap{$^*$}$0.117\,851$ & 0.25
&0.25
\\\hline
\hz5 triangles, \hz6 nodes
& \llap{$^*$}0.010\,281\,9 & 0.026\,446\,6
& 0.040\,824\,8 & 0.083\,333\,3
& 0.0225
\\
\hz5 triangles, \hz7 nodes
& 0.040\,824\,8 & 0.083\,333\,3  & \llap{$^*$}0.010\,281\,9 & 0.026\,446\,6
& 0.0833
\\\hline
\hz7 triangles, \hz7 nodes
& 0.001\,301\,4 
& 0.004\,008\,1 
& 0.005\,134\,9 &0.012\,787\,9
&0.0031
\\
\hz7 triangles, \hz8 nodes
& 0.003\,284\,9 & 0.010\,214\,9&\llap{$^*$}0.000\,805\,1 & 
0.002\,320\,7
& 0.0077
\\
\hz7 triangles, \hz9 nodes
&--&--&--&--
& 0.0417
\\\hline
\hz9 triangles, \hz8 nodes
& 0.000\,395\,6
&\llap{$^\triangle$}%
0.001\,147\,9
& \llap{$^*$}0.000\,279\,1 &\llap{$^\triangle$}%
0.000\,961\,6
& 0.0011
\\
\hz9 triangles, \hz9 nodes
&--&--&--&--
& 0.0001408
\\
\hz9 triangles, 10 nodes
&--&--&--&--
& 0.0016
\\
\hz9 triangles, 11 nodes
&--&--&--&--
& 0.025
\\\hline
11 triangles, \hz9 nodes
&--&--&--&--
&0.000\,322\,2
\\
11 triangles, 10 nodes
&--&--&--&-- 
&0.000\,004\,2
\\
11 triangles, 11 nodes
&--&--&--&--
&0.000\,056\,9
\\
11 triangles, 12 nodes
&--&--&--&--
&0.000\,297\,6
\\
11 triangles, 13 nodes
&--&--&--&--
&0.016\,7
\\\hline
  \end{tabular}
\vspace{4mm}

  \caption{Triangulations and dissections of the square with at most
    $8$ nodes with the optimal \RMS~values. 
	The last column shows the results obtained by
	Mansow~\cite{mansow_ungerade_2003} for optimizing the range~\range~among triangulations,
	and for comparison, we include in the two center columns the best
	    ranges~$R$ that we found during our computations.
$^*$The best solutions that we found for a given number of triangles are
marked with a star. 
\\
$^\dag$The column for dissections includes only those dissections that are
not triangulations. 
\\
$^{\triangle}$For
$9$ triangles and $8$ nodes,
the combinatorial type that gave the smallest range was different from
the combinatorial type that gave the smallest \RMS~error. In the other rows, 
the adjacent columns~\RMS\ and~\range~refer to the same dissection.
}
  \label{tab:diss_discr} 
\end{table}

The \RMS-optimal dissections with $3,5,7$, and $9$ triangles
and
with at most $8$ nodes are shown 
in Figures~\ref{fig:opt_diss3}--\ref{fig:opt_diss9}.
By Lemma~\ref{lem:combinatorial_prop},
the number of nodes for a given number~$n$ of triangles can be as large as $n+2$. 
Thus the results for 7 and 9 triangles are not complete.

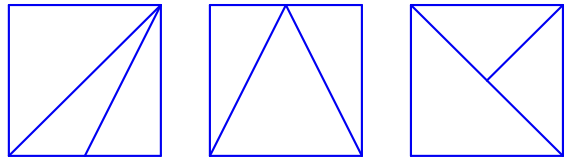
\begin{figure}[htp]
	\centering
	\begin{tabular}{c@{\hspace{0.5cm}}c@{\hspace{0.5cm}}c}
		\begin{tikzpicture}[scale=2,
edge/.style={color=blue!95!black,thick,line join=bevel,line cap=round}]

\coordinate (1) at (0.000000,0.000000);
\coordinate (2) at (0.500000,0.000000);
\coordinate (3) at (1.00000,0.000000);
\coordinate (4) at (1.00000,1.00000);
\coordinate (5) at (0.000000,1.00000);

\draw[edge] (1) -- (2);
\draw[edge] (1) -- (4);
\draw[edge] (1) -- (5);
\draw[edge] (2) -- (3);
\draw[edge] (2) -- (4);
\draw[edge] (3) -- (4);
\draw[edge] (4) -- (5);

\end{tikzpicture} & \begin{tikzpicture}[scale=2,
edge/.style={color=blue!95!black,thick,line join=bevel,line cap=round}]

\coordinate (1) at (0.000000,0.000000);
\coordinate (2) at (1.00000,0.000000);
\coordinate (3) at (1.00000,1.00000);
\coordinate (4) at (0.500000,1.00000);
\coordinate (5) at (0.000000,1.00000);

\draw[edge] (1) -- (2);
\draw[edge] (1) -- (4);
\draw[edge] (1) -- (5);
\draw[edge] (2) -- (3);
\draw[edge] (2) -- (4);
\draw[edge] (3) -- (4);
\draw[edge] (4) -- (5);


\end{tikzpicture} & \begin{tikzpicture}[scale=2,
edge/.style={color=blue!95!black,thick,line join=bevel,line cap=round}]

\coordinate (1) at (1.00000,1.00000);
\coordinate (2) at (0.000000,1.00000);
\coordinate (3) at (0.500000,0.500000);
\coordinate (4) at (1.00000,0.000000);
\coordinate (5) at (0.000000,0.000000);

\draw[edge] (1) -- (2);
\draw[edge] (1) -- (3);
\draw[edge] (1) -- (4);
\draw[edge] (2) -- (3);
\draw[edge] (2) -- (5);
\draw[edge] (3) -- (4);
\draw[edge] (4) -- (5);


\end{tikzpicture} \\
	\end{tabular}
	\caption{The \RMS-optimal dissections of the square
          with $3$ triangles
          ($\RMS=0.117851$).}\label{fig:opt_diss3}
\end{figure}

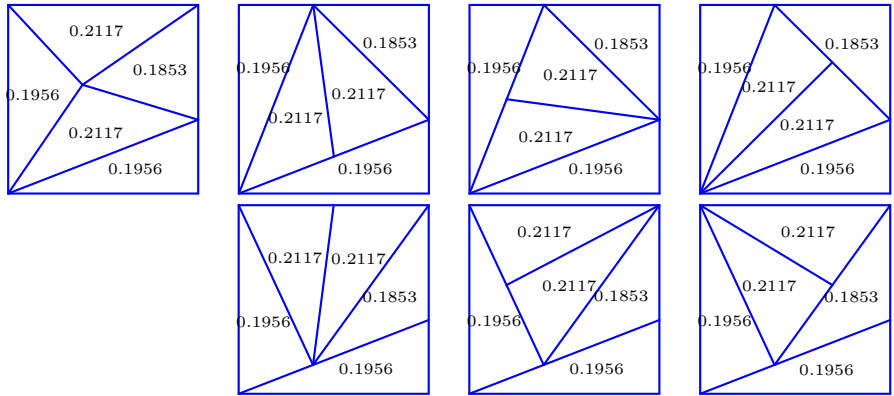
\begin{figure}[htp]
	\centering
	\begin{tabular}{c@{\hspace{0.25cm}}c@{\hspace{0.25cm}}c@{\hspace{0.25cm}}c}
	\begin{tikzpicture}[scale=2.5,
edge/.style={color=blue!95!black,thick,line join=bevel,line cap=round}]

\coordinate (1) at (0.000000,1.00000);
\coordinate (2) at (1.00000,1.00000);
\coordinate (3) at (0.391260,0.576542);
\coordinate (4) at (0.000000,0.000000);
\coordinate (5) at (1.00000,0.391260);
\coordinate (6) at (1.00000,0.000000);

\draw[edge] (1) -- (2);
\draw[edge] (1) -- (3);
\draw[edge] (1) -- (4);
\draw[edge] (2) -- (3);
\draw[edge] (2) -- (5);
\draw[edge] (3) -- (4);
\draw[edge] (3) -- (5);
\draw[edge] (4) -- (5);
\draw[edge] (4) -- (6);
\draw[edge] (5) -- (6);

\node at ($(1)!.5!(2)!0.33!(3)$) {\tiny 0.2117};
\node at ($(4)!.5!(1)!0.33!(3)$) {\tiny 0.1956};
\node at ($(4)!.5!(3)!0.33!(5)$) {\tiny 0.2117};
\node at ($(2)!.5!(5)!0.33!(3)$) {\tiny 0.1853};
\node at ($(5)!.5!(6)!0.33!(4)$) {\tiny 0.1956};

\end{tikzpicture} & \begin{tikzpicture}[scale=2.5,
edge/.style={color=blue!95!black,thick,line join=bevel,line cap=round}]

\coordinate (1) at (1.00000,0.39126);
\coordinate (2) at (0.500000,0.19563);
\coordinate (3) at (1.00000,0.00000);
\coordinate (4) at (1.000000,1.00000);
\coordinate (5) at (0.391260,1.000000);
\coordinate (6) at (0.00000,0.000000);
\coordinate (7) at (0.000000,1.000000);

\draw[edge] (1) -- (2);
\draw[edge] (1) -- (3);
\draw[edge] (1) -- (4);
\draw[edge] (1) -- (5);
\draw[edge] (2) -- (5);
\draw[edge] (2) -- (6);
\draw[edge] (3) -- (6);
\draw[edge] (4) -- (5);
\draw[edge] (5) -- (6);
\draw[edge] (5) -- (7);
\draw[edge] (6) -- (7);

\node at ($(3)!.5!(6)!0.33!(1)$) {\tiny 0.1956};
\node at ($(5)!.5!(1)!0.33!(2)$) {\tiny 0.2117};
\node at ($(5)!.5!(4)!0.33!(1)$) {\tiny 0.1853};
\node at ($(6)!.5!(7)!0.33!(5)$) {\tiny 0.1956};
\node at ($(5)!.5!(2)!0.33!(6)$) {\tiny 0.2117};

\end{tikzpicture} &  \begin{tikzpicture}[scale=2.5,
edge/.style={color=blue!95!black,thick,line join=bevel,line cap=round}]

\coordinate (1) at (0.39126,1.00000);
\coordinate (2) at (0.19563,0.500000);
\coordinate (3) at (0.00000,1.00000);
\coordinate (4) at (1.000000,1.00000);
\coordinate (5) at (1.000000,0.391260);
\coordinate (6) at (0.00000,0.000000);
\coordinate (7) at (1.000000,0.000000);

\draw[edge] (1) -- (2);
\draw[edge] (1) -- (3);
\draw[edge] (1) -- (4);
\draw[edge] (1) -- (5);
\draw[edge] (2) -- (5);
\draw[edge] (2) -- (6);
\draw[edge] (3) -- (6);
\draw[edge] (4) -- (5);
\draw[edge] (5) -- (6);
\draw[edge] (5) -- (7);
\draw[edge] (6) -- (7);

\node at ($(3)!.5!(6)!0.33!(1)$) {\tiny 0.1956};
\node at ($(5)!.5!(1)!0.33!(2)$) {\tiny 0.2117};
\node at ($(5)!.5!(4)!0.33!(1)$) {\tiny 0.1853};
\node at ($(6)!.5!(7)!0.33!(5)$) {\tiny 0.1956};
\node at ($(5)!.5!(2)!0.33!(6)$) {\tiny 0.2117};

\end{tikzpicture} & \begin{tikzpicture}[scale=2.5,
edge/.style={color=blue!95!black,thick,line join=bevel,line cap=round}]

\coordinate (1) at (0.391260,1.00000);
\coordinate (2) at (0.695630,0.695630);
\coordinate (3) at (1.00000,1.00000);
\coordinate (4) at (0.000000,1.00000);
\coordinate (5) at (0.000000,0.000000);
\coordinate (6) at (1.00000,0.391260);
\coordinate (7) at (1.00000,0.000000);

\draw[edge] (1) -- (2);
\draw[edge] (1) -- (3);
\draw[edge] (1) -- (4);
\draw[edge] (1) -- (5);
\draw[edge] (2) -- (5);
\draw[edge] (2) -- (6);
\draw[edge] (3) -- (6);
\draw[edge] (4) -- (5);
\draw[edge] (5) -- (6);
\draw[edge] (5) -- (7);
\draw[edge] (6) -- (7);

\node at ($(3)!.5!(6)!0.33!(1)$) {\tiny 0.1853};
\node at ($(5)!.5!(1)!0.33!(2)$) {\tiny 0.2117};
\node at ($(5)!.5!(4)!0.33!(1)$) {\tiny 0.1956};
\node at ($(6)!.5!(7)!0.33!(5)$) {\tiny 0.1956};
\node at ($(5)!.5!(2)!0.33!(6)$) {\tiny 0.2117};

\end{tikzpicture} \\
	 & \begin{tikzpicture}[scale=2.5,
edge/.style={color=blue!95!black,thick,line join=bevel,line cap=round}]

\coordinate (1) at (0.000000,1.000000);
\coordinate (2) at (0.500000,1.000000);
\coordinate (3) at (0.00000,0.000000);
\coordinate (4) at (0.391260,0.153084);
\coordinate (5) at (1.000000,1.00000);
\coordinate (6) at (1.00000,0.00000);
\coordinate (7) at (1.00000,0.39126);

\draw[edge] (1) -- (2);
\draw[edge] (1) -- (3);
\draw[edge] (1) -- (4);
\draw[edge] (2) -- (4);
\draw[edge] (2) -- (5);
\draw[edge] (3) -- (4);
\draw[edge] (3) -- (6);
\draw[edge] (4) -- (5);
\draw[edge] (4) -- (7);
\draw[edge] (5) -- (7);
\draw[edge] (6) -- (7);

\node at ($(4)!.5!(3)!0.33!(1)$) {\tiny 0.1956};
\node at ($(5)!.5!(4)!0.33!(2)$) {\tiny 0.2117};
\node at ($(7)!.5!(4)!0.33!(5)$) {\tiny 0.1853};
\node at ($(7)!.5!(6)!0.33!(3)$) {\tiny 0.1956};
\node at ($(1)!.5!(2)!0.33!(4)$) {\tiny 0.2117};

\end{tikzpicture} & \begin{tikzpicture}[scale=2.5,
edge/.style={color=blue!95!black,thick,line join=bevel,line cap=round}]

\coordinate (1) at (0.00000,0.000000);
\coordinate (2) at (0.00000,1.00000);
\coordinate (3) at (0.39126,0.153084);
\coordinate (4) at (1.000000,0.000000);
\coordinate (5) at (1.000000,0.391260);
\coordinate (6) at (1.000000,1.00000);
\coordinate (7) at (0.19563,0.576542);

\draw[edge] (1) -- (2);
\draw[edge] (1) -- (3);
\draw[edge] (1) -- (4);
\draw[edge] (2) -- (6);
\draw[edge] (2) -- (7);
\draw[edge] (3) -- (5);
\draw[edge] (3) -- (6);
\draw[edge] (3) -- (7);
\draw[edge] (4) -- (5);
\draw[edge] (5) -- (6);
\draw[edge] (6) -- (7);

\node at ($(7)!.5!(3)!0.33!(6)$) {\tiny 0.2117};
\node at ($(2)!.5!(7)!0.33!(6)$) {\tiny 0.2117};
\node at ($(5)!.5!(6)!0.33!(3)$) {\tiny 0.1853};
\node at ($(3)!.5!(2)!0.33!(1)$) {\tiny 0.1956};
\node at ($(4)!.5!(5)!0.33!(1)$) {\tiny 0.1956};

\end{tikzpicture} & \begin{tikzpicture}[scale=2.5,
edge/.style={color=blue!95!black,thick,line join=bevel,line cap=round}]

\coordinate (1) at (1.00000,0.391260);
\coordinate (2) at (1.00000,1.00000);
\coordinate (3) at (0.391260,0.153084);
\coordinate (4) at (1.00000,0.000000);
\coordinate (5) at (0.000000,0.000000);
\coordinate (6) at (0.000000,1.00000);
\coordinate (7) at (0.695630,0.576542);

\draw[edge] (1) -- (2);
\draw[edge] (1) -- (3);
\draw[edge] (1) -- (4);
\draw[edge] (2) -- (6);
\draw[edge] (2) -- (7);
\draw[edge] (3) -- (5);
\draw[edge] (3) -- (6);
\draw[edge] (3) -- (7);
\draw[edge] (4) -- (5);
\draw[edge] (5) -- (6);
\draw[edge] (6) -- (7);

\node at ($(7)!.5!(3)!0.33!(6)$) {\tiny 0.2117};
\node at ($(2)!.5!(7)!0.33!(6)$) {\tiny 0.2117};
\node at ($(5)!.5!(6)!0.33!(3)$) {\tiny 0.1956};
\node at ($(3)!.5!(2)!0.33!(1)$) {\tiny 0.1853};
\node at ($(4)!.5!(5)!0.33!(1)$) {\tiny 0.1956};

\end{tikzpicture} \\
\end{tabular}

\caption{The \RMS-optimal dissections of the square with $5$
  triangles ($\RMS=0.0102819$).  They all have the same
  multiset of areas.  We see that the optimum is achieved both by a
  triangulation (upper left) and by dissections which are not
  triangulations. The dissections in the same column are related by
  keeping the bottom triangle, reflecting the remaining trapezoid
  horizontally and shearing it.
The three dissections in the bottom row are obtained by halving a
triangle in three possible ways, and the same is true for the
corresponding dissections in the top row.
  (In addition, the second and third
  dissection in the top row are mirror images, and only one
of the two dissections was actually produced by the program.)  }\label{fig:opt_diss5}
\end{figure}
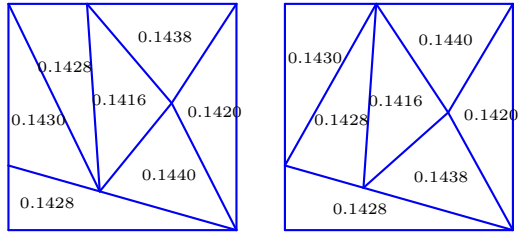
\begin{figure}[htp]
	\centering
	\begin{tabular}{c@{\hspace{0.25cm}}c}
	\begin{tikzpicture}[scale=3,
edge/.style={color=blue!95!black,thick,line join=bevel,line cap=round}]

\coordinate (1) at (0.344515,1.00000);
\coordinate (2) at (1.000000,1.00000);

\coordinate (3) at (0.715936,1-0.438628);
\coordinate (4) at (0.400474,1-0.828746);

\coordinate (5) at (0.000000,1.000000);
\coordinate (6) at (1.00000,0.00000);
\coordinate (7) at (0.00000,1-0.714351);
\coordinate (8) at (0.00000,0.000000);

\draw[edge] (1) -- (2);
\draw[edge] (1) -- (3);
\draw[edge] (1) -- (4);
\draw[edge] (1) -- (5);
\draw[edge] (2) -- (3);
\draw[edge] (2) -- (6);
\draw[edge] (3) -- (4);
\draw[edge] (3) -- (6);
\draw[edge] (4) -- (5);
\draw[edge] (4) -- (6);
\draw[edge] (4) -- (7);
\draw[edge] (5) -- (7);
\draw[edge] (6) -- (8);
\draw[edge] (7) -- (8);

\node at ($(4)!.5!(1)!0.33!(3)$) {\tiny 0.1416};
\node at ($(7)!.5!(8)!0.17!(6)$) {\tiny 0.1428};
\node at ($(4)!.5!(3)!0.33!(6)$) {\tiny 0.1440};
\node at ($(7)!.5!(5)!0.33!(4)$) {\tiny 0.1430};
\node at ($(1)!.5!(2)!0.33!(3)$) {\tiny 0.1438};
\node at ($(4)!.5!(5)!0.33!(1)$) {\tiny 0.1428};
\node at ($(3)!.5!(2)!0.33!(6)$) {\tiny 0.1420};

\end{tikzpicture} & \begin{tikzpicture}[scale=3,
edge/.style={color=blue!95!black,thick,line join=bevel,line cap=round}]

\coordinate (1) at (1.000000,0.000000);
\coordinate (2) at (1.00000,1.000000);
\coordinate (3) at (1-0.284064,0.519770);
\coordinate (4) at (1-0.655485,0.187239);
\coordinate (5) at (0.000000,0.00000);
\coordinate (6) at (0.00000,0.285649);
\coordinate (7) at (0.00000,1.00000);
\coordinate (8) at (1-0.599526,1.00000);

\draw[edge] (1) -- (2);
\draw[edge] (1) -- (3);
\draw[edge] (1) -- (4);
\draw[edge] (1) -- (5);
\draw[edge] (2) -- (3);
\draw[edge] (2) -- (8);
\draw[edge] (3) -- (4);
\draw[edge] (3) -- (8);
\draw[edge] (4) -- (6);
\draw[edge] (4) -- (8);
\draw[edge] (5) -- (6);
\draw[edge] (6) -- (7);
\draw[edge] (6) -- (8);
\draw[edge] (7) -- (8);

\node at ($(4)!.5!(3)!0.33!(1)$) {\tiny 0.1438};
\node at ($(8)!.5!(2)!0.33!(3)$) {\tiny 0.1440};
\node at ($(5)!.5!(6)!0.33!(1)$) {\tiny 0.1428};
\node at ($(4)!.5!(8)!0.33!(3)$) {\tiny 0.1416};
\node at ($(3)!.5!(2)!0.33!(1)$) {\tiny 0.1420};
\node at ($(6)!.5!(7)!0.33!(8)$) {\tiny 0.1430};
\node at ($(4)!.5!(6)!0.33!(8)$) {\tiny 0.1428};

\end{tikzpicture} \\
	\end{tabular}

	\caption{The \RMS-optimal dissections of the square with $7$
          triangles and at most $8$ vertices ($\RMS=0.0008051$). 
		  The two dissections are related by keeping the bottom triangle, 
		  reflecting the remaining trapezoid horizontally, and
                  then shearing it.
}\label{fig:opt_diss7}
\end{figure}

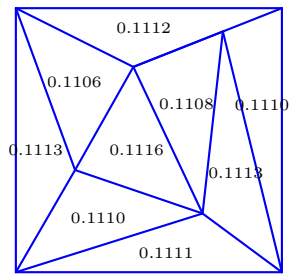
\begin{figure}[htp]
	\centering
	\begin{tikzpicture}[scale=3.5,
edge/.style={color=blue!95!black,thick,line join=bevel,line cap=round}]

\coordinate (1) at (0.000000,0.000000);
\coordinate (2) at (0.701785,0.222284);
\coordinate (3) at (1.00000,0.000000);
\coordinate (4) at (0.000000,1.00000);
\coordinate (5) at (0.222559,0.386865);
\coordinate (6) at (0.441653,0.777546);
\coordinate (7) at (0.778053,0.911573);
\coordinate (8) at (1.00000,1.00000);

\draw[edge] (1) -- (2);
\draw[edge] (1) -- (3);
\draw[edge] (1) -- (4);
\draw[edge] (1) -- (5);
\draw[edge] (2) -- (3);
\draw[edge] (2) -- (5);
\draw[edge] (2) -- (6);
\draw[edge] (2) -- (7);
\draw[edge] (3) -- (7);
\draw[edge] (3) -- (8);
\draw[edge] (4) -- (5);
\draw[edge] (4) -- (6);
\draw[edge] (4) -- (8);
\draw[edge] (5) -- (6);
\draw[edge] (6) -- (7);
\draw[edge] (6) -- (8);

\node at ($(3)!.5!(7)!0.33!(8)$) {\tiny 0.1110};
\node at ($(5)!.5!(4)!0.33!(6)$) {\tiny 0.1106};
\node at ($(1)!.5!(2)!0.33!(3)$) {\tiny 0.1111};
\node at ($(4)!.5!(5)!0.33!(1)$) {\tiny 0.1113};
\node at ($(4)!.5!(8)!0.33!(6)$) {\tiny 0.1112};
\node at ($(5)!.5!(2)!0.33!(1)$) {\tiny 0.1110};
\node at ($(5)!.5!(6)!0.33!(2)$) {\tiny 0.1116};
\node at ($(3)!.5!(2)!0.33!(7)$) {\tiny 0.1113};
\node at ($(7)!.5!(2)!0.33!(6)$) {\tiny 0.1108};

\end{tikzpicture}

	\caption{The \RMS-optimal dissection of the square
          with $9$ triangles and at most $8$ vertices
          ($\RMS=0.0002791$).}\label{fig:opt_diss9}
\end{figure}

We compare these results to some results from the diploma
thesis of
 Mansow~\cite{mansow_ungerade_2003}.
 Mansow generated all combinatorial types of 
triangulations of the square with up to 11 triangles,
using the program \plantri~\cite{plantri}. 
 For each type, she
set up the ``minimax'' problem 
for the difference between the largest and 
smallest triangle area when the nodes are restricted to the 
square. She used \emph{Matlab}'s \emph{Optimization Tool}
to search for the optimum from some starting value. 

For comparison with Mansow's results, our table reports also the 
smallest \emph{ranges} that we found
during our computations. Note that these are the ranges of the
RMS-optimal dissections (for each combinatorial type) and not
the range-optimal dissections, and obviously,
 different objective functions can lead to different results.
For example, Mansow found a triangulation with $6$ nodes with
a smaller range of $0.0225425$, 
compared to the range $\range=0.0264466$ that we found. 

Starting with $7$ triangles and up to $8$ nodes, 
dissections achieve smaller area deviation than triangulations,
both in terms of \RMS~error ($0.002 320 7$ versus
$0.004 008 1$) and
in terms of the range:
The best range of a dissection that we found
($0.002 320 7$, which is not even optimized) beats the best
triangulation (with $\range=0.0031$), which was found by Mansow.
(The comparisons regarding the \RMS~error are not conclusive, since
triangulations and dissections with 9 nodes are not included.)

\section{Upper bounds for the area range of dissections of the square}\label{sec:upper_bounds}

We extended our search for good dissections to larger numbers of
triangles, without trying to be exhaustive. The dissections that we
found suggested a pattern, which we describe and analyze in
Section~\ref{ssec:n9}. A more careful analysis leads to a family of dissections
with a superpolynomial decrease of the area range presented in Section~\ref{superpolynomial}.
In Section~\ref{ssec:experiments}, we compare, using experimental data, 
the area range of this family to a class of similar dissections.
In Section~\ref{ssec:triangulations}, we provide a class of triangulations 
that we suspect to have an exponential decrease of the area range.
In Section~\ref{ssec:heuristic} we provide a heuristic argument 
in favor of an exponential decrease of the area range.
Finally, in Section~\ref{sec:tarry}, we discuss the relation between 
minimizing the range of areas and the Tarry--Escott Problem.

\subsection{Monotonicity of the area deviation}
Before we look at special constructions, we mention an
 observation due to
Thomas \cite[Thm.~1]{thomas_dissection_1968},
which shows how we can easily go from a dissection
into $n$ triangles to $n+2$ triangles.
As $n$ increases, we can trivially achieve at
least an inverse linear improvement in the area deviation:
  
\begin{lemma}\label{add2}
Let $D$ be a dissection of the unit square into $n$ triangles.
Then there exists a dissection $D'$ of the unit square into $n+2$
triangles
with
$\range(D')=\frac{n}{n+2}\range(D)$ and
$\RMS(D')=\bigl(\frac{n}{n+2}\bigr)^{3/2}\RMS(D)$.
\end{lemma}
\begin{proof}
We can add two triangles of area $\frac{1}{n}$ on one side of the square to get a rectangle of area $\frac{n+2}{n}$, as in Figure~\ref{fig:add2}.
Scaling the rectangle to a square to get the dissection $D'$, the areas get multiplied by $\frac{n}{n+2}$.
Hence the range of areas in $D'$ gets multiplied by $\frac{n}{n+2}$.
The~\RMS\ formula is affected in a similar way: The two new triangles
add $0$ to the sum of squared differences, and the rescaling 
multiplies the \RMS\ error by $\bigl(\frac{n}{n+2}\bigr)^{3/2}
\RMS(D)$.
\end{proof}
\begin{figure}[!ht]
  \centering
  \includegraphics{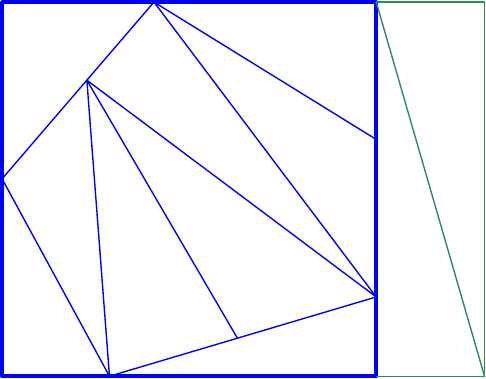}
  \caption{Adding two triangles with the desired area}
  \label{fig:add2}
\end{figure}

\subsection{A family of dissections with area range $O({1}/{n^5})$}\label{ssec:n9}

As a warm-up for the next section, we present a result of independent interest giving a polynomial upper bound on the area range.

\begin{theorem}\label{thm:n9}
Let $n\equiv 1 \pmod{4}$, and let $D_n$ be the dissection of the unit square 
into $n$ triangles
shown in {\rm Figure~\ref{fig:family_9}}, consisting of a right triangle on
top with area $\frac{1}{n}$ and $\frac{n-1}{4}$ trapezoidal slices
divided into $4$ triangles each.
The nodes of $D_n$ can be placed such that the range of areas satisfies $\range(D_n)\le O({1}/{n^5})$.
\end{theorem}

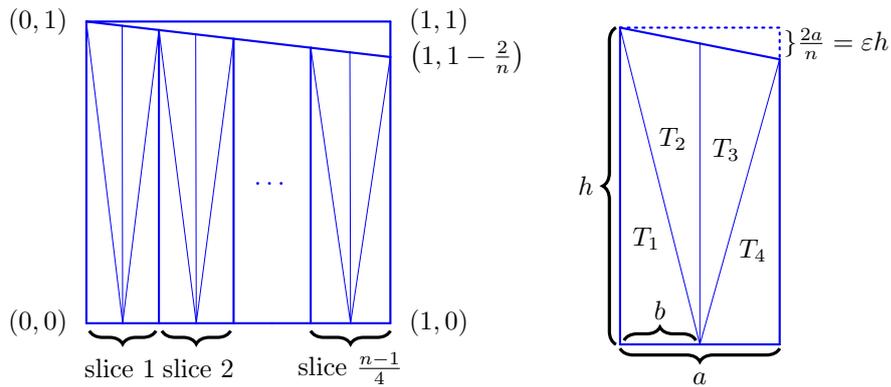
\begin{figure}[!ht]
\begin{tabular}{cc}
\begin{tikzpicture}[scale=4,
edge_slab/.style={color=blue!95!black,very thin,line join=bevel,line cap=round},
edge/.style={color=blue!95!black,thick,line join=bevel,line cap=round},
vertex/.style={circle,inner sep=1pt,fill=blue}]

\def\stepone{1}

\coordinate (1) at (1.00000,1.00000);
\coordinate (2) at (0.000000,1.00000);
\coordinate (3) at (0.000000,0.000000);
\coordinate (4) at (0.117621,0.986162);
\coordinate (5) at (0.119322,0.000000);
\coordinate (6) at (0.238640,0.971925);
\coordinate (7) at (0.238644,0.000000);
\coordinate (8) at (0.359659,0.957687);
\coordinate (9) at (0.361519,0.000000);
\coordinate (10) at (0.484386,0.943013);
\coordinate (11) at (0.484393,0.000000);
\coordinate (12) at (0.609114,0.928340);
\coordinate (13) at (0.611156,0.000000);
\coordinate (14) at (0.737911,0.913187);
\coordinate (15) at (0.737918,0.000000);
\coordinate (16) at (0.866708,0.898034);
\coordinate (17) at (0.868959,0.000000);
\coordinate (18) at (1.00000,0.882353);
\coordinate (19) at (1.00000,0.000000);

\draw[edge] (1) -- (2);
\draw[edge] (1) -- (18);
\draw[edge] (2) -- (3);
\draw[edge] (2) -- (4);
\draw[edge_slab] (2) -- (5);
\draw[edge] (3) -- (5);
\draw[edge_slab] (4) -- (5);
\draw[edge] (4) -- (6);
\draw[edge_slab] (5) -- (6);
\draw[edge] (5) -- (7);
\draw[edge] (6) -- (7);
\draw[edge] (6) -- (8);
\draw[edge_slab] (6) -- (9);
\draw[edge] (7) -- (9);
\draw[edge_slab] (8) -- (9);
\draw[edge] (8) -- (10);
\draw[edge_slab] (9) -- (10);
\draw[edge] (9) -- (11);
\draw[edge] (10) -- (11);
\draw[edge] (10) -- (12);
\draw[edge] (11) -- (13);
\draw[edge] (12) -- (14);
\draw[edge] (13) -- (15);
\draw[edge] (14) -- (15);
\draw[edge] (14) -- (16);
\draw[edge_slab] (14) -- (17);
\draw[edge] (15) -- (17);
\draw[edge_slab] (16) -- (17);
\draw[edge] (16) -- (18);
\draw[edge_slab] (17) -- (18);
\draw[edge] (17) -- (19);
\draw[edge] (18) -- (19);

\node[edge] at ($(12)!.5!(13)$) {$\dots$}; 

\draw[very thick,decorate,decoration={brace,amplitude=5pt},xshift=0cm,yshift=-0.025cm] (0.228644,0) -- (0.01,0) node [black,midway,yshift=-0.5cm] {slice 1\,};
\draw[very thick,decorate,decoration={brace,amplitude=5pt},xshift=0cm,yshift=-0.025cm] (0.474393,0.000000) -- (0.248644,0) node [black,midway,yshift=-0.5cm] {\,slice 2};
\draw[very thick,decorate,decoration={brace,amplitude=5pt},xshift=0cm,yshift=-0.025cm] (1,0.000000) -- (0.737918,0.000000) node [black,midway,yshift=-0.5cm] {slice $\frac{n-1}{4}$};

\node[label=right:{$\left(1,1-\frac{2}{n}\right)$}] at (18) {};

\node[label=left:{$\left(0,0\right)$}] at (3) {};
\node[label=left:{$\left(0,1\right)$}] at (2) {};
\node[label=right:{$\left(1,1\right)$}] at (1) {};
\node[label=right:{$\left(1,0\right)$}] at (19) {};

%
%

\end{tikzpicture} & \begin{tikzpicture}[scale=2.1,
edge_slab/.style={color=blue!95!black,very thin,line join=bevel,line cap=round},
edge/.style={color=blue!95!black,thick,line join=bevel,line cap=round},
vertex/.style={circle,inner sep=1pt,fill=blue}]

\def\stepone{1}

\coordinate (2) at (0.000000,2.00000);
\coordinate (3) at (0.000000,0.000000);
\coordinate (4) at (0.5,1.9);
\coordinate (5) at (0.5,0.000000);
\coordinate (6) at (1,1.8);
\coordinate (7) at (1,0.000000);

\draw[edge] (2) -- (3);
\draw[edge] (2) -- (4);
\draw[edge_slab] (2) -- (5);
\draw[edge] (3) -- (5);
\draw[edge_slab] (4) -- (5);
\draw[edge] (4) -- (6);
\draw[edge_slab] (5) -- (6);
\draw[edge] (5) -- (7);
\draw[edge] (6) -- (7);

\draw[edge,dotted] (0,2) -- (1,2) -- (1,1.8);


\draw[very thin,white,
xshift=0cm,xshift=0.025cm] (1,2) -- (1,1.8) node
[black,midway,xshift=0cm] {\rlap{$\}\frac{2a}{n}=\eps h$}};
\draw[very thick,decorate,decoration={brace,amplitude=5pt},xshift=0cm,yshift=-0.025cm] (1,0) -- (0,0) node [black,midway,yshift=-0.4cm] {$a$};
\draw[very thick,decorate,decoration={brace,amplitude=5pt},xshift=0cm,yshift=0.025cm] (0.025,0) -- (0.475,0) node [black,midway,yshift=0.4cm] {$b$};
\draw[very thick,decorate,decoration={brace,amplitude=5pt},xshift=0cm,xshift=-0.025cm] (0,0) -- (0,2) node [black,midway,xshift=-0.4cm] {$h$};

\node at ($(3)!.5!(2)!0.33!(5)$) {$T_1$};
\node at ($(2)!.5!(5)!0.33!(4)$) {$T_2$};
\node at ($(5)!.5!(4)!0.33!(6)$) {$T_3$};
\node at ($(5)!.5!(7)!0.33!(6)$) {$T_4$};

\end{tikzpicture}
\end{tabular}
\caption{On the left, a dissection of the unit square with $\frac{n-1}{4}$ rectangle trapezoidal slices and a top triangle of area $\frac{1}{n}$. On the right, the dimensions describing a slice.}\label{fig:family_9}
\end{figure}

\begin{proof} 
 We restrict the area of each slice to be exactly $\frac{4}{n}$.
This determines for each slice
the height~$h$ of the longer vertical side
 and the length $a$ of the horizontal base,
see Figure~\ref{fig:family_9}.
Both the longer vertical height $h$
and the shorter vertical height $h'$ lie between $1$ and $1-2/n$.
From the area formula $a(h+h')/2 = 4/n$ we derive
$a=4/n\cdot(1+O(\frac1n))
=O(\frac1n)$.

In each slice, we position the node on the horizontal base
such a way that the triangle $T_1$ has area $1/n$.
(This is not the best choice, but it simplifies the computations.
The optimal choice would improve the error only by a factor of
about~2.)
This determines the area of $T_4$. The triangles $T_2$ and $T_3$
can then share the remaining area equally by adjusting the edge
between them.
The base $b$ of the triangle $T_1$ is computed as follows.
Since the upper edge of the slice has slope $-2/n$, 
the shorter vertical side of the slice has height $h'=h(1-\eps)$, 
with $\eps:=\frac{2a}{nh}=O(1/n^2)$.
Thus the area of the slice is $ah(1-\eps/2)=4/n$.
Comparing this
with the area of the triangle,
$bh/2=1/n$, we deduce that
$b=\frac{a}{2}(1-\frac\eps2)$.
The areas~$A_i$ of the triangles~$T_i$ are now
\begin{align*}
 A_1 &= \tfrac{hb}{2}=\tfrac{1}{n}, \\
 A_4 &
= \tfrac{1}{2}
h(1-\eps)
(a-b)
\\
&= \tfrac{1}{2}
h(1-\eps)\tfrac a2 (1+\eps/2)
\\
&= \tfrac{1}{2}
h(1-\eps)\frac b{1-\eps/2} (1+\eps/2)
\\
&= \tfrac{hb}{2}
(1-\eps)({1-\eps/2})^{-1} (1+\eps/2)
\\
&= \tfrac{1}{n}
(1-\eps)\left(1+\eps/2+O(\eps^2)\right)(1+\eps/2)
\\
&= \tfrac{1}{n}
(1+O(\eps^2))
= \tfrac{1}{n}
(1+O(1/n^4))
=\tfrac{1}{n}+O(\tfrac{1}{n^5}), \\
 A_2 = A_3 &= \tfrac12 \left(\tfrac{4}{n} - A_1 - A_4 \right) =  \tfrac{1}{n} +
 O(\tfrac{1}{n^5}).
\qedhere
\end{align*}
\end{proof}

\begin{remark}
The sum of squared residuals  $\deltassr$ for this family of
dissections
is $n\times O( 1/{n^5})^2= O(1/n^9)$.
We optimized $\deltassr$ via the function \texttt{minimize} of the \texttt{python} library \texttt{scipy} with a tolerance of $1\times 10^{-19}$ 
and the method ``L-BFGS-B'' for $n$ up to $57$. 
The optimal values were approximately equal to $(0.53n-0.25)^{-9}$ using a least-square approximation, 
which suggests that the above construction is very close to the optimal representative of the combinatorial type.
\end{remark}

\subsection{A family of dissections with superpolynomially small area range}
\label{superpolynomial}
The previous construction can be improved using the Thue--Morse sequence.

\begin{definition}\label{def:thuemorse_basic}
The \Dfn{Thue--Morse sequence}
\[
\newcommand\nospaces[1]{\ifx #1\cdots \let\next=#1 \else
  {#1}\let\next=\nospaces\fi \next}
\newcommand\singledot{\raise1ex\hbox to 1pt
{\hss.\hss}}
s_1s_2s_3\ldots =
\nospaces+\singledot-\singledot-+\singledot-++-\singledot
-++-+--+\singledot-++-+--++--+-++-\singledot\cdots
\]
is defined recursively by $s_1=+1$ and
  \begin{align}
    s_{2j-1} &:= + s_{j}, \label{odd-case} \\
    s_{2j}   &:= - s_{j}, \label{even-case}
  \end{align}
for all $j\geq1$. 
\end{definition}
Classically, the Thue--Morse sequence is defined as a sequence of 0's
and 1's \cite[Sect.~2.2]{lothaire_combinatorics_1997}, but for
our purposes the values $\pm1$ (recorded as $+$ and $-$) are more
convenient.
We have inserted punctuation at the powers of 2 to highlight the
recursive structure.
A direct characterization of $s_i$ can be obtained from the binary
representation of $i-1$: $s_i=+1$ if and only if $i-1$ has an even number of 1's
in its binary
representation. 

The Thue--Morse sequence annihilates powers in the following sense:

\begin{lemma}
\label{lem:annihilate}
Let $k\ge 0$, $b\ne0$, and
let $f(x)$ be a polynomial of degree~$d$. 
If $d\ge k$, then there is a polynomial $F(x)$ of degree $d-k$ such that the following
identity holds for all~$x_0$\textup:
\begin{equation*}
 \sum _{i=1}^{2^k} s_i f(x_0+ib) = F(x_0).
\end{equation*}
Otherwise, if $d<k$, the above sum is zero.
\end{lemma}

The last claim was stated already by Prouhet in
1851~\cite{prouhet}; see also Section~\ref{sec:tarry}.
 For completeness, we give the easy proof by induction on $k$:
\begin{proof} 
For $k=0$, we can take $F(x)=f(x)$.
If $k>0$, we group the sum into pairs and use
\thetag{\ref{odd-case}--\ref{even-case}} from
 Definition~\ref{def:thuemorse_basic}:
\begin{align*}
 \sum _{i=1}^{2^k} s_i f(x_0+ib) 
& = 
 \sum _{j=1}^{2^k/2} \bigl[s_{2j-1} f(x_0+(2j-1)b) 
+ s_{2j} f(x_0+2jb) \bigr]
\\  
& = 
 \sum _{j=1}^{2^k/2} s_{j} \bigl[f(x_0+(2j-1)b) - f(x_0+2jb) \bigr] 
= 
 \sum _{j=1}^{2^{k-1}} s_{j} \hat f(x_0+j\hat b), 
\end{align*}
with $\hat b := 2b$ and $\hat f(x):=f(x-b)-f(x)$.
The polynomial $\hat f$ is identically zero if $d=0$. 
If $d>0$, $\hat f$~has degree $d-1$.
The parameter $k$ is also reduced by 1, and the induction goes through.
\end{proof}

We can now describe the main construction of our paper:

\begin{theorem}\label{thm:superpoly}
Let $n$ be odd, and let $D_n$ be the dissection of the square
with corners $(0,0)$, $(1,0)$, $(1,1)$ and $(0,1)$ with $n$ triangles
shown in {\rm Figure~\ref{fig:family_superpoly}}. 
The nodes of $D_n$ can be placed such that the range of areas 
is bounded from above by 
\begin{gather*}
  \range(D_n)\le 
\frac{8n^4}{n^{\log_2 n }\log_2n }
(1+O(\tfrac {\log n}{n}))
= \frac{1}{n^{\log_2 n -O(1)} }. 
\end{gather*} 
\end{theorem}
  
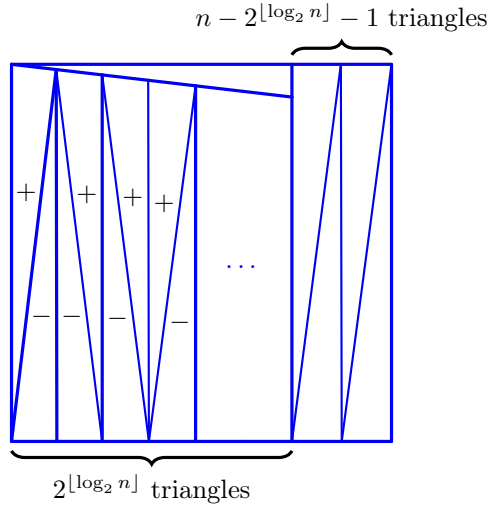
\begin{figure}[!ht]
\begin{tikzpicture}[scale=5,
edge/.style={color=blue!95!black,thick,line join=bevel,line cap=round},
thickedge/.style={edge,very thick}]

\coordinate (1) at (1.00000,1.00000);
\coordinate (2) at (0.000000,1.00000);
\coordinate (3) at (0.000000,0.000000);
\coordinate (4) at (0.117621,0.986162);
\coordinate (5) at (0.119322,0.000000);
\coordinate (6) at (0.238640,0.971925);
\coordinate (7) at (0.238644,0.000000);
\coordinate (8) at (0.359659,0.957687);
\coordinate (9) at (0.361519,0.000000);
\coordinate (10) at (0.484386,0.943013);
\coordinate (11) at (0.484393,0.000000);
\coordinate (12) at (0.609114,0.928340);
\coordinate (13) at (0.611156,0.000000);
\coordinate (14) at (0.737911,0.913187);
\coordinate (14s) at (0.737911,1);
\coordinate (15) at (0.737918,0.000000);
\coordinate (16) at (0.866708,1);
\coordinate (17) at (0.868959,0.000000);
\coordinate (18) at (1.00000,0.882353);
\coordinate (19) at (1.00000,0.000000);

\draw[thickedge] (1) -- (2);
\draw[thickedge] (1) -- (18);
\draw[thickedge] (2) -- (3);
\draw[thickedge] (2) -- (4);
\draw[thickedge] (3) -- (5);
\draw[thickedge] (5) -- (7);
\draw[thickedge] (7) -- (9);
\draw[thickedge] (9) -- (11);
\draw[thickedge] (11) -- (13);
\draw[thickedge] (13) -- (15);
\draw[thickedge] (15) -- (17);
\draw[thickedge] (17) -- (19);
\draw[thickedge] (18) -- (19);
\draw[thickedge] (4) -- (6);
\draw[thickedge] (6) -- (8);
\draw[thickedge] (8) -- (10);
\draw[thickedge] (10) -- (12);
\draw[thickedge] (12) -- (14);
\draw[thickedge] (14) -- (14s);

\draw[thickedge] (3) -- (4);
\draw[thickedge] (4) -- (5);
\draw[edge] (4) -- (7);
\draw[thickedge] (6) -- (7);

\draw[edge] (6) -- (9);
\draw[edge] (8) -- (9);
\draw[edge] (9) -- (10);
\draw[thickedge] (10) -- (11);

\draw[thickedge] (14) -- (15);
\draw[edge] (15) -- (16);
\draw[edge] (16) -- (17);
\draw[edge] (17) -- (1);   

\draw[very thick,decorate,decoration={brace,amplitude=5pt},xshift=0cm,yshift=-0.025cm] (0.737911,0) -- (0,0) node [black,midway,yshift=-0.5cm] {$2^{\lfloor \log_2 n \rfloor}$ triangles};
\draw[very thick,decorate,decoration={brace,amplitude=5pt},xshift=0cm,yshift=0.025cm] (0.737918,1.00000) -- (1,1.000000) node [black,midway,yshift=0.5cm] {$n-2^{\lfloor \log_2 n \rfloor}-1$ triangles};

\node[edge] at ($(12)!.5!(13)$) {$\dots$}; 



\node at ($(8)!.5!(9)!0.33!(10)$) {$+$}; 
\node at ($(6)!.5!(8)!0.33!(9)$) {$+$}; 
\node at ($(4)!.5!(6)!0.33!(7)$) {$+$}; 
\node at ($(3)!.5!(2)!0.33!(4)$) {$+$}; 

\node at ($(3)!.5!(4)!0.33!(5)$) {$-$}; 
\node at ($(4)!.5!(7)!0.33!(5)$) {$-$}; 
\node at ($(6)!.5!(7)!0.33!(9)$) {$-$}; 
\node at ($(9)!.5!(10)!0.33!(11)$) {$-$}; 


\end{tikzpicture} 
\caption{A dissection of the square into $n$ triangles. The $(+,-)$-sequence is determined by the first $2^{\lfloor\log_2 n\rfloor}$ terms in the Thue--Morse sign-sequence.}\label{fig:family_superpoly}
\end{figure}

\begin{proof}
First consider the case $n=2^k+1$.
 As in Figure~\ref{fig:family_9}, we start by cutting a flat triangle
of area $\frac{1}{n}$ from the top edge,
see Figure~\ref{fig:general}.
From the trapezoid $PQRS$ that remains, we cut $n-1$ triangles
of areas $a_1,a_2,\ldots,a_{n-1}$ from left to right.
Of course, the areas must have the correct sum:
\begin{equation*}
  a_1+a_2+\cdots+a_{n-1} = \frac{n-1}{n}.
\end{equation*}
For each triangle, we choose an orientation $\tau_i$:
It either has a side on the bottom side $PQ$ ($\tau_i=-1$) or on the top side $RS$ ($\tau_i=+1$).
The family of dissections in Theorem~\ref{thm:n9} corresponds to the periodic orientation sequence
$-1,+1,+1,-1,\,\allowbreak -1,+1,+1,-1,\allowbreak\ldots$.

\begin{figure}[!hbtp]
  \centering
  \includegraphics{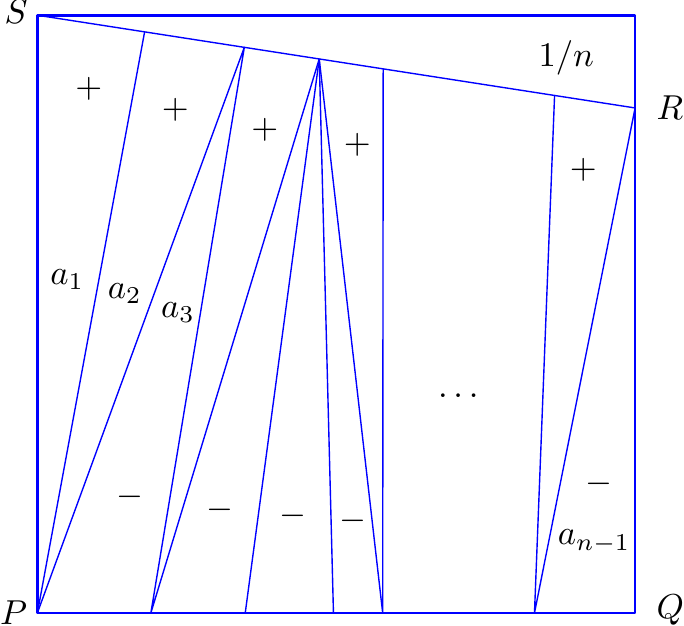}
  \caption{A generalization of the construction from Section~\ref{ssec:n9}}
  \label{fig:general}
\end{figure}

We want to choose the areas and orientations in such a way that the
last triangle with area~$a_{n-1}$ ends flush with the right boundary $QR$.
The calculations for this condition are illustrated schematically in Figure~\ref{fig:calculations}.
Let $O$ be the intersection of the lines spanned by $PQ$ and
$RS$.
Then $\overline{PO}=n/2$, and the triangle ${POS}$ has area $n/4$. 

\begin{figure}
  \centering
  \includegraphics{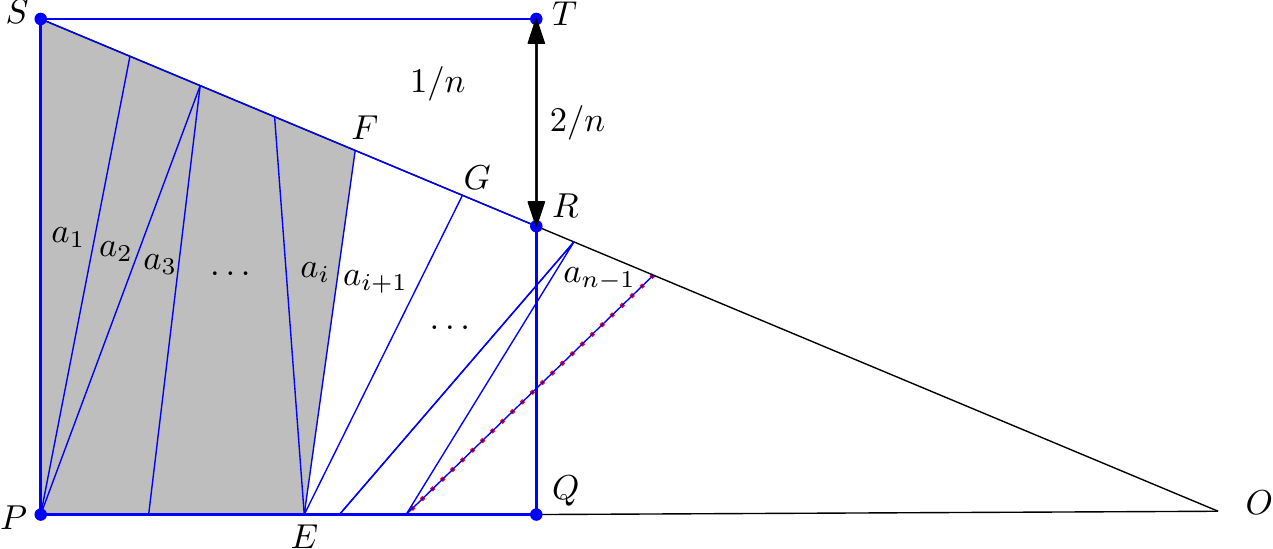}
  \caption{Cutting off the successive triangles $a_1,a_2,\ldots$.  Here,
    the right edge of the last triangle $a_{n-1}$ misses to end up at the
    vertical edge $QR$.  The area $\frac{1}{n}$ of the triangle
    ${RST}$ has been
    exaggerated in order to keep the proportions manageable.  The
    shaded area is $A_i$.}
  \label{fig:calculations}
\end{figure}

Let $A_i=a_1+a_2+\cdots+a_i$.
Suppose that the $(i+1)$-th triangle has orientation $\tau_{i+1}=+1$, as in
Figure~\ref{fig:calculations}.
The relation between its vertices $F$ and $G$ is best understood by
looking at the remaining part ${EFO}$ and ${EGO}$ of the big triangle,
of area $n/4-A_{i}$ and $n/4-A_{i+1}$, respectively.
The ratio of these areas equals the ratio of the lengths $\overline{FO}$ and $\overline{GO}$:
\begin{equation*}
	\frac{\overline{GO}}{\overline{FO}}
= \frac {n/4-A_{i+1}}{n/4-A_i}.
\end{equation*}
If the $(i+1)$-th triangle had orientation $\tau_{i+1}=-1$, we would instead foreshorten the lower side~$EO$ by this proportion.
To have the right edge of $a_{n-1}$ vertical, the product of the foreshortening factors for the top triangles must be equal to the product of the foreshortening factors for the bottom triangles, namely $\overline{RO}/\overline{SO} = \overline{QO}/\overline{PO}$.
We can write this in product form:
\begin{equation}
  \label{eq:balance}
\prod_{i=1}^{n-1}
 \Bigl(\frac {n/4-A_{i}}{n/4-A_{i-1}}\Bigr)^{\tau_i}
 = 1.
\end{equation}

Now we proceed as follows.
We fix the orientations according to the Thue--Morse sequence
$\tau_i=s_{i}$. 
We initially set all areas to the ``ideal'' area $a_i^0:=\frac{1}{n}$ for all $i\geq 1$.
Then we transform the product in \eqref{eq:balance} into a sum of logarithms and develop them into a power series in~$\frac{1}{n}$.
The Thue--Morse sequence will cancel all terms up to degree $k$ in the power series, and thus fulfills~\eqref{eq:balance} to a high degree.
Finally, we perturb the areas in order to satisfy
\eqref{eq:balance} exactly.

If we bound the absolute deviation from the ideal area by $\varepsilon$,
the greatest effect is achieved if we
 perturb every bottom triangle in one direction,
 $a_i = \frac{1}{n} + \varepsilon$, and every top triangle in the opposite direction
 $a_i = \frac{1}{n} - \varepsilon$. The value of $\varepsilon$ may be
 negative,
 and it is obviously bounded by
 \begin{equation*} 
|\varepsilon|< \frac{1}{n}.
    \end{equation*}

By Definition~\ref{def:thuemorse_basic}, successive triangles $a_{2j-1}$ and $a_{2j}$ have opposite orientations: $s_{2j}=-s_{2j-1}$.
Therefore, the perturbations cancel after an even number of triangles, and we obtain
\begin{align} 
  A_{i} &=  
\label{odd-perturb}
          \begin{cases}
A_{i-2}+a_{i-1} +a_i 
=A_{i-2}+(\tfrac 1n\pm \varepsilon)
+(\tfrac 1n\mp \varepsilon)  
= 
\frac i n, &\text{for $i$ even,}\\ 
A_{i-1}+a_i 
= \frac in + s_i\varepsilon, &\text{for odd $i$.}\\
          \end{cases}
\end{align}
Let us denote the product in \eqref{eq:balance} by $\Phi$.
We split it into two factors $\Phi=\Phi^0 \times \Phi^*$.
The factor $\Phi^0$ is the value when we substitute the ideal values $A_i^0=i/n$,
and $\Phi^*$ denotes the deviation caused by perturbing $A_i^0$ to $A_i$.
\begin{align*}
\Phi & 
=
\prod_{i=1}^{n-1}
 \Bigl(\frac {n/4-A_{i}}{n/4-A_{i-1}}
\Bigr)^{s_i}
=\Phi^0 \times \Phi^*
\\  
&=
\prod_{i=1}^{n-1}
 \Bigl(\frac {n/4-A_{i}^0}{n/4-A_{i-1}^0}
\Bigr)^{s_i}
\times
\prod_{i=1}^{n-1}
 \Bigl(\frac {n/4-A_{i}}{n/4-A_{i}^0}
\Bigr)^{s_i}
\cdot
\prod_{i=1}^{n-1}
 \Bigl(\frac {n/4-A_{i-1}}{n/4-A_{i-1}^0}
\Bigr)^{-s_i}
\\
\intertext{We change the iteration variable in the last product and get}
\Phi^*
&=
\prod_{i=1}^{n-1}
 \Bigl(\frac {n/4-A_{i}}{n/4-A_{i}^0}
\Bigr)^{s_i}
\cdot
\prod_{i=0}^{n-2}
 \Bigl(\frac {n/4-A_{i}}{n/4-A_{i}^0}
\Bigr)^{-s_{i+1}}
\end{align*}
Since
$A_i$
differs from $A_i^0$ only for odd $i$, we can simplify this:
\begin{align*}
\Phi^*
&=
\prod_{\substack{1\le i\le n-1\\i\text{ odd}}}
 \Bigl(\frac {n/4-A_{i}}{n/4-A_{i}^0}
\Bigr)^{s_i}
\cdot
\prod_{\substack{0\le i\le n-2\\i\text{ odd}}}
 \Bigl(\frac {n/4-A_{i}}{n/4-A_{i}^0}
\Bigr)^{-s_{i+1}} 
=\left[
\prod_{\substack{1\le i\le n-2\\i\text{ odd}}}
 \Bigl(\frac {n/4-A_{i}}{n/4-A_{i}^0}
\Bigr)^{s_i}
\right]^2
\end{align*}
The last equation holds because $s_{i}=-s_{i+1}$ for odd~$i$.
Now we substitute the values from~\eqref{odd-perturb} and take
logarithms.
\begin{align}
\nonumber
  \ln \Phi^*
&=
2
\sum_{\substack{0\le i\le n-2\\i\text{ odd}}}
s_i
\ln
\frac {n/4-i/n-s_{i}\varepsilon}{n/4-i/n}
\nonumber
\\ &=
2
\sum_{\substack{0\le i\le n-2\\i\text{ odd}}}
s_i
\ln
\left(1-\frac {s_{i}\varepsilon}{n/4-i/n}\right) 
\nonumber
\\ &=
2
\sum_{\substack{0\le i\le n-2\\i\text{ odd}}}
\left(
s_i
\frac {-s_{i}\varepsilon}{n/4-i/n}+O(\tfrac \varepsilon n)^2\right)
\nonumber
\\ &=
-2
\sum_{\substack{0\le i\le n-2\\i\text{ odd}}}
\left(
 \frac {4s_i^2\eps}n(1+O(\tfrac{1}{n}))+O(\tfrac \varepsilon n)^2
\right)
\nonumber
\\ &
=
-2 \frac{n-1}2
\cdot
\frac {4\eps}n(1+O(\tfrac 1n))+O(\tfrac {\varepsilon^2} n)
=
-
\varepsilon
(1+O(\tfrac 1n))
\label{eq-eps}
\end{align}
Now we look at the term $\Phi^0$ and rewrite it 
in terms of a more convenient parameter $$u:=4/n^2$$
as follows:
\begin{gather}
  \Phi^0 
=
\prod_{i=1}^{n-1}
 \Bigl(\frac {n/4-A_{i}^0\hfill}{n/4-A_{i-1}^0}
\Bigr)^{s_i} 
=
\prod_{i=1}^{n-1}
 \Bigl(\frac {n/4-i/n\hfill}{n/4-(i-1)/n}
\Bigr)^{s_i} 
=
\prod_{i=1}^{n-1}
 \Bigl(\frac {1-iu\hfill}{1-(i-1)u}
\Bigr)^{s_i},
\nonumber
\\
 \label{Phi0}
  \ln \Phi^0 
=
\sum_{i=1}^{n-1}
{s_i}
\ln (1-iu)
-
\sum_{i=1}^{n-1}
{s_i}
\ln (1-(i-1)u).
\end{gather}
We use the Taylor formula 
$$\ln (1-x) = -x-\frac{x^2}2
-\frac{x^3}3
-\cdots
-\frac{x^{k}}{k}
-\frac{x^{k+1}}{k+1}
/{(1-\theta x)}^{k}
 =f(x) +\rho(x),
$$
for some $\theta$ with $0\le \theta \le 1$, 
with a polynomial $f(x)$ of degree $k$ and the remainder term $\rho(x)$,
which is bounded by $|\rho(x)|\le 
{x^{k+1}}/[({k+1})(1- x))^{k}]$ for positive~$x$.
This gives
\begin{align*}
  \ln \Phi^0 
&=
\sum_{i=1}^{n-1}
{s_i}f(iu)
-
\sum_{i=1}^{n-1}
{s_i}
f((i-1)u)
+
\sum_{i=1}^{n-1}s_i[\rho(iu)-\rho((i-1)u)].
\end{align*}
By Lemma~\ref{lem:annihilate},
the first two terms can be rewritten in terms of a degree-0 (constant)
polynomial~$F$, and they cancel:
\begin{equation}
  \label{eq:cancel1}
\sum_{i=1}^{2^k}
{s_i}f(iu)
-
\sum_{i=1}^{2^k}
{s_i}
f(-u+iu)
= F(0) - F(-u) = 0.
\end{equation}
The remainder terms are bounded as follows. We assume $n\ge5$ and use
the bound
 $ iu \le nu = 4/n$.
\begin{align*}
\lvert\ln \Phi^0|=
\left|\sum_{i=1}^{n-1}s_i[\rho(iu)-\rho((i-1)u)]\right|
\le 2n \cdot \frac 1{k+1} \Bigl( \frac 4n \Bigr)^{k+1} \Bigl(\frac 1{1-4/n}\Bigr)^{k}
= \frac 8{k+1} 
\Bigl( \frac 1{n/4-1} \Bigr)^{k}
\end{align*}
To satisfy \eqref{eq:balance} and get a dissection, we have to set
$\ln \Phi^* +\ln \Phi^0=0$, or,
using \eqref{eq-eps},
\begin{align*}
\ln \Phi^* = - \varepsilon (1+O(\tfrac 1n)) = - \ln \Phi^0,
\end{align*}
from which we get
\begin{equation}
  \label{eq:eps-bound+}
  |\varepsilon| \le 
{ \lvert\ln \Phi^0\rvert} 
\cdot
(1+O(\tfrac 1n)).
\end{equation}
The expression
$\lvert\ln \Phi^0\rvert$ has been bounded above.
Substituting $k=\log_2(n-1)$ and
assuming $k\ge1$,
we get 
\begin{equation}
  \label{eq:eps}
|\varepsilon| \le 
 \frac{8}{(n/4-1)^{\log_2 (n-1)}(\log_2(n-1)+1)}  (1+O(\tfrac 1n)) .
\end{equation}
The ``$-1$'' terms in the denominator can be swallowed by increasing the error term.
\begin{equation*}
|\varepsilon| \le 
\frac{8}{(n/4)^{\log_2 n}(\log_2n+1)} 
(1+O(\tfrac {\log n}n)) 
=\frac{8n^2}{n^{\log_2 n}(\log_2n+1)} 
(1+O(\tfrac {\log n}n)) .
\end{equation*} 
This is valid for values of $n$ where $n-1$ is a power of 2.
In general, let
$n=2^k+\ell$, where $k = \lfloor \log_2 n \rfloor$ and $\ell\ge 1$ is odd.
By the scaling trick of Lemma~\ref{add2}, we can reduce this to the case $n' =
2^k+1\ge n/2$  
and obtain the bound 
\begin{align*}
|\varepsilon| \le \frac{2^k+1}{2^k+\ell}|\varepsilon'|\le
|\eps'| 
&\le
\frac{8n'^2}{n'^{\log_2 n'}(\log_2n'+1)}
(1+O(\tfrac {\log n'}{n'}))
\\
&\le
\frac{8n^2}{(n/2)^{\log_2 n -1 }(\log_2n -1+1)}
(1+O(\tfrac {\log n}{n}))
\\
&=
\frac{8n^2\cdot n/2}{(n/2)^{\log_2 n }\log_2n}
(1+O(\tfrac {\log n}{n})) 
=
\frac{4n^4}{n^{\log_2 n }\log_2n }
(1+O(\tfrac {\log n}{n})). 
\end{align*} 
Since the range is $2|\eps|$, 
Theorem~\ref{thm:superpoly} follows.
\end{proof}

%
%
%
\newcommand\setN{}
\def\setN #1 {#1
\def \findLine ##1:##2\\{
  \ifnum ##1=-1 \let\Fnext=\relax \else
  \let\Fnext=\findLine
  \ifnum #1=##1 \gdef\tableEntry{##2:}\fi\fi
  \Fnext
}%
\findLine
 3:$   -32.000$:$-$\\
 5:$    85.333$:$-$\\
 7:$    60.952$:$-$\\
 9:$    2.0480$:$-$\\
11:$    1.6756$:$-$\\
13:$    1.4178$:$-$\\
15:$    1.2288$:$-$\\
17:$  0.028682$:$ 0.5538$\\
19:$  0.025663$:$ 0.5411$\\
21:$  0.023219$:$ 0.5305$\\
23:$  0.021200$:$ 0.5213$\\
25:$  0.019504$:$ 0.5132$\\
27:$  0.018059$:$ 0.5061$\\
29:$  0.016814$:$ 0.4998$\\
31:$  0.015729$:$ 0.4940$\\
33:$0.00013313$:$ 0.7113$\\
35:$0.00012552$:$ 0.7018$\\
37:$0.00011874$:$ 0.6932$\\
39:$0.00011265$:$ 0.6852$\\
41:$0.00010715$:$ 0.6778$\\
43:$0.00010217$:$ 0.6710$\\
45:$9.7629{\times}10^{-5}$:$ 0.6646$\\
47:$9.3475{\times}10^{-5}$:$ 0.6587$\\
49:$8.9660{\times}10^{-5}$:$ 0.6531$\\
51:$8.6143{\times}10^{-5}$:$ 0.6478$\\
53:$8.2893{\times}10^{-5}$:$ 0.6428$\\
55:$7.9879{\times}10^{-5}$:$ 0.6382$\\
57:$7.7076{\times}10^{-5}$:$ 0.6337$\\
59:$7.4463{\times}10^{-5}$:$ 0.6295$\\
61:$7.2022{\times}10^{-5}$:$ 0.6255$\\
63:$6.9735{\times}10^{-5}$:$ 0.6217$\\
65:$1.8172{\times}10^{-7}$:$ 0.7857$\\
67:$1.7630{\times}10^{-7}$:$ 0.7808$\\
69:$1.7119{\times}10^{-7}$:$ 0.7761$\\
71:$1.6636{\times}10^{-7}$:$ 0.7716$\\
73:$1.6181{\times}10^{-7}$:$ 0.7673$\\
75:$1.5749{\times}10^{-7}$:$ 0.7632$\\
77:$1.5340{\times}10^{-7}$:$ 0.7592$\\
79:$1.4952{\times}10^{-7}$:$ 0.7554$\\
129:$6.8719{\times}10^{-11}$:$ 0.8287$\\
257:$6.9405{\times}10^{-15}$:$ 0.8567$\\
513:$1.8289{\times}10^{-19}$:$ 0.8763$\\
1025:$1.2390{\times}10^{-24}$:$ 0.8910$\\
2049:$2.1378{\times}10^{-30}$:$ 0.9025$\\
4097:$9.3410{\times}10^{-37}$:$ 0.9117$\\
8193:$1.0298{\times}10^{-43}$:$ 0.9192$\\
16385:$2.8583{\times}10^{-51}$:$ 0.9255$\\
32769:$1.9941{\times}10^{-59}$:$ 0.9309$\\
65537:$3.4936{\times}10^{-68}$:$ 0.9356$\\
131073:$1.5359{\times}10^{-77}$:$ 0.9396$\\
262145:$1.6936{\times}10^{-87}$:$ 0.9432$\\
524289:$4.6822{\times}10^{-98}$:$ 0.9464$\\
1048577:$3.2443{\times}10^{-109}$:$ 0.9492$\\
-1:: \\}
 \begin{table}[!ht]
\def\takeSecond #1:#2:{#2}
\def \g{\expandafter\takeSecond\tableEntry}
  \centering
 \def\star{\rlap{$^*$}}
\newif\ifflip
\newif\ifflipp
\let\small=\footnotesize
\def\oppositeS #1{\ifx +#1-\else\ifx-#1+\else\let\opposite\relax\fi\fi\opposite}
\def\small$#1${\ifflip
   \global\flipptrue
   \let\opposite\oppositeS
\else
   \global\flippfalse
   \let\opposite\relax
\fi
\raise 0,123ex\hbox{$\scriptstyle\opposite#1\relax$}}
\newcommand\checksign{}
\newcommand\setsign{}
\def\setsign#1{\ifx #1-\global\flipptrue\else\global\flippfalse\fi}
\def\checksign #1{\ifx #1-\ifflipp\else\errmessage{wrong sign}\fi
\else\ifflipp\errmessage{wrong sign}\fi\fi}
\noindent \hskip 0pt minus 0,5em
  \begin{tabular}{rlllccc}
$n$&optimal sign sequence $s$&$\hfil\eps$
&\hfil \RMS&$\Lambda(\range_{\mathrm{opt}})$&$\Lambda(\range_C)$&$\Lambda(\range^*)$\\[1ex]
\setN  3 \star &\fliptrue \small${+}{-}$& $ \checksign-0.16667$& $  0.13608$& $
0.7943$& $ 0.7943$&\g\\
\setN  5 \star &\flipfalse \small${+}{-}{-}{+}$& $  \checksign+0.01250$& $  0.01118$& $
0.9935$& $ 0.9935$&\g\\
\setN  7 &\fliptrue \small${+}{-}{+}{-}{-}{+}$& $\checksign-0.00010248$&
$0.00009488$& $ 1.2468$& $ 0.8584$&\g\\
\setN  9 \star &\fliptrue \small${+}{-}{-}{+}{-}{+}{+}{-}$& $\checksign-0.00016360$&
$0.00015424$& $ 1.0734$& $ 1.0734$&\g\\
\setN 11 &\fliptrue \small${+}{-}{+}{-}{+}{-}{-}{+}{-}{+}$&
$\checksign-4.1201{\times}10^{-6}$& $3.9284{\times}10^{-6}$& $ 1.1879$& $
0.9958$&\g\\
\setN 13 &\flipfalse \small${+}{-}{-}{+}{+}{-}{-}{+}{-}{+}{+}{-}$&
$\checksign+5.9928{\times}10^{-6}$& $5.7577{\times}10^{-6}$& $ 1.0927$& $
0.9403$&\g\\
\setN 15 &\fliptrue \small${+}{-}{+}{-}{+}{-}{+}{-}{-}{+}{-}{+}{-}{+}$&
$\checksign-5.2871{\times}10^{-7}$& $5.1079{\times}10^{-7}$& $ 1.1404$& $
0.8982$&\g\\
\setN 17 \star &\fliptrue \small${+}{+}{-}{+}{-}{-}{+}{-}{-}{+}{-}{+}{-}{-}{+}{+}$& $\checksign-3.4708{\times}10^{-8}$& $3.3672{\times}10^{-8}$& $ 1.1930$& $ 1.1076$&\g\\
\setN 19 &\flipfalse \small${+}{-}{-}{+}{-}{-}{+}{-}{+}{+}{+}{+}{-}{-}{+}{-}{+}{-}$& $\checksign+4.2052{\times}10^{-8}$& $4.0931{\times}10^{-8}$& $ 1.1413$& $ 1.0699$&\g\\
\setN 21 &\fliptrue \small${+}{-}{+}{-}{-}{+}{-}{+}{+}{-}{-}{-}{+}{+}{+}{+}{-}{-}{-}{+}$& $\checksign-5.5778{\times}10^{-9}$& $5.4434{\times}10^{-9}$& $ 1.1702$& $ 1.0383$&\g\\
\setN 23 &\flipfalse \small${+}{+}{-}{-}{+}{-}{-}{+}{+}{-}{-}{+}{-}{+}{+}{-}{-}{+}{-}{+}{-}{+}$& $\checksign+3.5359{\times}10^{-9}$& $3.4581{\times}10^{-9}$& $ 1.1503$& $ 1.0114$&\g\\
\setN 25
&\fliptrue \small${+}{+}{+}{-}{+}{-}{-}{-}{-}{-}{+}{+}{+}{-}{-}{-}{-}{+}{+}{+}{+}{-}{-}{+}$&
$\checksign-7.457
{\times}10^{-10}$& $7.307 
{\times}10^{-10}$& $ 1.1660$& $ 0.9880$&\g\\
\setN 27
&
& $\checksign-1.266{\times}10^{-10}$& $1.242{\times}10^{-10}$& $ 1.1875$& $ 0.9675$&\g\\
\setN 29
\global\def\checksign#1{} 
&
& $\checksign+9.026
{\times}10^{-12}$& $8.869
{\times}10^{-12}$& $ 1.2297$& $ 0.9492$&\g\\
\setN 31
&
& $\checksign+2.446
{\times}10^{-12}$& $2.406{\times}10^{-12}$& $ 1.2373$& $ 0.9329$&\g\\
\setN 33 \star
&
& $\checksign-1.423
{\times}10^{-12}$& $1.401
{\times}10^{-12}$& $ 1.2277$& $ 1.1237$&\g\\
\setN 35
&
& $\checksign+1.777
{\times}10^{-13}$& $1.752 
{\times}10^{-13}$& $ 1.2537$& $ 1.1066$&\g\\
\setN 37
&
& $\checksign+1.100
{\times}10^{-14}$& $1.086 
{\times}10^{-14}$& $ 1.2930$& $ 1.0909$&\g\\
\setN 39 
&
& $\checksign-2.119{\times}10^{-14}$& $2.092
{\times}10^{-14}$& $ 1.2610$& $ 1.0765$&\g\\
  \end{tabular}

  \caption{The optimal sign sequence in comparison to the systematic construction. The values of the form
    $n=2^k+1$ are marked with a star.}
  \label{tab:construction}
\end{table}







\subsection{Experimental improvements}
\label{ssec:experiments}

For small values of $n$, we have computed the optimal dissection
within the above framework by trying all sign sequences with equally
many + and $-$ signs. 
Table~\ref{tab:construction} reports the optimal sign sequence (for
$n\le 25$) and the
resulting value of $\eps$. In these calculations, we have always kept the
upper right triangle that is cut off initially at its original size $1/n$. 
We miss some solutions with a smaller range in this way,
but for larger $n$, the loss is negligible.
The range is $\range=2\eps$ and the~\RMS, which is also reported, is $\eps\sqrt{\frac{n-1}n}$,
since we have $n-1$ triangles with an error of $\eps$ and one triangle
with error~0.
The best sign sequence was always unique up to flipping all signs,
and
by flipping signs if necessary, we have ensured a positive~$\eps$.

We can observe that the quality of these solutions is not monotone
in~$n$. The solution for $n=7$ has smaller errors than the Thue--Morse
solution for $n=9$, which is optimal within its class. 
For $n=9$, it is therefore better to use the solution for $n=7$ and
extend it with the help of Lemma~\ref{add2}.
The next inversion occurs between $n=11$ and $n=13$.
For $n=17$ and $n=33$, the Thue--Morse sequence is also not the winner.
For $n=17$, it is only the third-best solution, with $\eps$ about ten times larger
than for the optimum.
For $n=33$, the best solution is about 75 times better than the Thue--Morse sequence.
The reason is that, when we
look at the difference in the left-hand side of~\eqref{eq:cancel1} 
for the full power series of $\ln(1-x)$ instead of the truncated
series $f(x)$,
 the lower-order terms
can be very small for the particular value of $u$,
despite the fact that they don't cancel systematically.
In Section~\ref{ssec:heuristic} we attempt to give another explanation for this phenomenon.

Theorem~\ref{thm:superpoly} predicts a decrease roughly of the order
$\range \approx \mathrm{const}/2^{(\log_2 n)^2}$. Therefore we report
the value
$$\Lambda(\range) = \sqrt{\log_n  \frac 1\range},$$ 
which should converge
to a constant if the Thue--Morse sequence gives the optimal value.
A~larger value $\Lambda(\range)$ indicates a smaller range and therefore a better solution.
We also report the corresponding values $\Lambda(\range_C)$ for the general construction of
Theorem~\ref{thm:superpoly}, which uses the Thue--Morse sequence for
the closest power of~$2$ and tiles it with triangles of area $1/n$.
For comparison, the last column gives $\Lambda(\range^*)$ for the ``promised''
value $\range^*=2|\eps|$ in the proof of
Theorem~\ref{thm:superpoly}
 that results from ignoring the $O(1/n)$ term in~\eqref{eq:eps}. 
 This converges to $1$ as $n$ increases, with
 intermediate deteriorations between the powers of~$2$.

Table~\ref{tab:systematic} gives the results for the systematic
construction for the powers of two, together with the promised range
$\range^*$.
Both are also expressed in terms of the function $\Lambda$ as above, to
exhibit the converging behavior.
The function $\Lambda$ makes the differences appear small, while they are really
more spectacular.
For example, for $n=65537$, the ``promised'' range is
$1.74\times10^{-68}$, but the true range is only $4.15\times10^{-105}$.
\begin{table}
\def\takeSecond #1:#2:{#2}
\def\takeFirst #1:#2:{#1}
\newcommand \ggu{\expandafter\takeFirst\tableEntry}
\newcommand \g{\expandafter\takeSecond\tableEntry}
  \centering
  \begin{tabular}{rllcc}
$n$&\hfil $\range_C=2|\eps|$&\hfil $\range^*$&$\Lambda(\range_C)$&$\Lambda(\range^*)$\\[1ex] 
\setN       3 & $   0.33333333$& \hfil$    -$& $ 0.7943$&\g\\ %
\setN       5 & $     0.02500000$& \hfil$    -$& $ 0.9935$&\g\\ %
\setN       9 & $0.00032719$&\ggu& $ 1.0734$&\g\\
\setN      17 & $6.7688{\times}10^{-7}$&\ggu& $ 1.1076$&\g\\
\setN      33 & $2.1229{\times}10^{-10}$&\ggu& $ 1.1237$&\g\\
\setN      65 & $9.8506{\times}10^{-15}$&\ggu& $ 1.1326$&\g\\
\setN     129 & $6.6218{\times}10^{-20}$&\ggu& $ 1.1385$&\g\\
\setN     257 & $6.3377{\times}10^{-26}$&\ggu& $ 1.1428$&\g\\
\setN     513 & $8.5151{\times}10^{-33}$&\ggu& $ 1.1465$&\g\\
\setN    1025 & $1.5875{\times}10^{-40}$&\ggu& $ 1.1497$&\g\\
\setN    2049 & $4.0679{\times}10^{-49}$&\ggu& $ 1.1525$&\g\\
\setN    4097 & $ 1.4210{\times}10^{-58}$&\ggu& $ 1.1552$&\g\\
\setN    8193 & $6.7214{\times}10^{-69}$&\ggu& $ 1.1576$&\g\\
\setN   16385 & $4.2794{\times}10^{-80}$&\ggu& $ 1.1598$&\g\\
\setN   32769 & $3.6489{\times}10^{-92}$&\ggu& $ 1.1619$&\g\\
\setN   65537 & $4.1484{\times}10^{-105}$&\ggu& $ 1.1638$&\g\\
\setN  131073 & $6.2638{\times}10^{-119}$&\ggu& $ 1.1656$&\g\\
\setN  262145 & $1.2518{\times}10^{-133}$&\ggu& $ 1.1673$&\g\\
\setN  524289 & $3.3006{\times}10^{-149}$&\ggu& $ 1.1689$&\g\\
\setN 1048577 & $1.1451{\times}10^{-165}$&\ggu& $ 1.1704$&\g\\
\end{tabular}

  \caption{The range $\range_C$ of the systematic construction with the
    Thue--Morse sequence,
and in comparison, the upper bound $\range^*$ from the proof of Theorem~\ref{thm:superpoly}.}
  \label{tab:systematic}
\end{table}

\subsection{Towards a family of triangulations of exponential decreasing range}
\label{ssec:triangulations}

So far, we have constructed families of good \emph{dissections}.
We now describe a family of \emph{triangulations} that is rich enough
so that we may hope to find triangulations with small area range
among them.
In contrast to Section~\ref{superpolynomial}, we have no systematic
construction that would yield good candidates.
Our family is
 characterized by
containing a
path $P$ through all interior nodes, starting in the upper left corner and
terminating at the right
edge, see Figure~\ref{fig:construction}.
Such a structure was already observed by Mansow
in many triangulations that she found
\cite[Observation 4 and Figure 3.21 on p.~34]{mansow_ungerade_2003}.
In addition,
we restrict all additional vertices to lie on the bottom edge.
The combinatorial type is characterized by a sequence of 1's and
2's that sum up to $n$. 
At each step, we either add ``2'' triangles that connect an edge of
$P$ with the upper right corner and the bottom side, or we add ``1''
triangle with an edge on the bottom side.
 When all areas are set to $1/n$, the path will
not terminate at the right edge, and thus the triangle areas must be
adjusted
in order to obtain a valid triangulation.

\begin{figure}[htb]
  \centering
  \includegraphics{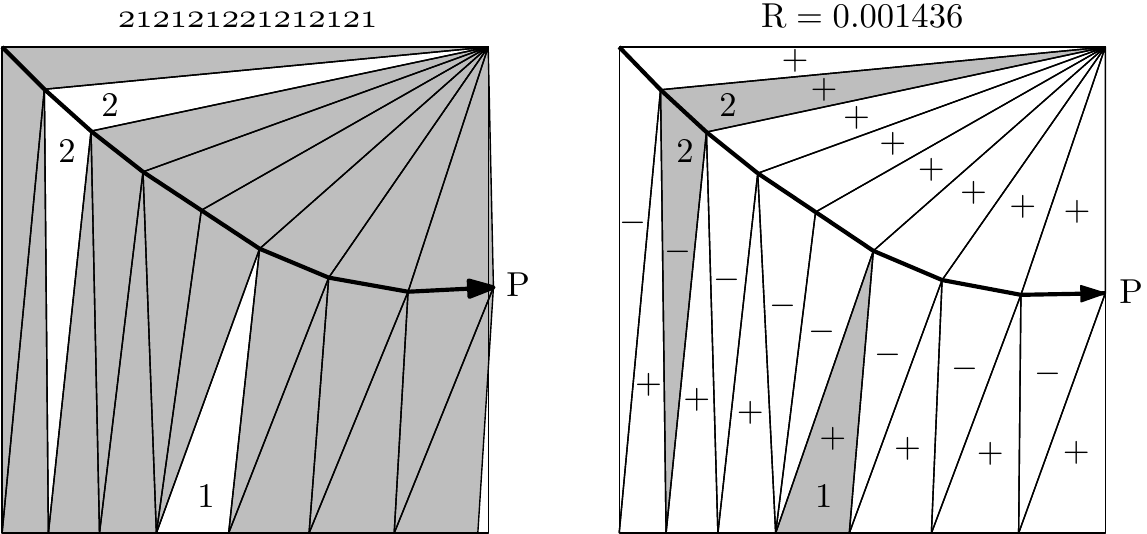}
  \caption{A typical specimen of our family of triangulations with $n=23$
    triangles.
In the left picture, all triangles have area $1/n$, and the path $P$
does not end on the right edge. In the right picture, the areas are adjusted.
  The signs in the triangles indicate whether their areas are bigger or
    smaller than the ideal value $1/n$. All triangles with the same
    sign have exactly the same area, in this case
 $ 1/n+\frac{8}{23}\range$ and
 $ 1/n-\frac{15}{23}\range$, respectively.}
  \label{fig:construction}
\end{figure}

For small values of $n$, we have enumerated all combinatorial types of
this family.
Table~\ref{tab:triangulations} records the triangulations that yield
that smallest range.
It is not obvious which areas should be increased or which should
be decreased when the areas are adjusted. Thus we don't even ensure that
the reported ranges are optimal for the given combinatorial types.
For example, the best triangulation with $n=11$ triangles found by
Mansow has range $0.000{,}004{,}152$, as reported in Table~\ref{tab:diss_discr}.
This triangulation falls in our family: its pattern is
\oldstylenums{1212221}.  The solution reported by Mansow \cite[Figure
3.14 on p.~28]{mansow_ungerade_2003} uses more than two distinct
areas.  With our method of adjusting the areas either up or down to
two distinct values, we could only reach a range of
$0.000{,}004{,}673$, see Table~\ref{tab:triangulations}.
%
\begin{table}[htb]
  \centering
\def \osn#1 {\oldstylenums{#1}}
  \begin{tabular}{|r|l|l|}
\hline
    $n$&code&\hfil range $\pm$\range\\
\hline
  3 & \osn21 & $    -0.25$ \\ 
  5 & \osn1210 & $-0.0225425$ \\ 
  7 & \osn21210 & $\phantom+0.00312387$ \\ 
  9 & \osn1221100 & $-0.00337346$ \\ 
 11 & \osn1212221 &  $\phantom+0.0000046733325$ \\ 
 13 & \osn22212211 & $-0.000210277$ \\ 
 15 & \osn22111212100 & $\phantom+5.41030{\times}10^{-5}$ \\ 
 17 & \osn22211121221 & $\phantom+6.93952{\times}10^{-6}$ \\ 
 19 & \osn1211212222110 & $\phantom+4.10530{\times}10^{-7}$ \\ 
 21 & \osn2112212222211 & $-1.70884{\times}10^{-6}$ \\ 
 23 & \osn21121222222211 & $-1.25472{\times}10^{-7}$ \\ 
 25 & \osn2211112212222121 & $\phantom+2.16265{\times}10^{-8}$ \\ 
 27 & \osn21222122211221121 & $\phantom+7.89259{\times}10^{-9}$ \\ 
 29 & \osn12112212121212112210 & $-2.69003{\times}10^{-8}$ \\ 
 31 & \osn12212212222212210000 & $\phantom+2.97133{\times}10^{-9}$ \\ 
 33 & \osn211121121212212112121000 & $\phantom+9.03785{\times}10^{-10}$ \\ 
 35 & \osn22212221112212211221110 & $\phantom+1.00009{\times}10^{-12}$ \\ 
 37 & \osn121222112222212112211211 & $\phantom+1.88583{\times}10^{-11}$ \\ 
 39 & \osn2121211222212121221122112 & $\phantom+7.32381{\times}10^{-14}$ \\ 
 41 & \osn211122212121222112212112211 & $\phantom+4.53634{\times}10^{-13}$ \\ 
 43 & \osn1211212221222222112112211211 & $\phantom+1.00055{\times}10^{-16}$ \\ 
%
\hline
  \end{tabular}

  \caption{Best triangulations found. The sign of the range \range\
    indicates in which direction the areas of the type-1 triangles
    that sit on the bottom edge were perturbed.}
  \label{tab:triangulations}
\end{table}

It may happen that the path reaches the bottom-right corner prematurely
(Figure~\ref{fig:examples}b).
In such cases, the remaining area becomes a triangle, which can be trivially
partitioned into the right number of equal-sized triangles. Such
triangles are indicated by 0's in the code of Table~\ref{tab:triangulations}.

\begin{figure}[htb]
  \centering
\ \vbox{%
\halign{\hfil#\hfil\qquad&\hfil#\hfil\cr
  \includegraphics[scale=0.8]{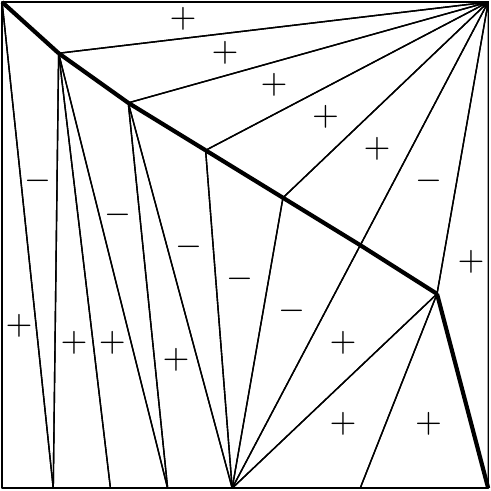}&
  \includegraphics[scale=0.8]{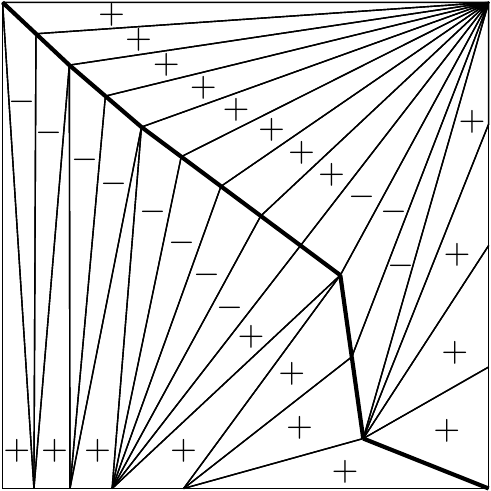}\cr
(a)&(b)\cr}}
  \caption{The best triangulations for $n=19$ triangles
(code \oldstylenums{1211212222110})
 and $n=31$ triangles
(code \oldstylenums{12212212222212210000})
    that we found}
  \label{fig:examples}
\end{figure}

The numbers in the table indicate an exponential decrease, but
in contrast to the case of dissections, we cannot even show a
superpolynomial decrease.

\subsection{A heuristic argument for an exponential decrease}
\label{ssec:heuristic}
The experiments indicate that, as~$n$ gets large, there are much
better solutions than the ones provided by the Thue--Morse sequence. We attempt
a non-rigorous argument, based on an analogy with a random experiment,
why one might even expect an exponentially small range.
We use the setup of the proof of Theorem~\ref{thm:superpoly}.
We view the quantity in~\eqref{Phi0},
\begin{equation}
  \label{eq:T}
T:=\ln \Phi^0 
=
\sum_{i=1}^{n-1}
{s_i}
\ln \frac{1-iu}{1-(i-1)u}
= 
\sum_{i=1}^{n-1}
{s_i}(u+O(u^2))
\end{equation}
as a random variable:
Instead of the Thue--Morse sequence $s_i$, we choose a sign $\tau_i=\pm1$ independently uniformly
at random. 
$T$ is a discrete random variable with $N=2^{n-1}$ equally likely
outcomes.
By the Central Limit Theorem, $T$ is approximately Gaussian with mean 0 and standard
deviation $\sigma\approx \sqrt{n-1}\cdot u \approx 4/n^{3/2}$.
We now make a leap of faith and assume that the error terms $O(u^2)$
in~\eqref{eq:T} act like random fluctuations that eradicate all
systematic dependencies and all traces how the
values were generated. In particular, we assume that,
around the origin, $T$ behaves
like $N$ independent samples from the approximating Gaussian.
This means that $T$ is locally distributed
 like a Poisson process
with density $\lambda = f(0)\cdot N$, where $f(0)= 1/{\sqrt{2\pi \sigma}}$
 is the density function of the
approximating Gaussian at the point~$0$.
In a Poisson process with density $\lambda$, the expected smallest
absolute
value (i.e., the expected distance from the origin to the closest point) is
$1/(2\lambda)$.
Putting all of this together, we obtain
$$E(\min\, \lvert T\rvert) =
E(\min\, \lvert\ln \Phi^0\rvert) =
 \frac 1{2\lambda} = 
{\sqrt{2\pi}}\cdot
\frac{n^{3/4}}{2^n}
.$$
This would give an exponentially small upper bound on the expectation of
$\lvert\ln \Phi^0 |$ for the best sequence
$\tau_1,\ldots,\tau_{n-1}$. By~\eqref{eq:eps-bound+}, this translates
directly into a bound on the expected range~$2|\eps|$.  
Thus, accepting the probabilistic model, this would give an
exponential upper bound on the area range that holds with high
probability, for a given~$n$. It does not rule out that, for a few
exceptional values of~$n$, much smaller minima exist, and thus this
argument
can obviously not be used for a lower bound.

We note that
only the $2^{n/2}$ sequences $\tau_1,\ldots,\tau_n$ that consist of
pairs $+-$ and $-+$ were considered in this argument (and these were
reduced to sequences  $\tau_1,\ldots,\tau_{n-1}$  of length $n-1$ in the
analysis), whereas
the experiments reported in the table consider all $\binom n2$
sequences in which the signs are balanced. 

If we extend the above arguments to an even wilder speculation, it
would imply the ``meta-theorem'' that an optimization problem with $N$
combinatorially distinct configurations should be expected to have a
minimum of the order $\mathrm{poly}(n)/N$.  In our setup, we have
considered a restricted set of $N=2^{n/2}$ dissections of a special
type.  The total number of dissections is also just singly exponential
in $n$, so our restriction causes at most a deterioration in the base
of the exponential growth in this argument.

The number of \emph{triangulations} grows also exponentially with $n$.
(Already the family of combinatorial types considered in our
experiments in
Section~\ref{ssec:triangulations} is exponential.)
Thus, even for triangulations, we can ``expect'' an exponentially
small area deviation.

\subsection{The Tarry--Escott Problem}
\label{sec:tarry}

The question of assigning signs $s_i=\pm1$ in order to cancel
the first $k$
powers
is related to
the so-called Tarry--Escott Problem
 (or Prouhet--Tarry--Escott Problem, or Tarry Problem)~\cite{t-e-p}.
This problem asks for two distinct sets of integers
$\alpha_1,\ldots,\alpha_n$
and
$\beta_1,\ldots,\beta_n$
such that
$$
\alpha_1^d+\dots+\alpha_n^d
=
\beta_1^d+\dots+\beta_n^d,
\text{ for all $d=0,1,2,\ldots,k$}
$$
The solution that corresponds to the first $2^{k+1}$ elements of
the Thue--Morse sequence (Lemma~\ref{lem:annihilate}) was proposed already in 1851
by Eug\`ene  Prouhet~\cite{prouhet}, even in a generalized
setting
where $b^{k+1}$ numbers in an arithmetic progression are partitioned into $b$ sets with equal sums
of powers.

The objective
in the Tarry--Escott Problem
 is to find solutions of small size $n$,
and to come close to the lower bound of $n=k+1$.
In our application, we have the additional constraint that the two
sets form a
partition of the successive integers
$\{1,\ldots,2n\}$ into two parts (see also~\cite{monthly-problem}). 

Some computer runs for small exponents $k$ have found no improvements
over the Thue--Morse sequence.
For example, for $k=3$, the only sequence lengths that allow a partition
with equal sums of powers are $16,24,32,48,\ldots$, and the shortest
one of length $2n=16$ is the Prouhet solution with
the Thue--Morse sequence.
 
\section{Even dissections with unequal areas}
\label{sec:even}
 
Because of Monsky's Theorem,
we have concentrated on dissections with an odd number of triangles.
However, there are also combinatorial types of dissections with an \emph{even}
number of triangles for which the areas cannot be equal.
As pointed out by one of the referees,
our bounds also apply in these cases.

\begin{example}
Figure~\ref{even} shows a dissection of a square into $4$ triangles,
and two triangulations with $16$ and $26$ triangles. In example~(b),
the areas cannot all be equal to $1/16$ because 
the highlighted triangle would then contain more than half of the area of the
square. This is impossible, 
as no triangle contained in the square
can contain more than half of the area.
 Example~(c) is a bit more delicate; it requires a little calculation,
 which we leave as a challenge for the reader.
(It is a manifestation of the fact that the \emph{octahedron graph},
the graph obtained from the lower left half of   Example~(c)
by removing the degree-3 vertices, is not \emph{area-universal}, see
\cite{r-eg-90,k-dpgpf-16}.)

\begin{figure}
  \centering
  \includegraphics{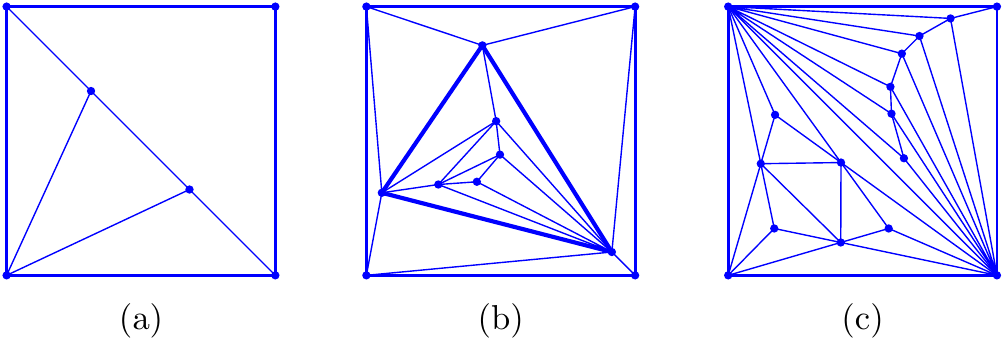}
  \caption{(a) Four triangles where the most balanced area partition
    is
    $\frac12,\frac16,\frac16,\frac16$.
    (b)~A triangulation into 16 triangles, 9 of which
    lie inside the thick triangle.
(c)~A~triangulation with 26 triangles. 
  }
  \label{even}
\end{figure}
\end{example}
All these examples contain a separating triangle.
We are not aware of any 4-connected even triangulation for which one cannot
achieve equal areas.

The approach that we have taken for dissections where the number of
triangles is odd, does not directly carry over
to the even case, 
for the following reason:
We have extended our concept of dissections,
and some of its rules, when taken in isolation, allow drawings that
violate the requirements of a dissection, 
see Examples~\ref{illegal} and \ref{slide-outside}, and
Figure~\ref{fig:ex_frame_emb2}.  Nevertheless, Monsky's approach was
powerful enough to show that even these ``generalized'' equipartitions
cannot exist.

However, the following lemma allows us to extend our approach at least
to the case when there are no ``flipped-over'' triangles, and the
\emph{signed} areas are positive.

\begin{lemma}\label{dissections-for-even}
Let $\phi$ be a framed map of a simplicial graph $\simplgraph$ of a dissection of a simple
$k$-gon~$P$ where all triangular faces of $\simplgraph$ have nonnegative signed area, \textup(In particular, the triangular faces of $\graph$ have the correct orientation.\textup)
Then the triangular faces of $\simplgraph$ with positive area form a dissection of~$P$. 
\end{lemma}

\begin{proof} 
For a  point $x$ of the plane, let $\chi(x)$ 
be the number of triangles
  that contain $x$.
Our goal is to show that, for points $x$ that don't lie on an edge of
$\simplgraph$,
$\chi(x)=1$ if $x$ lies in $P$ and $\chi(x)=0$ otherwise.
The standard argument for this establishes that
 $\chi(x)$ remains unchanged when $x$ crosses a triangle edge,
except when the edge
forms the boundary of~$P$, and in that case it changes in the ``correct'' way,
see for example \cite[Section 3]{triangulations-polytope-1996} and \cite[Theorem
4.3]{firla-ziegler}. 
In our case,
the presence of zero-area triangles makes the argument more complicated.

In addition to
  the given drawing $D_0$ of
  $\simplgraph$ (potentially with overlapping edges and coincident
  nodes, and
  conceivably even with crossing edges), let us consider a
plane drawing $D$ of
  $\simplgraph$,
see Figure~\ref{fig:orientation-is-enough}a--b.
We look at an arbitrary point $v$ of the plane that
does not lie on any edge of $D_0$, and we ask in how many triangles of 
  $\simplgraph$ it is contained. We find a curve $W$ from $v$ to
  infinity that avoids all nodes and all intersections between edges,
  and crosses the edges in a finite number of points.
Let us focus on a point
  $K$ where $W$ crosses a line segment of the drawing $D_0$.
The curve $W$ may cross
  several edges of $\simplgraph$ 
simultaneously. We select in
$D$ the edges which are crossed by $W$ at this point $K$, and we orient
  the corresponding edges of the
  dual graph $D^*$ of $D$ accordingly.
  \begin{figure}
    \centering
    \includegraphics{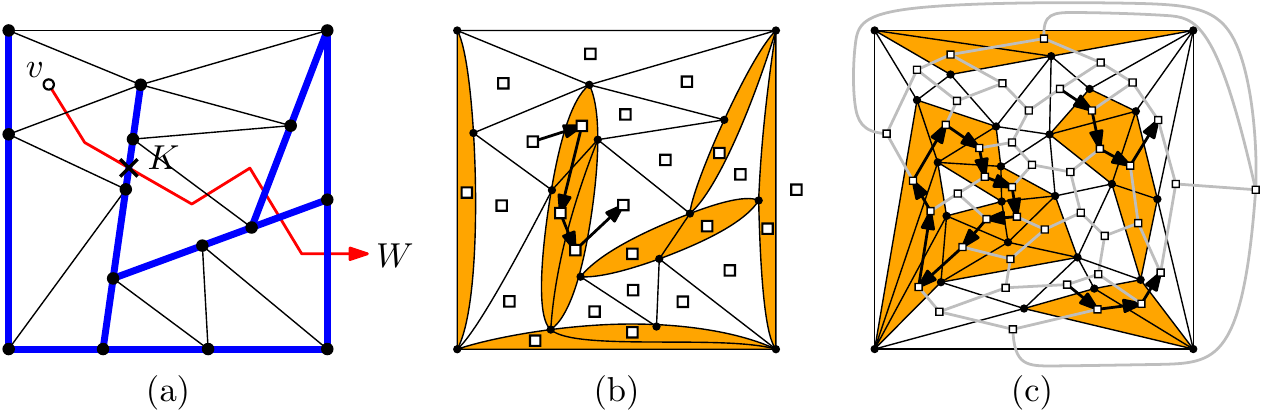}
    \caption{(a) A drawing $D_0$ and a curve $W$ from point $v$ to infinity, crossing a
      segment at $K$;
(b)~the dual path in the drawing $D$ 
that corresponds to the crossing at~$K$;
(c)~a hypothetical dual graph $D^*$, and the directed subgraph $U$ that might
correspond to a crossing.
(The subgraph
induced by the shaded  zero-area triangles cannot actually contain a cycle,
but
we don't need this in our argument.)
}
    \label{fig:orientation-is-enough}
  \end{figure}
The following observations are immediate from this definition.
  \begin{enumerate}
  \item [(i)]
 If $W$ crosses an edge of a zero-area triangle $T$, it will
    cross another edge of $T$ in the opposite direction. This means
    that one of the edges enters $T$ and the other one leaves $T$.
  \item [(ii)]
If $W$ crosses an edge of a triangle $T$
with nonzero area or of the outer polygon $P$, it  crosses no other
edge of this face at this point.
  \end{enumerate}
Figure~\ref{fig:orientation-is-enough}c shows a more elaborate example
of a 
directed graph $U$
that satisfies conditions~(i) and~(ii).
These conditions mean that the directed subgraph $U$ of $D^*$ has indegree=outdegree for
  every  zero-area triangle,
  and degree $\le1$ for every nondegenerate triangle.
  It follows that $U$ consists of node-disjoint directed cycles and
  directed paths. The endpoints of such a path can be a nondegenerate triangle or
  the outer face.

Whenever a path starts in some nondegenerate triangle $T$, the curve
$W$ leaves this
triangle $T$ at the crossing $K$, and whenever a path ends in some
nondegenerate triangle $T$, $W$ enters $T$ at the crossing $K$. Two such changes taken
together have no net effect on~$\chi$.

The conclusion is that, if $W$ does not cross a boundary edge,
$\chi(x)$ does not change.
If $w$ crosses a boundary edge,
$\chi(x)$ changes by $\pm 1$ as appropriate for~$P$. This follows from the
assumption that the boundary nodes, and hence also the boundary
edges, are embedded at the correct corners and sides of $P$.
Since $\chi(x)$ has the correct value 0 when $x$ is far away, it
follows that
 $\chi(x)$ has the correct value everywhere, and therefore,
the triangles
cover $P$ with disjoint  interiors; in other words, they form a dissection.
\end{proof} 

With this lemma,
the lower bound of our main {theorem} (Theorem~\ref{thm:double_exp_lower})
carries over to even dissections
with a given combinatorial type
 for which the minimum is nonzero (for whatever reason).
 The proof can be used verbatim, except that
Lemma~\ref
{dissections-for-even} replaces the application of
Lemma~\ref{lem:discrep_poly}.

Since we were motivated by Monsky's Theorem,
our calculations in Sections~\ref{sec:experiment}
and \ref{sec:upper_bounds} were restricted to the odd case.
The question how small area deviations one can
actually achieve, in terms of explicit constructions,
would also be interesting
for
\emph{even} dissections for which equal areas cannot be achieved.
We leave this for future work.

\subsection*{Acknowledgments}

We are grateful to Moritz Firsching, Arnau Padrol, Francisco Santos, Raman Sanyal, and Louis Theran 
for their input, advice, and many valuable discussions. 

%
%

\bibliographystyle{amsalpha}  
\bibliography{biblio-stripped}  
\end{document}